\documentclass[a4paper,11pt,reqno]{amsart}

\usepackage[a4paper,margin=3cm]{geometry}
\usepackage{microtype} 

\usepackage{mathtools,amssymb,amsmath,amsthm}
\usepackage{thmtools}
\usepackage{xfrac}
\usepackage{stmaryrd,mathrsfs,bm,yfonts}
\usepackage{esint}
\usepackage{enumitem}

\usepackage[dvipsnames]{xcolor}
\usepackage[
  pagebackref,
  colorlinks=true,
  linkcolor=MidnightBlue,
  citecolor=Black,
  urlcolor=MidnightBlue,
  hypertexnames=false
]{hyperref}
\usepackage{doi}
\usepackage[capitalise,nameinlink]{cleveref}
\usepackage{enumitem}

\crefformat{equation}{(#2#1#3)}
\Crefformat{equation}{(#2#1#3)}
\crefrangeformat{equation}{(#3#1#4)--(#5#2#6)}
\Crefrangeformat{equation}{(#3#1#4)--(#5#2#6)}
\crefmultiformat{equation}
  {(#2#1#3)}
  { and~(#2#1#3)}
  {, (#2#1#3)}
  { and~(#2#1#3)}
\Crefmultiformat{equation}
  {(#2#1#3)}
  { and~(#2#1#3)}
  {, (#2#1#3)}
  { and~(#2#1#3)}

\usepackage{tikz,tikz-3dplot}

\usepackage{autonum}

\newtheorem{thm}{Theorem}[section]
\newtheorem*{thm*}{Theorem}
\newtheorem{lemma}[thm]{Lemma}
\newtheorem{proposition}[thm]{Proposition}
\newtheorem{corollary}[thm]{Corollary}

\theoremstyle{definition}
\newtheorem{definition}[thm]{Definition}
\newtheorem{rmk}[thm]{Remark}

\numberwithin{equation}{section}

\newcommand{\R}{\mathbb{R}}
\newcommand{\C}{\mathbb{C}}

\newcommand{\ang}[1]{\langle #1\rangle}

\newcommand{\vol}{\mathrm{Vol}}

\newcommand{\F}{\mathbf{F}}

\newcommand{\hau}{\mathcal{H}}

\newcommand{\eps}{\varepsilon}
\newcommand{\spt}{\mathrm{spt}}

\newcommand{\dist}{\mathrm{dist}}
\newcommand{\tr}{\mathrm{tr}}
\newcommand{\diam}{\mathrm{diam}}
\newcommand{\B}{\mathbf{B}}

\newcommand{\Lip}{\mathrm{Lip}}

\renewcommand{\epsilon}{\varepsilon}

\def\XXint#1#2#3{{\setbox0=\hbox{$#1{#2#3}{\int}$ }
  \vcenter{\hbox{$#2#3$ }}\kern-.6\wd0}}

\renewcommand{\div}{\textrm{div}}

\newcommand{\de}{\partial}

\makeatletter
\renewcommand{\tocsection}[3]{%
  \indentlabel{\@ifnotempty{#2}{\bfseries\ignorespaces\makebox[\@ifempty{#1}{25pt}{75pt}][l]{#1 #2\quad}}}\bfseries#3}
\renewcommand{\tocsubsection}[3]{%
  \indentlabel{\@ifnotempty{#2}{\ignorespaces\makebox[30pt][l]{#1 #2\quad}}}#3}
\newcommand\@dotsep{4}
\def\@tocline#1#2#3#4#5#6#7{\relax
  \ifnum #1>\c@tocdepth \else
    \par \addpenalty\@secpenalty\addvspace{#2}%
    \begingroup \hyphenpenalty\@M
    \@ifempty{#4}{\@tempdima\csname r@tocindent\number#1\endcsname\relax}{\@tempdima#4\relax}%
    \parindent\z@ \leftskip#3\relax \advance\leftskip\@tempdima\relax
    \rightskip\@pnumwidth plus1em \parfillskip-\@pnumwidth
    #5\leavevmode\hskip-\@tempdima{#6}\nobreak
    \leaders\hbox{$\m@th\mkern \@dotsep mu\hbox{.}\mkern \@dotsep mu$}\hfill
    \nobreak
    \hbox to\@pnumwidth{\@tocpagenum{\ifnum#1=1\bfseries\fi#7}}\par
    \nobreak
    \endgroup
  \fi}
\AtBeginDocument{%
\expandafter\renewcommand\csname r@tocindent0\endcsname{0pt}}
\def\l@section{\@tocline{1}{0pt}{1pc}{25pt}{}}
\def\l@subsection{\@tocline{2}{0pt}{2pc}{30pt}{}}
\makeatother

\begin{document}

\title[Anisotropic min-max with genus bounds]{Min-Max Construction of Anisotropic Minimal Surfaces with Genus Bound}

\author[A. De Rosa]{Antonio De Rosa}
\address{Department of Decision Sciences and BIDSA, Bocconi University, Milano, Italy}
\email{antonio.derosa@unibocconi.it}
\author[A. Halavati]{Aria Halavati}
\address{Department of Decision Sciences and BIDSA, Bocconi University, Milano, Italy}
\email{aria.halavati@unibocconi.it}
\author[L. Wang]{Ling Wang}
\address{Department of Decision Sciences and BIDSA, Bocconi University, Milano, Italy}
\email{ling.wang@unibocconi.it}

\begin{abstract}
We establish an anisotropic analogue of the celebrated theorem of Meeks--Simon--Yau: every minimizing sequence of surfaces within a fixed isotopy class converges to a smooth stable anisotropic minimal surface, with genus lower semicontinuity. This result also strengthens White's foundational existence theory for anisotropic minimal disks.
As an application, we develop an anisotropic Simon--Smith min--max theory. In every closed $3$-manifold, we construct anisotropic min--max sequences within fixed isotopy classes whose limits are stable anisotropic minimal surfaces that are smooth except possibly at a single point. If the integrand satisfies either an ellipticity bound or a $C^3$-pinching condition, we remove the singular point by proving two independent removable singularity theorems for anisotropic minimal surfaces that are smooth and stable away from finitely many points.
These removable singularity results also allow to remove the singularities arising in the anisotropic Almgren--Pitts min--max construction in $3$-manifolds of De Philippis--De Rosa and in its multiparameter variants.
\end{abstract}

\maketitle
\tableofcontents

\section{Introduction}

The study of the area functional and its critical points is one of the oldest themes in geometric analysis. Given a Riemannian manifold $N^n$ and an integer $1 \leq k \leq n-1$, one seeks $k$-dimensional surfaces that are critical points of the area functional. 
A classical approach to this question is through minimization. This became possible through the foundational work of Caccioppoli and De Giorgi \cite{DeGiorgi} on sets of finite perimeter, and the Federer--Fleming theory \cite{FedererFleming1960} of currents, which provided the necessary compactness and lower semicontinuity tools. However, direct minimization produces non-trivial minimal surfaces only if the topology of the ambient manifold is rich enough. 

A fundamentally different approach, allowing for the construction of non-minimizing critical points, was introduced by Almgren \cite{Almgren1965Varifolds,Almgren62} via min--max methods and the theory of varifolds. In codimension one, the related regularity theory was developed by Pitts \cite{Pitts}, who proved that every closed $n$-dimensional manifold with $n \leq 6$ contains a non-trivial closed minimal hypersurface. Schoen--Simon \cite{SS} later extended this framework to all dimensions, establishing the existence of minimal hypersurfaces with optimal regularity, namely with singular sets of dimension at most $n-8$. However, controlling the topology of the constructed minimal hypersurfaces seems out of reach in this framework.

In order to obtain a genus control, a natural refinement is to seek critical points within a fixed \emph{isotopy class}. This introduces significant difficulties in the regularity theory, as the class of admissible competitors is significantly smaller. Also, there are examples of sequences of surfaces along which the genus is not lower semicontinuous, hence the need of a fine control of the limit.
A major breakthrough in this direction was achieved in $3$-dimensional ambient manifolds by Meeks--Simon--Yau \cite{MSY}, who showed that minimizing sequences within a fixed isotopy class converge in the sense of varifolds to smooth limits, and enjoy lower semicontinuity of genus. Their approach relies on the regularity theory developed by Almgren--Simon \cite{AS}, and has had important topological applications.

Building upon the regularity framework of \cite{MSY,AS}, Simon--Smith \cite{smith} outlined a min--max theory for surfaces in $3$-manifolds constrained to isotopy classes. This program was completed by Colding--De Lellis \cite{ColdingDeLellisMinMax} and De Lellis--Pellandini \cite{Delellis-Pellandini}, and later refined by Ketover \cite{K}, who established sharp genus bounds for Heegaard sweepouts.

\medskip

The aim of this paper is to build the anisotropic counterpart of the Simon--Smith min-max theory \cite{ColdingDeLellisMinMax,Delellis-Pellandini,K,smith}, by developing an anisotropic analogue of the existence and regularity theorem of Meeks--Simon--Yau \cite{MSY}. More precisely, we fix a closed ambient $3$-manifold $N$ endowed with an anisotropic surface energy of the form
\begin{align}
    \F(\Sigma) = \int_{\Sigma} F(x,\nu_\Sigma)\,\mathrm{d}\hau^2,
\end{align}
where $F$ is an even elliptic integrand defined on the unit tangent bundle of $N$, and $\nu_\Sigma$ denotes the unit normal to $\Sigma$. Our goal is to construct closed smooth anisotropic minimal surfaces with controlled topology.

Extending the existence and regularity theory of minimal surfaces to critical points of such anisotropic energies has been a central topic in geometric measure theory since their introduction by Almgren \cite{Almgren68}. We refer to the survey \cite{DeR-survey} for an overview on the anisotropic minimal surface theory.  
For minimizers, the theory in codimension one largely parallels that of the area functional, with the main difference that the optimal regularity of anisotropic energy minimizers is a singular set of zero $\hau^{n-3}$-measure, as proved by Almgren--Schoen--Simon \cite{SSA}. This is almost sharp in view of the singular minimizing cone in $\R^4$ constructed by Morgan \cite{Morgan1990}, see also the recent work by Mooney--Yang \cite{MooneyYang2024}. In the last decade there has been a large body of work aimed at expanding the existence and regularity theory of anisotropic minimal surfaces \cite{CL,DRK,DKS,DRR,DRT,SantilliKolasinski2025}. 

Recently, De Philippis--De Rosa \cite{DePhilippisDeRosa} established the min-max construction of anisotropic minimal surfaces that are smooth away from at most one point, via the Almgren--Pitts scheme. In subsequent work with Li \cite{DDL}, they could remove the singular point and they extended the construction to every dimension, yielding anisotropic minimal closed hypersurfaces smooth outside an $\hau^{n-3}$-negligible set. However, as their isotropic counterparts \cite{Pitts,SS}, these results do not provide a control on the topology of the resulting surfaces. Moreover, their strategy for removing the singular point relies on a delicate construction of nested sweep-outs. Such a construction does not appear to extend naturally to multiparameter min--max schemes, thereby obstructing the study of higher-index minimal surfaces. It also presents a significant obstacle to establishing an anisotropic analogue of the multiplicity one theorem proved by Zhou \cite{Zhou2020} and to the related program of proving Yau's conjecture \cite{YC} for generic metrics envisioned by Marques--Neves in their work on the Morse Index Conjecture \cite{MarquesNeves2016MorseIndexMultiplicity,MarquesNeves2018MorseIndex}. We remark that De Rosa--Pigati have also recently initiated a program to construct anisotropic min-max minimal hypersurfaces via a suitable anisotropic Allen-Cahn approximation \cite{DeRosaPigati2025}, constructing (possibly singular) anisotropic min-max hypersurfaces, yielding an anisotropic analogue of the classical result of Hutchinson--Tonegawa \cite{HutchinsonTonegawa2000}.

Our goal is to develop the variational theory for anisotropic minimal surfaces in fixed isotopy classes in ambient $3$-manifolds. Progress in this direction has been limited. To the best of our knowledge, the only results in this direction are the existence of at least one minimizing disk with prescribed boundary in a convex set proved by White \cite{W-1}, and the compactness for smooth anisotropic stationary surfaces under genus bounds \cite{White}.

A major obstacle to extending the classical isotropic theory is the lack of a monotonicity formula in the anisotropic setting. As shown by Allard \cite{Allard}, monotonicity essentially characterizes the area functional. Without it, key tools break down: tangent varifolds are not necessarily cones, density bounds are generally unavailable, and the structure theory used in \cite{MSY}, which relies heavily on the analysis of tangent cones and their spherical links as geodesic nets \cite{Allard-geodesic}, no longer applies.

\medskip

In this paper, we overcome these obstructions by adopting a completely different approach. Specifically, we combine the construction of White \cite{W-1} with the strategy of Hass and Scott \cite{HS} to show that limit varifolds of minimizing sequences within a fixed isotopy class admit replacements in annuli, in the sense of Pitts \cite{Pitts}. By subsequently implementing Pitts’ regularity scheme, we establish the full anisotropic counterparts of the results of Meeks–Simon–Yau \cite{MSY} (\cref{thm:Meeks-Simon-Yau-without-boundary}) and Almgren–Simon \cite{AS} (\cref{thm:white-improvement}).
We also note that White \cite{W-1} proved that, for a fixed boundary on the 2-sphere, there exists at least one minimizing sequence of smooth disks converging to a smooth disk. Using our regularity scheme, we strengthen this result by showing that the same conclusion holds for the limit of any minimizing sequence of disks: namely, every such sequence converges to a smooth disk.
We summarize the three contributions outlined above in the following informal statement, and refer the reader to \Cref{sec:mainresults} for the precise formulations.

\begin{thm*}[\cref{thm:white-improvement}, \cref{thm:Meeks-Simon-Yau-without-boundary} and \cref{thm:Meeks-Simon-Yau-with-boundary}]
    In any Riemannian $3$-manifold endowed with an even elliptic integrand $F$, every $\mathbf{F}$-minimizing sequence of surfaces $\{\Sigma_i\}$ in a fixed isotopy class (with or without prescribed boundary) converges, in the varifold sense, to a varifold induced by a smooth embedded surface $\Sigma$ (possibly with multiplicity). Furthermore, the genus is lower semicontinuous under this convergence.
\end{thm*}

Using the theorem above, we construct replacements for varifolds arising from the min–max procedure within a fixed isotopy class. A central difficulty in the anisotropic min–max regularity theory is to ensure that consecutive replacements glue smoothly. This issue was addressed by De Philippis–De Rosa \cite{DePhilippisDeRosa} and De Philippis–De Rosa–Li \cite{DDL} through the construction of anisotropic constant (non-zero) mean curvature surfaces, which allows one to focus on the smooth gluing at points of multiplicity one only.
However, this strategy is not applicable in our setting, as minimizers of the prescribed mean curvature functional over isotopy classes enjoy only $C^{1,1}$ regularity, as shown by Sarnataro--Stryker \cite{SarnataroStryker2025OptimalRegularity}. We overcome this difficulty by proving smooth gluing at points of arbitrary multiplicity via an induction argument on the multiplicity. As a consequence, we prove in \cref{thm:aniso-simon-smith-one-singularity} that there exists a min--max sequence whose associated varifolds converge to a smooth limit away from at most one point.
We state below an informal version of this result and refer the reader to \cref{thm:aniso-simon-smith-one-singularity} for the precise statement.

\begin{thm*}[\cref{thm:aniso-simon-smith-one-singularity}]
    In any Riemannian $3$-manifold endowed with an even elliptic integrand $F$, for any saturated set of generalized families of surfaces $\Lambda$, there is a min–max sequence obtained from $\Lambda$ and converging in the sense of varifolds to a smooth embedded anisotropic minimal surface, away from at most one point.
\end{thm*}

In the area case, the removability of isolated singularities follows from the monotonicity formula together with a standard capacity argument based on stability. In contrast, for surfaces that are stationary with respect to an anisotropic integrand, an isolated singularity may carry infinite multiplicity and have positive $2$-capacity, so that this classical approach no longer applies. What can essentially go wrong is the accumulation of sheets as we zoom in around a singular point, which will blow up the density estimate. 

We develop two independent approaches to remove the last singularity arising in \cref{thm:aniso-simon-smith-one-singularity}, corresponding to two different assumptions on the integrand $F$. In the first method, we follow the lamination approach developed in \cite{MPR-removable-singularity,FcS}. The starting point is a Bernstein-type theorem for complete $\F$-stationary and $\F$-stable surface in $\R^3\setminus\{p\}$, under the ellipticity bound
\begin{equation}\label{assumption1}
\frac{\max_{\nu}\lambda_{\max} \Psi_F(p,\nu)}
{\min_{\nu}\lambda_{\min} \Psi_F(p,\nu)}<8.
\end{equation}
Here \(\Psi_F\) is the ellipticity tensor defined in
\eqref{eq:def-Psi_F}.
Under \eqref{assumption1}, the Bernstein-type conclusion is given in
\cref{thm:anisotropic-bernstein}.  The proof combines the stability inequality with conformal methods to show that such a surface has vanishing Gaussian curvature.
Given an isolated singularity, we rescale the surface around the singular point and prove via curvature estimates that any sequence of blow-ups subconverges to an $\F$-minimal lamination of $\R^3\setminus\{0\}$. By the Bernstein theorem, every leaf of the limiting lamination must be flat, hence we deduce quantitative curvature decay near the singularity. The asymptotic flatness of all blow-up laminations is then used to control the topology of small level sets of the distance function and to obtain a uniform density bound at the singular point. Once finite density is established, a standard capacity argument extends the stability inequality across the singularity, implying that the singularity is removable.

The second removability criterion we introduce is based on a separation argument, inspired by Bamler--Kleiner \cite{BK}. We introduce a separation function, which measures the distance between nearby graphical sheets of the surface on small annuli around the singularity. Stability gives curvature control at the scale of the distance to the origin, so locally the surface decomposes into graphical sheets. If the density were large, many such sheets would have to accumulate. When two sheets are sufficiently close, their separation satisfies an approximate anisotropic Jacobi equation. The key assumption is that the extrinsic distance $|x|$ is a sub-solution for the same anisotropic Jacobi operator $\mathcal{J}_F$, in the sense that
\begin{equation}\label{assumption2}
\mathcal J_F |x| \geq -C.
\end{equation}
Comparing the separation with $|x|$, and applying a De Giorgi-type iteration, we prove that the sublevel sets of the separation decay super-exponentially. This strong decay rules out excessive accumulation of sheets and implies that the density can blow up at most logarithmically near the singularity. A log-log cutoff then shows that the singularity has zero capacity, so stability extends across it and  yields removability.
Assumption \eqref{assumption2} is satisfied, for instance, under a $C^3$-pinching condition on $F$ (see \cref{def:C3-pinching}).

In fact, we establish a more general result, independent of the min–max scheme:

\begin{thm*}[\cref{thm:removable-singularity}]
Assume $F$ satisfies either \eqref{assumption1} or \eqref{assumption2}.
    Let $\Sigma$ be an $\F$-stationary surface which is smooth and $\F$-stable away from a point $p\in \overline{\Sigma}\setminus\Sigma$. Then $p$ is a removable singularity.
\end{thm*}

As a direct corollary, we are able to remove the singularities arising in our min–max construction and thereby obtain one of our main results, \cref{thm:aniso-simon-smith-ketover}. Moreover, we can use \cref{thm:removable-singularity} to show that multiparameter min–max procedures also produce smooth minimal surfaces (see \cref{rmk:aniso-almgren-pitts}). In particular, our approach shows that the anisotropic minimal surfaces constructed in $3$-manifolds by De Philippis--De Rosa in \cite{DePhilippisDeRosa} are smooth as well, when $F$ satisfies either \eqref{assumption1} or \eqref{assumption2}.
While the singular points in \cite{DePhilippisDeRosa} were already removed in \cite{DDL}, the argument therein relies on a delicate construction of nested sweepouts, which does not  extend to multiparameter min–max schemes. Our results therefore improve the current state of the art in the anisotropic min–max theory \`a la Almgren–Pitts \cite{DePhilippisDeRosa,DDL}, allowing for the construction of higher-index minimal surfaces. This also opens the way to studying the anisotropic analogue of the proof of the multiplicity one conjecture in \cite{Zhou2020}, as well as its applications to Yau’s conjecture \cite{YC} for generic metrics via the Morse index conjecture \cite{MarquesNeves2016MorseIndexMultiplicity,MarquesNeves2018MorseIndex}, when $F$ satisfies either \eqref{assumption1} or \eqref{assumption2}.

\newtheorem*{ack*}{Acknowledgement}
\begin{ack*}
    The authors were funded by the European Union: the European Research Council (ERC), through StG ``ANGEVA'', project number: 101076411. Views and opinions expressed are however those of the authors only and do not necessarily reflect those of the European Union or the European Research Council. Neither the European Union nor the granting authority can be held responsible for them.
\end{ack*}

\subsection{Organization and sketch of proofs}
Here we outline the structure of the paper.

In \cref{sec:setup}, we set up the problem and define the notation. For readers who wish to use the results of our paper as a black box, we refer them to \cref{sec:mainresults}, where the main results are stated.

In \cref{sec:aniso-MSY}, we prove the anisotropic counterpart of Meeks-Simon-Yau's celebrated compactness theorem \cite[Theorem 1]{MSY}, and we also improve White's existence result in \cite[Theorem 3.4]{W-1}. The main strategy can be sketched as follows:
\begin{enumerate}
    \item Given a sequence of isotopic $\F$-minimizing surfaces, we use \cref{prop:replacement-thm-AS-aniso} and finitely many neck-pinch surgeries to reduce the problem to proving regularity of a sequence of $\F$-minimizing stacked disks. Roughly speaking, we only need to show that the limit of {\emph{any}} sequence of $\F$-minimizing disks is smooth, which in turn is an improvement of the results of White \cite{W-1} on the existence of \emph{one} sequence of $\F$-minimizing disks with this property.

    \item Unlike Almgren--Simon \cite{AS}, we cannot study blow-up limits. However, in \cref{prop:replacements-in-annuli}, we show that the varifold limit of $\F$-minimizing disks admits replacements in all annuli smaller than some uniform constant. To construct these replacements, we cover the annulus with a Whitney covering, with respect to the inner boundary, by balls with convex boundary. On each such ball, we replace our sequence with an $\F$-minimizing disk, whose existence was proved in \cite{W-1}. We then use the gluing strategy of Hass--Scott \cite{HS} to show that, on the overlap regions, these replacements glue smoothly together in the limit. We also show, by a simple estimate in \cref{lemma:tentacles-estimate}, that such replacements can be performed without affecting the isotopy or the mass of the limit varifold in the inner ball.

    \item The smooth gluing here is simpler than in the min-max case and follows from the same strategy of Hass--Scott, since the minimizing radius does not degenerate in the limit. Since this radius is uniform around all points, we can cover the entire limit with annuli and conclude smoothness.
\end{enumerate}

In \cref{sec:min-max}, we construct min-max sequences and establish the corresponding regularity theory. Most of the Simon--Smith min-max theory follows identically as in \cite{ColdingDeLellisMinMax}. The main point of divergence is the smooth gluing of consecutive replacements. Unlike De Philippis--De Rosa--Li \cite{DePhilippisDeRosa,DDL}, who use CMC surfaces to ensure that in most points the replacements will have multiplicity one, in the case of isotopies the best regularity one can hope for CMC surfaces is $C^{1,1}$ (see the example by Sarnataro--Stryker \cite{SarnataroStryker2025OptimalRegularity}).  Hence we deal with sheets of high multiplicity directly using an inductive gluing argument, which we roughly sketch below:
\begin{enumerate}
    \item Since the analysis of blow-ups is not feasible here, we use an induction on the multiplicity of the points at the interface of two consecutive replacements. The case of multiplicity one follows \cite{DDL}. Now consider a sheet of multiplicity $k\geq 2$. Using a strategy similar to \cite{DDL}, namely by taking further replacements in thin annuli bounded by spherical caps meeting at a small angle, we ensure that if by contradiction the second replacement splits into at least two sheets, then their multiplicities are strictly less than $k$. We then center an annulus around the split and perform another replacement. By induction, the sheets with multiplicity less than $k$ glue smoothly. Then, using the maximum principle and unique continuation, we can ensure that the sheets of the second replacement continue smoothly into the interior of the third replacement, meeting at the spine. Combining this with the regularity of replacements yields a contradiction.

    \item As in the results of \cite{DePhilippisDeRosa}, we obtain a limit surface that is smooth away from at most one singular point. In the general multiparameter case, this would be replaced by finitely many points. If this singular point is removable and the limit is smooth, we prove the optimal genus upper bound as in \cite{K}. In the next sections, we show how to remove this singularity under additional hypotheses.
\end{enumerate}

In \cref{sec:removable-singularity-bernstein}, we show that if the ellipticity of $F$ is strictly smaller than $8$, i.e. if \eqref{assumption1} holds, then isolated singularities are removable. This holds in general for $\F$-minimal surfaces that are smooth and $\F$-stable away from finitely many isolated points. We sketch the proof as follows:
\begin{enumerate}
    \item From \cite[Theorem 2.8]{DePhilippisDeRosa}, treating the singularity as a boundary point, we know that stability implies $|A_\Sigma|(x) \leq C\dist(x,p)^{-1}$. We use this bound to show that we can take blow-up limits around $p$ in the sense of laminations (as in \cite{CM}), with each leaf inheriting $\F$-minimality and $\F$-stability from the sequence.

    \item We then use the stability operator and the ellipticity bound to show, using the main result of \cite{C}, that the conformal universal cover of each leaf is either $\C$ or $\C\setminus\{0\}$.

    \item In both cases, using the result of \cite{FcS}, we show that any blow-up lamination limit is in fact flat, thereby improving the curvature estimate to $|A| = o(\dist(x,p)^{-1})$.

    \item The rest of the proof is not exactly as we state here, but intuitively the argument is as follows. We consider the ambient distance $\dist(x,p)^2$ restricted to $\Sigma$ as a Morse function and show that the improved curvature bound implies that there are no saddle critical points at sufficiently small radii. This implies bounded topology, which combined with Gauss--Bonnet, further implies an $L^2$ bound on the second fundamental form. This provides a uniform upper density bound, from which we conclude by a standard capacity argument.
\end{enumerate}

In \cref{sec:remove-singularity-separation} we prove another result of removal of singularity, which without loss of generality we assume to be at $0$, under the assumption that $\mathcal{J}_\F |x| \geq -C$. Here $\mathcal{J}_F$ denotes the anisotropic Jacobi operator and $\Psi_F$ is the ellipticity tensor defined in \eqref{eq:def-Psi_F}. We sketch the proof below:
\begin{enumerate}
    \item First, we show that if $F$ satisfies a $C^3$-pinching condition (roughly $|\nabla_{\mathbb{S}^2} \Psi_F| \leq 2\lambda_{\min}(\Psi_F)$ pointwise on the sphere), then the assumption is satisfied. For the precise statement, see \cref{def:C3-pinching}.

    \item As in \cref{sec:removable-singularity-bernstein}, we know that $|A| \lesssim |x|^{-1}$. If the origin has infinite density ratio, then, in small annuli, one sees more and more layers accumulating on one another. We then define a \textit{separation} function $\mathbf{s}$, which measures the distance between two nearby sheets, and show that it is almost a Jacobi field for $\mathcal{J}_F$.

    \item We then show in \cref{prop:log-separation-PDE-Lp} that $\log\left(\frac{|x|}{\mathbf{s}}\right)$ satisfies a weighted differential inequality and belongs to $L^p$ for all $1\leq p < \infty$.

    \item Using a general Hardy-type weighted inequality on anisotropic minimal surfaces in \cref{lem:hardy}, together with a Moser--De Giorgi-type iteration, we show that the sublevel set of the separation, $\{\mathbf{s}(x) \leq e^{-k}|x|\}$, decays double exponentially in $k$, see \cref{prop:superexponential-decay}.

    \item Finally, we use the simple but crucial \cref{lem:key-lemma} to show that this decay implies that the density blows up at most logarithmically. This is enough, using a log-log cutoff, to conclude that the origin has zero capacity, and hence is a removable singularity.
\end{enumerate}

\section{Preliminaries and main results}
\subsection{Setup}\label{sec:setup}

Let \(N\) be a complete Riemannian \(3\)-manifold satisfying the
homogeneous regularity assumptions of Meeks--Simon--Yau \cite{MSY}. More precisely,
there exist constants \(\rho_0>0\) and \(\mu<\infty\), independent of the
base point, such that for every \(x_0\in N\) the exponential map gives a
diffeomorphism from the Euclidean ball \(B_{\rho_0}(0)\subset T_{x_0}N\)
onto the geodesic ball \(B_{\rho_0}(x_0)\subset N\), and the metric
coefficients in these normal coordinates satisfy the uniform bounds
specified in \cite[(1.1)-(1.2)]{MSY}. In particular, all local estimates below are
understood to hold uniformly for every center \(x_0\in N\) and every
radius \(r\le \rho_0\). We remark that these assumptions are automatic
when \(N\) is compact; in the non-compact case they are equivalent to
completeness, a positive lower bound for the injectivity radius, and a
uniform bound for the sectional curvatures.

We denote by \(B_r(x)\) the geodesic ball of radius \(r\) centered at
\(x\). We write \(G_2(N)\) for the bundle of unoriented \(2\)-planes in
\(TN\), and \(G_2(\Omega)\) for its restriction to an open set
\(\Omega\subset N\). We refer to \cite{SimonLN} for the standard
definitions and notation concerning varifolds.

Let
\[
    F:G_2(N)\to (0,+\infty)
\]
be a positive anisotropic integrand. For a \(2\)-varifold
\(V\in\mathbf{V}_2(N)\)  and an open set \(\Omega\subset N\), we define
\[
    \F(V,\Omega)
    :=
    \int_{G_2(\Omega)} F(x,T)\,\mathrm{d}V(x,T).
\]
When \(\Omega=N\), we simply write \(\F(V)\). If
\(\Sigma\subset N\) is \(\mathcal H^2\)-rectifiable, we also
write
\[
    \mathbf F(\Sigma,\Omega)
    :=
    \int_{\Sigma\cap\Omega}
        F(x,T_x\Sigma)\,\mathrm{d}\mathcal H^2,
\]
and \(\mathbf F(\Sigma):=\mathbf F(\Sigma,N)\).

Since we are in codimension one, it is convenient to associate to \(F\)
the even one-homogeneous function
\[
    G:TN\to [0,+\infty)
\]
defined by
\[
    G(x,0):=0,
    \qquad
    G(x,\lambda\nu)
    :=
    |\lambda|\,F(x,\nu^\perp),
    \qquad
    \lambda\in\mathbb R,\quad \nu\in S_xN.
\]
Thus
\[
    G(x,-\nu)=G(x,\nu),
    \qquad
    G(x,\lambda\nu)=|\lambda|G(x,\nu),
\]
and, whenever \(\Sigma\) is two-sided with unit normal \(\nu\),
\[
    \mathbf F(\Sigma,\Omega)
    =
    \int_{\Sigma\cap\Omega} G(x,\nu)\,\mathrm{d}\mathcal H^2.
\]
The right-hand side is independent of the choice of the unit normal,
because \(G\) is even. For one-sided surfaces the same formula is
understood locally, or equivalently through the associated unoriented
tangent-plane bundle.

We shall use the following notation for the ellipticity tensor associated
with \(F\). For \((x,\nu)\in SN\), define
\[
    \Psi_F(x,\nu):\nu^\perp\to \nu^\perp
\]
by
\begin{equation}\label{eq:def-Psi_F}
    \big\langle \Psi_F(x,\nu)v,w\big\rangle
    =
    D^2_{\nu\nu}G(x,\nu)[v,w],
    \qquad v,w\in \nu^\perp .
\end{equation}
Equivalently, \(\Psi_F(x,\nu)\) is the restriction of the vertical Hessian
of the one-homogeneous extension \(G\) to the tangent space
\(\nu^\perp=T_\nu S_xN\).

Throughout the paper we assume the following uniform ellipticity and
regularity hypotheses. After decreasing \(\lambda\in(0,1]\) if necessary,
all constants below are encoded in the same parameter \(\lambda\).

\begin{enumerate}[label=(H\arabic*), ref=H\arabic*]
    \item \label{H:comparability}
    \emph{Uniform comparability.}
    For every \((x,T)\in G_2(N)\),
    \[
        \lambda \le F(x,T)\le \lambda^{-1}.
    \]
    Consequently, for every \(\mathcal H^2\)-rectifiable set
    \(\Sigma\subset N\) and every open set \(\Omega\subset N\),
    \[
        \lambda\,\mathcal H^2(\Sigma\cap\Omega)
        \le
        \mathbf F(\Sigma,\Omega)
        \le
        \lambda^{-1}\mathcal H^2(\Sigma\cap\Omega).
    \]

    \item \label{H:regularity}
    \emph{Uniform regularity bounds.}
    The function \(G\) is \(C^3\) on \(TN\setminus\{0\}\), and satisfies,
    for all \((x,\nu)\in SN\),
    \[
        \sum_{1\le i+j\le 3}
        |D_x^iD_\nu^jG(x,\nu)|
        \le
        \lambda^{-1}.
    \]

    \item \label{H:ellipticity}
    \emph{Uniform ellipticity / uniform convexity.} For every \((x,\nu)\in SN\) and every \(w\in \nu^\perp\), \begin{equation}
        \langle \Psi_F(x,\nu)w,w\rangle\ge\lambda |w|^2.\label{uniform-convexity-assumption}
    \end{equation}
\end{enumerate}

The one-homogeneity of \(G\) implies Euler's identity
\begin{equation}
    \ang{D_\nu G(x,\nu),\nu}=G(x,\nu),
    \qquad (x,\nu)\in SN.\label{integrand:homogen}
\end{equation}
In particular, by \eqref{H:comparability} and \eqref{integrand:homogen}, we deduce
\[
    |D_\nu G(x,\nu)|\ge \lambda,
    \qquad (x,\nu)\in SN.
\]
Moreover, \eqref{H:ellipticity} implies the strict support inequality
\begin{equation}
    G(x,\nu)>
    |\ang{D_\nu G(x,\bar\nu),\nu}|\label{integrand:convexity}
\end{equation}
whenever \((x,\nu),(x,\bar\nu)\in SN\) and
\(\nu\neq \pm\bar\nu\).

Let \(g\in C_c^1(N;TN)\) be a compactly supported vector field and let
\(\{\varphi_t\}_{|t|<\varepsilon}\) be its flow. The anisotropic first
variation of a varifold
\(V\in\mathbf{V}_2(N)\) evaluated at $g$ is defined by
\[
    \delta_{\F}V(g)
    :=
    \left.\frac{\mathrm{d}}{\mathrm{d}t}\right|_{t=0}
    \F((\varphi_t)_\#V).
\]
In local coordinates, if \(\operatorname{spt}g\subset\Omega\), this is
given by
\begin{equation}
    \delta_{\F}V(g)
    =
    \int_{G_2(\Omega)}
    \left[
        \ang{\mathrm{d}_xF(x,T),g(x)}
        +
        B_F(x,T):\nabla g(x)
    \right]\,\mathrm{d}V(x,T).\label{eq:1st-variation}
\end{equation}
Here \(\mathrm{d}_xF\) denotes the horizontal derivative of \(F\) in the base
variable, \(\nabla g\) is the covariant derivative of \(g\), and \(B_F(x,T)\in T_xN\otimes T_xN\) is the fiber derivative of \(F\) with
respect to the plane variable, which is explicitly computed in \cite[Lemma A.2]{DDG}. In normal coordinates one may equivalently
write \(B_F(x,T):Dg(x)\). We say that \(V\) is \textit{\(\F\)-stationary} in \(\Omega\) if $\delta_{\F}V(g)=0$
for every \(g\in C_c^1(\Omega;TN)\). 

If \(\Sigma\subset N\) is a smooth two-sided surface with unit normal
\(\nu\), we define its scalar anisotropic mean curvature \(H_F\),
with respect to \(\nu\), by
\[
    \delta_{\F}\mathbf{v}(\Sigma)(g)
    =
    \int_{\Sigma}
        H_F\,\ang{g,\nu}\,\mathrm{d}\mathcal H^2
\]
for every compactly supported vector field \(g\) whose support is
disjoint from \(\partial\Sigma\), whenever \(\partial\Sigma\neq\emptyset\). If \(\Sigma\subset N\) is a smooth
embedded surface, we say that \(\Sigma\) is \textit{\(\F\)-minimal} in \(\Omega\)
if its associated varifold is \(\F\)-stationary in
\(\Omega\).
In particular, if \(\Sigma\) is two-sided, then \(\Sigma\) is
\(\F\)-minimal in \(\Omega\) if and only if \(H_F=0\) on $\Sigma\cap\Omega$. 
We say that \(\Sigma\) is \textit{\(\F\)-stable} in \(\Omega\) if $\delta_{\F}^2\Sigma(g)\ge 0$
for every \(g\in C_c^1(\Omega;TN)\), where $\delta_{\F}^2\Sigma(g)$ is the anisotropic second variation along \(g\in C_c^1(\Omega;TN)\)
defined by
\[
    \delta_{\F}^2\Sigma(g)
    :=
    \left.
    \frac{\mathrm{d}^2}{\mathrm{d}t^2}
    \right|_{t=0}
    \mathbf F(\varphi_t(\Sigma),\Omega).
\]

Let \(U\subset N\) be an open set with \(C^2\) boundary, and let \(\eta\)
denote the outward unit normal to \(\partial U\). We say that \(U\) is
\textit{\(\F\)-convex}, or equivalently that \(\partial U\) is
\(\F\)-mean convex, if $H_F\ge 0$ on $\partial U$,
where \(H_F\) is computed with respect to \(\eta\). We say that \(U\) is
\textit{strictly \(\F\)-convex} if $H_F>0$ on $\partial U$,
again with respect to the outward normal.

We also use the standard
Simon--Smith min-max terminology for generalized families of surfaces. A
\textit{generalized family} is a one-parameter family
\(\{\Sigma_t\}_{t\in[0,1]}\) of closed subsets of \(N\), consisting of
smooth embedded closed surfaces except possibly at finitely many
parameters and finitely many points, with continuous \(\F\)-energy. A
collection \(\Lambda\) of generalized families is called
\textit{saturated} if it is closed under smooth ambient isotopies in
\(\mathrm{Diff}_0(N)\). We always assume the standard bounded-singularity
condition for the saturated families considered below.
For such a saturated family \(\Lambda\), its \textit{anisotropic width}
is
\[
    m_0(\Lambda)
    :=
    \inf_{\{\Sigma_t\}\in\Lambda}
    \max_{t\in[0,1]}\F(\Sigma_t).
\]
A sequence
\(\{\Sigma_t^j\}\subset\Lambda\) is called \textit{minimizing} if
\[
    \lim_{j\to\infty}\max_{t\in[0,1]}\F(\Sigma_t^j)
    =
    m_0(\Lambda),
\]
and a sequence \(\Sigma^j:=\Sigma_{t_j}^j\) is called a
\textit{min-max sequence} if \(\{\Sigma_t^j\}\) is minimizing and
\[
    \lim_{j\to\infty}\F(\Sigma_{t_j}^j)
    =
    m_0(\Lambda).
\]
The precise definitions are recalled in
\cref{sec:generalized-family-width}.

\subsection{Notation}
We record some notation that we will use throughout the paper:
\begin{itemize}
    \item $\mathbf{B}$ denotes the closed unit $3$-ball with center $0$ in $\mathbb{R}^3$.
    \item $\mathbf{D}$ denotes the closed unit disk with center $0$ in $\mathbb{R}^2$.
    \item $\mathcal{M}$ denotes the space of all surfaces with boundary in $N$ (or in $\mathbb{R}^3$, depending on the context) which are $C^2$ diffeomorphic to the standard disk $\mathbf{D}$.
    \item $\mathcal{C}$ denotes the collection of all connected compact (not necessarily orientable) smooth $2$-dimensional surfaces without boundary embedded in $N$.
    \item $\mathcal{C}_1$ denotes the collection of compact embedded surfaces $\Sigma$ such that each connected component of $\Sigma$ is an element of $\mathcal{C}$.
    \item Given \(\Sigma\in\mathcal{C}_1\), we let \(\mathcal{G}(\Sigma)\) denote the ambient isotopy class of \(\Sigma\). If \(N\) is non-compact, the ambient isotopies are required to be compactly supported. If \(\Sigma\) has boundary, \(\mathcal{G}(\Sigma)\) denotes the corresponding ambient isotopy class with boundary fixed.
    
    \item As in \cite[Lemma 3]{MSY}, we define:
    \[
        d_U(x) = \dist(x,U)\,, \mbox{\quad  and \quad } U(s) = \{x\in N: d_U(x) < s\} \;\;\text{ for each } s>0.
    \]
    \item $\mathbf{v}(\Sigma)$ is the canonical varifold associated to the smooth surface $\Sigma$. 
    \item $\mathbf{V}_2(N)$ denotes the space of all \(2\)-varifolds in
    \(N\), namely the space of (positive) Radon measures on \(G_2(N)\).
\end{itemize}

\subsection{Main results}\label{sec:mainresults}

Throughout the paper, unless otherwise specified, we work under the
baseline regularity assumptions that the Riemannian metric on \(N\) is
of class \(C^4\), and that the
anisotropic integrand \(F\), equivalently its one-homogeneous extension
\(G\), is of class $C^3$. Under these assumptions, the anisotropic
regularity theory produces \(\F\)-minimal surfaces of class \(C^2\).
Higher regularity of the ambient metric and of \(F\) gives the
corresponding higher regularity of the resulting \(\F\)-minimal surfaces
by elliptic bootstrapping. In what follows, unless a stronger regularity
is explicitly required, the word \textit{smooth}, when applied to an
\(\F\)-minimal surface or to a limiting surface produced by the
regularity theory, is understood in this sense. The same convention
applies to regularity up to the boundary, provided the relevant boundary
components and boundary curves have the corresponding regularity.

As explained in the introduction, our first main theorem is a strengthening of \cite[Theorem 3.4]{W-1}, where White proved that there exists at least one $\F$-minimizing sequence of smooth disks converging to an $\F$-minimizing disk. We generalize this result by showing the same conclusion for \emph{every} $\F$-minimizing sequence of smooth disks:
\begin{thm}\label{thm:white-improvement}
    Let $\textbf{B}\subset \R^3$ be the unit ball such that its boundary is $\F$-convex and consider any $\F$-minimizing sequence of smooth disks $M_k\in\mathcal{M}$ with $\de M_k\subset \de \textbf{B}$, that is, for any $P\in\mathcal{M}$ with $\de P = \de M_k$ we have $\F(M_k) \leq \F(P)+\epsilon_k$ with $\epsilon_k\to0$.
    \begin{enumerate}
        \item Then for any compact set $K\subset \textbf{B}$, up to a subsequence (not relabeled) $M_k\to \Sigma$ in $K$ in the sense of varifolds, where $\Sigma$ is a smooth $\F$-minimal surface in $K$.
        \item Moreover if $\de M_k$ converges smoothly to a simple closed loop $\Gamma\subset \de\textbf{B}$, then  $M_k\to\Sigma$ in the sense of varifolds in the whole $\overline{\textbf{B}}$, with $\partial\Sigma=\Gamma$,  and $\Sigma$ is an $\F$-minimizing disk.
        \item If there is a relatively open set $S\subset \de \textbf{B}$ such that $\de M_k \cap S = \Gamma_k$ is a sequence of simple arcs converging smoothly to a limit simple arc $\Gamma$, then the limit surface $\Sigma$ has a component that extends smoothly up to $\Gamma$.
    \end{enumerate}
\end{thm}
We use \cref{thm:white-improvement} as a building block, in combination with strategies reminiscent of \cite{AS}, to prove the anisotropic counterpart of the result of Meeks-Simon-Yau \cite[Theorem 1]{MSY}:
\begin{thm}\label{thm:Meeks-Simon-Yau-without-boundary}
    Take any compact $3$-manifold $N$ and anisotropic integrand $F$ as defined above, let $\Sigma_0\in \mathcal{C}_1$ be a reference surface and let $0 < I = \inf\{\F(\Sigma):\, \Sigma \in \mathcal{G}(\Sigma_0)\}$ be the infimum of the $\F$-area in its isotopy class. Take any minimizing sequence of surfaces $\Sigma_k$, that is $\lim_{k\to\infty} \F(\Sigma_k) = I$. Then, there are positive integers $R,n_1,\cdots,n_R$ and pairwise disjoint smooth $\F$-minimal surfaces $\Sigma^{(1)},\cdots,\Sigma^{(R)}$ such that (up to a subsequence, not relabeled):
    \begin{align}
        \Sigma_k\to n_1\Sigma^{(1)}+n_2\Sigma^{(2)}+\cdots+n_R\Sigma^{(R)}\,,\label{eq:S_k-converges}
    \end{align}
    in the sense of varifolds. Furthermore, if $g_i = \mathrm{genus}(\Sigma^{(i)})$, then the following bound holds:
    \begin{align}
        \sum_{i\in \mathcal{U}} \frac12n_i(g_i-1) + \sum_{i\in \mathcal{O}} n_ig_i \leq \mathrm{genus}(\Sigma_k)\,,\label{eq:genus-bound}
    \end{align}
    for large enough $k$. Here \(\mathcal U\) denotes the set of indices \(i\) for which \(\Sigma^{(i)}\) is one-sided, and \(\mathcal O\) denotes the set of indices \(i\) for which \(\Sigma^{(i)}\) is two-sided.
    Moreover, for every $i \in \mathcal{O}$ we have that $\Sigma^{(i)}$ is $\F$-stable. If  $\Sigma_k$ is two-sided for all $k$, then all $\Sigma^{(1)},\cdots,\Sigma^{(R)}$ are $\F$-stable. 
\end{thm}
We also prove a version of \cref{thm:Meeks-Simon-Yau-without-boundary} for a reference surface $\Sigma_0$ with boundary, which will be used in our construction of the anisotropic counterpart of Simon-Smith min-max theory.
\begin{thm}\label{thm:Meeks-Simon-Yau-with-boundary}
    Let $A\subset B_{\rho_0}\subset  N$ be an open set and take a smooth surface $\Sigma_0\subset A$ with boundary $\Gamma = \de\Sigma_0 \subset \de A$. Moreover let $0<I = \inf\{\F(\Sigma):\, \Sigma \in \mathcal{G}(\Sigma_0)\}$ and take any minimizing sequence $\Sigma_k$ with $\Gamma = \de\Sigma_k$ achieving the infimum $I = \lim_{k\to\infty }\F(\Sigma_k)$. Then the following holds, up to subsequences (not relabeled):
    \begin{enumerate}
        \item There exists a surface $\Sigma$ which is smooth in any compact set $K\subset A$ such that $\Sigma_k\to \Sigma$ in the sense of varifolds in $K$. Moreover $\Sigma$ is $\F$-stationary and $\F$-stable in $K$ (however the topological information might have been lost).
        
        \item If the boundary $\de A$ is $\F$-convex, then $\de\Sigma = \Gamma$ and $\Sigma$ is smooth up to the boundary and it is made up of $\F$-minimal and $\F$-stable two-sided surfaces $\Sigma^{(1)},\cdots,\Sigma^{(R)}$ such that:
        \begin{align}
           \label{gg}
           \sum_{i=1}^R \mathrm{genus}(\Sigma^{(i)}) \leq \mathrm{genus}(\Sigma_k)\,,
        \end{align}
        for large enough $k$.
        
        \item In the case that $\de A$ is made up of some $\F$-convex parts, but is not entirely $\F$-convex, then the boundary smoothness and estimates are carried all the way up to the convex parts of the boundary with uniform curvature estimate only depending on the mass, $F,\de A$ and the ambient manifold $N$ (the genus bounds \eqref{gg} may not hold, since the genus might not be well defined in the limit).
    \end{enumerate}
\end{thm}

We apply the previous anisotropic versions of Meeks-Simon-Yau theory to prove the following anisotropic analogue of Simon-Smith min-max theorem with at most
one possible isolated singular point.
\begin{thm}
\label{thm:aniso-simon-smith-one-singularity}
Let $N$ be a closed $3$-manifold with a Riemannian metric. For any
saturated set of generalized families of surfaces $\Lambda$, there exists
a min-max sequence obtained from $\Lambda$ which converges in the sense of
varifolds to an $\F$-stationary integral varifold $V$ with anisotropic
width
$m_0(\Lambda)$.
Moreover, there exists a set $\mathcal S\subset N$ with
$\#\mathcal S\leq 1$ such that $\operatorname{spt}V\setminus\mathcal S$
is a smooth embedded $\F$-minimal surface (multiplicity is allowed).

More precisely, there exist positive integers $n_1,\cdots,n_L$ and
pairwise disjoint connected smooth embedded $\F$-minimal surfaces $\Sigma^{(1)},\cdots,\Sigma^{(L)}\subset N\setminus\mathcal S$
which are locally $\F$-stable in $N\setminus\mathcal S$, such that
\[
    V=\sum_{i=1}^L n_i\,\mathbf v(\Sigma^{(i)}).
\]
Furthermore,
\[
    \overline{\Sigma^{(i)}}\setminus\Sigma^{(i)}\subset\mathcal S
    \qquad\text{for each } i=1,\cdots,L .
\]

\end{thm}

The remaining possible singularity is removed under either of the two
structural assumptions \eqref{assumption1}-\eqref{assumption2} on \(F\), by means of the following removal of singularity result, which we believe to be of independent interest.

\begin{thm}
\label{thm:removable-singularity}
Assume that $F$ satisfies either \eqref{assumption1} or
\eqref{assumption2}. Let $U\subset N$ be an open set and let
$p\in U$, and
$\Sigma\subset U\setminus\{p\}$
be a smooth properly embedded $\F$-stationary surface which is
$\F$-stable in $U\setminus\{p\}$. Suppose that
$p\in\overline{\Sigma}$. Then $p$ is a removable singularity. More
precisely, $\overline{\Sigma}\cap U$ is a smooth embedded
$\F$-minimal surface in $U$, and it is $\F$-stable in $U$.
\end{thm}

Combining the preceding two results with an anisotropic version of Ketover's optimal genus bound \cite{K}, we obtain the
following sharp anisotropic analogue of the Simon--Smith theorem.

\begin{thm}
\label{thm:aniso-simon-smith-ketover}
Assume that $F$ satisfies either \eqref{assumption1} or
\eqref{assumption2}. Let $N$ be a closed $3$-manifold with a Riemannian
metric. For any saturated set of generalized families of surfaces
$\Lambda$, there exists a min-max sequence $\{\Sigma_t^j\}$ obtained from $\Lambda$ that
converges in the sense of varifolds to
\[
    V=\sum_{i=1}^L n_i\,\mathbf v(\Gamma^i),
\]
where the $\Gamma^i$'s are pairwise disjoint smooth embedded
$\F$-minimal surfaces and $n_i\in \mathbb N$. 
Furthermore, denote by $\mathcal O$ and
$\mathcal U$ the set of indices \(i\) of orientable and non-orientable connected
components among the $\Gamma^i$'s, respectively.
Then
\[
    \sum_{i\in \mathcal O}
    n_i\,\operatorname{genus}(\Gamma^i)
    +
    \frac12
    \sum_{i\in \mathcal U}
    n_i\bigl(\operatorname{genus}(\Gamma^i)-1\bigr)
    \leq
    \liminf_{j\to\infty}\,\liminf_{\tau\to t_j}
    \operatorname{genus}(\Sigma_\tau^j).
\]
\end{thm}

\begin{rmk}
\label{rmk:aniso-almgren-pitts}
Although the main focus of this paper is the anisotropic Simon--Smith
construction, the removable singularity theorem (\cref{thm:removable-singularity}) also applies to the
three-dimensional anisotropic Almgren--Pitts theory developed by De Philippis-De Rosa \cite{DePhilippisDeRosa} and to its multiparameter generalization.

More precisely, assume that $F$ satisfies either \eqref{assumption1} or
\eqref{assumption2}. Let $N$ be a closed $3$-manifold with a Riemannian
metric, and let $\Pi$ be a non-trivial $m$-parameter Almgren--Pitts
homotopy class of sweepouts of $N$. Following the proof in \cite{DePhilippisDeRosa}, the anisotropic
Almgren--Pitts min--max construction produces an integral
$\F$-stationary varifold realizing the corresponding anisotropic width,
which is almost minimizing in annuli. The replacement regularity theory
then gives a smooth embedded locally $\F$-stable $\F$-minimal surface away from finitely many
isolated singularities.

Under either \eqref{assumption1} or \eqref{assumption2}, these isolated
singularities are removable by \cref{thm:removable-singularity}. Hence
the min--max varifold is induced by a smooth embedded closed
$\F$-minimal surface. While the singular point in \cite{DePhilippisDeRosa} was already removed in \cite{DDL}, the argument therein relies on a delicate construction of nested sweepouts, which does not  extend to multiparameter min–max schemes. On the contrary, our results apply to the multiparameter Almgren-Pitts anisotropic min-max, and allows to construct higher-index minimal surfaces. This also paves the way to studying the anisotropic analogue of the multiplicity one conjecture \cite{Zhou2020}, as well as its applications to Yau’s conjecture \cite{YC} for generic metrics via the Morse index conjecture \cite{MarquesNeves2016MorseIndexMultiplicity,MarquesNeves2018MorseIndex}.
\end{rmk}

\section{Anisotropic Meeks--Simon--Yau}\label{sec:aniso-MSY}
In this section, we prove that any $\F$-minimizing sequence in a fixed isotopy class of surfaces, with or without boundary, converges to a smooth limit while preserving genus bounds. In particular, we establish \cref{thm:white-improvement}, \cref{thm:Meeks-Simon-Yau-without-boundary}, and \cref{thm:Meeks-Simon-Yau-with-boundary}. The main challenge is that restricting the minimization problem to a prescribed isotopy class significantly reduces the available competitors, making regularity substantially more difficult to obtain. For the area functional, this problem was studied in detail in \cite{MSY} and \cite{AS}, where the key ingredients are the monotonicity formula and the analysis of tangent cones, building on \cite{Allard-geodesic}. However, the monotonicity formula is highly rigid and generally fails for anisotropic integrands, even for perturbations that are arbitrarily close to the area functional. Consequently, tangent varifolds are not known to be conical, making the classical approach unavailable.

To overcome these difficulties, we abandon the analysis of tangent varifolds and instead rely on the regularity theory developed by Pitts \cite{Pitts}. We first prove in \cref{thm:integrality-limit-varifold} that the limit varifold of any isotopic $\F$-minimizing sequence is an integral stationary varifold, following essentially the same strategy as in \cite{AS}. We then combine the key estimate of \cref{lemma:tentacles-estimate} with ideas from \cite{HS} and the existence of at least one $\F$-minimizing disk established in \cite{W-1} to show that the limit varifold admits replacements, in the sense of Pitts, both in the interior and up to the boundary. This allows us to implement Pitts' regularity scheme and deduce the smoothness of the limit varifold. Finally, to transfer the topological information to the limit, we employ the $\gamma$-reduction technique introduced in \cite{MSY}, which yields the desired genus control.

\subsection{Preliminary lemmas}
We first observe the following consequence of
\cite[Lemma 1]{MSY}.
\begin{lemma}\label{lem:anis-1}
There exists \(\varepsilon_0\in(0,1)\) (independent of \(N\) and \(\rho_0\)) such that if \(\Sigma\in \mathcal{C}\) satisfies
\begin{equation}\label{eq:anis-smallness}
  \mathbf F(\Sigma,B_{\rho_0}(x_0))< \varepsilon_0^2\rho_0^2
\end{equation}
for each $x_0\in N$, then there exists a unique compact set \(K_\Sigma\subset N\) with \(\partial K_\Sigma=\Sigma\) and
\begin{equation}\label{eq:vol-control}
    \operatorname{Vol}(K_\Sigma\cap B_{\rho_0}(x_0)) \leq C_1\rho_0^3,\quad\forall\,x_0\in N,
\end{equation}
and moreover
\begin{equation}\label{eq:ani-iso-ineq}
    \operatorname{Vol}(K_\Sigma) \leq C_2\bigl(\mathbf{F}(\Sigma)\bigr)^{3/2},
\end{equation}
where the constants \(C_1,C_2>0\) depend only on the geometric constants of \(N\), \(\rho_0\), and \(\lambda\).
\end{lemma}

\begin{proof}
By the uniform comparability assumption \eqref{H:comparability}, for every \(x_0\in N\),
\[
    \mathcal H^2(\Sigma\cap B_{\rho_0}(x_0))
    \le
    \lambda^{-1}\mathbf F(\Sigma,B_{\rho_0}(x_0)).
\]
Let \(\varepsilon_0^2=\lambda\delta_{\mathrm{MSY}}^2\), then by \eqref{eq:anis-smallness} we have
\[
    \mathcal H^2(\Sigma\cap B_{\rho_0}(x_0))
    <
    \delta_{\mathrm{MSY}}^2\rho_0^2
    \qquad\text{for every }x_0\in N,
\]
where $\delta_{\mathrm{MSY}}$ is exactly the same constant as in \cite[Lemma 1]{MSY}.
This is precisely the smallness hypothesis of \cite[Lemma 1]{MSY}.
Applying \cite[Lemma 1]{MSY} gives the unique compact set
\(K_\Sigma\subset N\) with \(\partial K_\Sigma=\Sigma\) and
\[
    \operatorname{Vol}(K_\Sigma\cap B_{\rho_0}(x_0))
    \le
    \delta_{\mathrm{MSY}}^2\rho_0^3
    \qquad\text{for every }x_0\in N.
\]
This proves \eqref{eq:vol-control} with
\(C_1=\delta_{\mathrm{MSY}}^2\).

The same lemma also gives
\[
    \operatorname{Vol}(K_\Sigma)
    \le
    c_{\mathrm{MSY}}
    \bigl(\mathcal H^2(\Sigma)\bigr)^{3/2}.
\]
Using \eqref{H:comparability} once more,
we deduce \eqref{eq:ani-iso-ineq} with \(C_2=c_{\mathrm{MSY}}\lambda^{-3/2}\).
\end{proof}

We modify \cite[Lemma 2]{MSY} as follows:
\begin{lemma}\label{lem:disjoint-boundaries}
    Suppose $M_1,\cdots,M_R$ are diffeomorphic to $\mathbf{D}$, suppose $M_i\setminus\partial M_i\subset A\setminus\partial A$, $\partial M_i\subset\partial A$, $i=1,2,\cdots,R$, where $A\subset N$ is diffeomorphic to $\mathbf{B}$, and suppose that $\partial M_i\cap\partial M_j=\emptyset$ and that either $M_i\cap M_j=\emptyset$ or $M_i$ intersects $M_j$ transversely for all $i\neq j$.

    Then there exist pairwise disjoint $\tilde{M}_1,\cdots,\tilde{M}_R$ with $\tilde{M}_i\setminus\partial\tilde{M}_i\subset A\setminus\partial A$, $\partial\tilde{M}_i=\partial M_i$ and $\mathbf{F}(\tilde{M}_i)\leq\mathbf{F}(M_i)$, $i=1,2,\cdots,R$.
\end{lemma}
\begin{proof}
    The proof follows the same line of argument as that of Lemma 2 in \cite{MSY}, with the only modification being the replacement of the area functional by the anisotropic energy $\mathbf{F}(\cdot)$. Since the adaptation is straightforward, we omit the details here.
\end{proof}

In what follows, we modify \cite[Lemma 3]{MSY} to work in the case of anisotropic integrands.
\begin{lemma}\label{lem:lemma3-anisotropic-MSY}
    Suppose $U\subset B_{\rho_0}(x_0)$, $\overline{U}\approx \mathbf{B}$, and $U$ is convex in the strong sense that $d_U$ is a convex function on $\{x\in N: d_U(x) < \theta\rho_0 \}$ for some $\theta\in(0,\frac12)$, and let $\beta\geq1$ be some constant such that, for each $s\in(\theta\rho_0/2,\theta\rho_0)$,
    \begin{align}\label{lemma3:isoperimetric-condition}
        \min\{\F(E),\F(\de U(s)\setminus E)\} \leq \beta (\hau^1(\de E))^2\,,
    \end{align}
    whenever $E$ is a disc contained in $\de U(s)$.

    Let $\delta_1:=\min\{\lambda^{1/2} \varepsilon_0,\,c\,\lambda\beta^{-1/2}\theta\}$, where
$\varepsilon_0$ is chosen as in \cref{lem:anis-1} and $c>0$ is a universal constant.
If $M$ is any smooth disc with $\partial M\subset N\setminus U$, with $M$ intersecting $\partial U$ transversely, and if
\begin{equation}\label{lemma3:area-bound-assumption}
\mathcal H^2(\partial U)+\mathcal H^2(M)\le \delta_1^2\rho_0^2/16,
\end{equation}
    and if ${\Sigma_0}$ is any component of $M\setminus U$ with $\de M \cap {\Sigma_0} = \emptyset$. Then there is a unique $K_{\Sigma_0} \subset N\setminus U$ such that
    \begin{align}\label{lemma3:vol-bound-assumption}
        \vol(K_{\Sigma_0}) \leq c\delta_1^3\rho_0^3,\;\;\;\; \de K_{\Sigma_0} = {\Sigma_0} \cup Q\,,
    \end{align}
    where $Q \subset \de U$ is a compact (not necessarily connected) surface with $\de Q = \de {\Sigma_0}$ and
    \begin{itemize}
        \item The bound below holds:
        \begin{align}\label{lemma3:aniso-main-result}
        \F(Q) < \left[1+C(\lambda,\beta)\left(\sqrt{\hau^2(\de U) + \hau^2(M) + |K_{\Sigma_0}|}\right)\right]\F\left({\Sigma_0}\cap U(\theta\rho_0)\right)\,.
    \end{align}
    \item Moreover if any principal curvature of $\de U$ is at least $(100\,\diam(U))^{-1}$ pointwise, then:
    \begin{align}\label{lemma3:aniso-main-result-better}
        \F(Q) < \F\left({\Sigma_0} \cap U(\theta\rho_0)\right)\,.
    \end{align}
    \end{itemize}
\end{lemma}
\begin{proof}
    Most of the proof of \cite[Lemma 3]{MSY} follows in a similar way with some modifications. First, we find a closed set $Q_0 \subset \de U$ with $\de Q_0 = \de{\Sigma_0}$. Then \cref{lemma3:area-bound-assumption} and the quantification of $\delta_1$ implies that $$ \F(Q_0) + \F({\Sigma_0}) \leq \lambda^{-1}(\mathcal H^2(Q_0) + \mathcal H^2({\Sigma_0})) \leq \lambda^{-1}\delta_1^2\rho_0^2 \leq \epsilon_0^2\rho_0^2$$ and \cref{lem:anis-1} implies that there is a compact $W$ with
    \begin{align}\label{lemma3:temp-volume-bound}
        \vol(W) \leq c\delta_1^3\rho_0^3,\;\;\;\;\de W = Q_0\cup{\Sigma_0}\,.
    \end{align}
    Then we set $K_{\Sigma_0} = W\setminus U$. Then the assumption \cref{lemma3:vol-bound-assumption} holds with $Q= Q_0$ or $Q = \de U \setminus Q_0$. Then, like \cite{MSY}, for $t\in[0,\theta\rho_0]$ we define 
    \begin{align}
        Q_t= K_{\Sigma_0} \cap \{x: d_U(x) = t\}\,,\;\;\;\; E_t = \{x\in{\Sigma_0}: d_U(x) < t\}\,.
    \end{align}
    In what follows next we diverge from \cite{MSY} to adapt the proof to the anisotropic case. For this aim consider the vector field $D_\nu G(x,\nabla d_U)$ and note that:
    \begin{align}
        \div(D_\nu G(x,\nabla d_U(x))) = \sum_i D_{e_i}(D_\nu G)\cdot e_i + \ang{D^2_{\nu\nu} G:\nabla^2 d_U}\,.
    \end{align}
    However, by assumption $G$ is strictly convex and $d_U$ is convex in the direction of $(\nabla d_U)^{\perp}$, hence the second term above is positive and the first term is bounded by the $C^2$ norm of the integrand:
    \begin{align}
        \div(D_\nu G(x,\nabla d_U(x))) > -\frac{1}{\lambda} \quad\text{on }\{0<d_U<\theta\rho_0\}.
    \end{align}
    Note the bound above is independent of $U$. Then we use the divergence theorem on $\{x\in K_{\Sigma_0}: t_1 \leq d_U(x) \leq t_2\}$
    \begin{align}
    \begin{aligned}\label{inequality-1-lemma3}
        -\frac{1}{\lambda}\vol(\{x\in K_{\Sigma_0}: t_1 \leq d_U(x) \leq t_2\}) < \int_{\{x\in K_{\Sigma_0}: t_1 \leq d_U(x) \leq t_2\}} \div(D_\nu G(x,\nabla d_U(x)))\\
        = \int_{Q_{t_2}} \ang{D_\nu G(x,\nabla d_U),\hat{n}} - \int_{Q_{t_1}} \ang{D_\nu G(x,\nabla d_U),\hat{n}} + \int_{E_{t_2}\setminus E_{t_1}} \ang{D_\nu G(x,\nabla d_U),\hat{n}}\,.
    \end{aligned}
    \end{align}
    Note the strict inequality holds for $0\leq t_1<t_2 \leq \theta\rho_0$. Since $Q_t$ is the level sets of $d_U$, it is easy to see that $\nabla d_U$ agrees with the outward normal on $Q_t$. Now, invoking the assumptions \cref{integrand:homogen}, \cref{integrand:convexity} and coarea formula combined with \eqref{H:comparability}, (using the bound \cref{lemma3:temp-volume-bound}) we arrive at:
    \begin{align}\label{lemma3:temp-1}
        0 < \left[\F(Q_{t_2}) -  \F(Q_{t_1})\right] + \left[\F(E_{t_2}) - \F(E_{t_1})\right] + C\int_{t_1}^{t_2}\F(Q_t)\,\mathrm{d}t \,.
    \end{align}
    We follow the same strategy of \cite{MSY}. For $t_1=0$ we get for $t \in (0,\theta\rho_0]$:
    \begin{align}
        \F(Q) - \F(E_t) < \F(Q_t) + C(\lambda)\int_{0}^{t}\F(Q_s)\,\mathrm{d}s\,.
    \end{align}
    Moreover by the coarea formula and \cref{lemma3:temp-volume-bound}:
    \begin{align}
        \int_{0}^{\theta\rho_0} \F(Q_t)\,\mathrm{d}t \leq \frac{1}{\lambda}\vol(K_{\Sigma_0}\cap U(\theta\rho_0))\leq \frac{c}{\lambda}\delta_1^3\rho_0^3\,,
    \end{align}
    hence
    \begin{align}
        \F(Q_t) \leq 4\frac{c}{\lambda} \theta^{-1}\delta_1^3\rho_0^2\,,
    \end{align}
    for a subset of $(0,\theta\rho_0)$ with measure at least $3\theta\rho_0/4$. We use \cref{lemma3:temp-1}, the area bound \cref{lemma3:area-bound-assumption} on $N$ and the definition of $\delta_1$ to see that (for a possibly different $c$):
    \begin{align}
        \F(Q_t) \leq C(\lambda)\left[\hau^2(\de U) + \hau^2(M) + |K_{\Sigma_0}| \right]\leq \frac{c}{\lambda} \delta_1^2\rho_0^2\;\;\;\text{ for all } t\in[0,3\theta\rho_0/4) \,.
    \end{align}
    On the other hand that by the condition \eqref{H:comparability}, \cref{lemma3:isoperimetric-condition} and $\de E_t = \de Q_t$:
    \begin{align}
    \begin{aligned}
        \F(Q_t) &\leq \lambda^{-1} \hau^2(Q_t) \leq \lambda^{-1} \beta |\de Q_t|^2 \\&\leq \lambda^{-3} \beta \left(\int_{\de E_t} G(x,\hat{n}|_{E_t}) \,\mathrm{d}\hau^1\right)^2 \leq 
        \beta \lambda^{-3} \left(\frac{\mathrm{d}}{\mathrm{d}t} \F(E_t)\right)^2\,.
        \end{aligned}
    \end{align}
    for almost all $t\in (\theta\rho_0/2,\theta\rho_0)$. First note that for $t\in[0,3\theta\rho_0/4)$ by the calculation above:
    \begin{align}
        \begin{aligned}
        \int_{0}^t \F(Q_t)\,\mathrm{d}t &\leq \sup_{0\leq t\leq 3\theta\rho_0/4}\sqrt{\F(Q_t)}\int_{0}^t\sqrt{\F(Q_t)}\,\mathrm{d}t \\ &\leq C(\lambda,\beta)\sup_{0\leq t\leq 3\theta\rho_0/4}\sqrt{\F(Q_t)} \int_{0}^t \frac{d}{ds}\F(E_s) \,\mathrm{d}s\\
        &\leq C(\lambda,\beta)\left(\sqrt{\hau^2(\de U) + \hau^2(M) + |K_{\Sigma_0}|}\right)\F(E_t)
        \end{aligned}
    \end{align}
    Now we name $$f(t) = \F(Q) - \left[1+C(\lambda,\beta)\left(\sqrt{\hau^2(\de U) + \hau^2(M) + |K_{\Sigma_0}|}\right)\right]\F(E_t)\,,$$ with $C$ the constant in the above display, we arrive at:
    \begin{align}
        f(t) < \beta\lambda^{-2} (f'(t))^2\,.
    \end{align}
    Notice that $f'\!=-\frac{\mathrm{d}}{\mathrm{d}t} \F(E_t)\le 0$, so $f(t)$ is a non-increasing function, hence by naming $h(t) = \max\{0,f(t)\}$, we see that $-h'(t)\ge (\lambda/\sqrt\beta)\sqrt{h(t)}$, i.e.
    \begin{align}
        \frac{\mathrm{d}}{\mathrm{d}t}\sqrt{h(t)} \le -\frac{\lambda}{2\sqrt\beta}\;\;\;\text{ for all } t\in(\theta\rho_0/2,3\theta\rho_0/4)\,,
    \end{align}
    provided that (by contradiction) $f(3\theta\rho_0/4) \geq 0$. Integrating this, we see that:
    \begin{align}
        \sqrt{f(\theta\rho_0/2)} > \lambda\beta^{-1/2}\theta\rho_0/8\,.
    \end{align}
    However $\sqrt{\F(Q)} < \delta_1\rho_0$, hence:
    \begin{align}\label{lemma3:last-inequality}
        \delta_1 \geq \lambda\beta^{-1/2}\theta/8\,.
    \end{align}
    On the other hand, we can take $\delta_1\leq c\lambda\beta^{-1/2}\theta$ for $c>0$ small enough so that \cref{lemma3:last-inequality} is not possible. This means that $f(\theta\rho_0) \leq f(3\theta\rho_0/4) < 0$, which is indeed our claim.

    In the case where $\de U$ has all its principal curvature at least $1/(100\,\diam(U))$ is simple, in fact we can calculate that:
    \begin{align}
        \div(D_\nu G(x,\nabla d_U(x))) \geq 0\,,
    \end{align}
    if $c$ is taken small enough. The rest of the proof will remain unchanged, except the volume error terms will be omitted. With this we conclude the proof.
\end{proof}

The next lemma will be crucial in the proof later on:
\begin{lemma}[Tentacles estimate]\label{lemma:tentacles-estimate}
    There exists $\rho_1(N,\lambda)>0$ small enough with the following property: Let $U\subset N$ be a smooth open convex set with $\diam(U)<\rho_1$ bi-Lipschitz to $\B_{\diam(U)/2}$ with constant at most $10$ and such that all principal curvatures of $\de U$ are at least $(100\,\diam(U))^{-1}$. Moreover, take $M\in \mathcal{M}$ with $\de M \subset N\setminus U$ and with $M$ intersecting $\de U$ transversely. Moreover for any $P\in\mathcal{M}$ with $\de P = \de M$ we have:
    \begin{align}
        \F(M) \leq \F(P) + \epsilon\,.
    \end{align}
    Now take ${\Sigma_0}$ to be the component of $M\setminus U$ with ${\Sigma_0} \cap \de M = \emptyset$ (not necessarily connected). Then we have:
    \begin{align}
        \hau^{2}\left(\left\{x\in {\Sigma_0}:\, \dist(x,U) \geq \diam(U)/8\right\}\right) \leq C(N,\lambda)\epsilon\,,
    \end{align}
    for some constant $C(N,\lambda)>0$.
\end{lemma}
\begin{proof}
    Name $\rho = \diam(U)/2$, and take $d_U(x) = \dist(x,U)$. By the assumptions, we see that the two biggest eigenvalues of $\nabla^2 d_U$ are at least $C(N)/\rho$ in $U(C\rho)$ where $C > 10\delta_1^{-1}$ in \cref{lem:lemma3-anisotropic-MSY}. By the uniform \(C^3\)-bounds in \eqref{H:regularity}, after choosing \(\rho_1=\rho_1(N,\lambda)>0\) sufficiently small, we have
    \[
        \div\bigl(D_\nu G(x,\nabla d_U)\bigr)\ge\frac{C(N,\lambda)}{\rho}.
    \]
    Now define $Q\in \de U$ similarly and $K_{\Sigma_0}$ such that $\de K_{\Sigma_0} = Q + {\Sigma_0}$. Now note that the strict convexity of the integrand $F$, tells us for any two unit vectors $\nu,\bar{\nu}$:
    \begin{align}
        G(x,\nu) - C(\lambda)|\overline{\nu}-\nu|^2 \geq \ang{D_\nu G(x,\overline{\nu}),\nu}\,.
    \end{align}
   Then following the proof of \cref{lem:lemma3-anisotropic-MSY} and inserting the above display into the last term of \cref{inequality-1-lemma3} and collecting the error terms coming from the divergence term, we see that:
    \begin{align}
        \F(Q) + \frac{C(\lambda,N)}{\rho}|K_{\Sigma_0}\cap U(C\rho)| + C(\lambda)\int_{{\Sigma_0}\cap U(C\rho)} |n - \eta|^2 \,\mathrm{d}\hau^2 \leq \F({\Sigma_0}\cap U(C\rho))\,.
    \end{align}
    Here $n$ is the outward unit normal to ${\Sigma_0}$. Then we use the almost minimizing property of $M$ and compare with $P = \overline{Q}\cup(M\setminus {\Sigma_0})$, we see that:
    \begin{align}
        \F(M) \leq \F(P) + \epsilon = \F(M\setminus {\Sigma_0}) + \F(Q) + \epsilon\,.
    \end{align}
    Hence:
    \begin{align}\label{estimate:remainder-area-diff}
        \frac{1}{\rho}|K_{\Sigma_0}\cap U(C\rho)| + \F({\Sigma_0} \setminus U(C\rho)) + \int_{{\Sigma_0}\cap U(C\rho)} |n - \eta|^2 \,\mathrm{d}\hau^2 \leq C(N,\lambda)\epsilon\,.
    \end{align}
    For the sake of notation, name $K_{\Sigma_0} \cap U(C\rho) = S$ and $\de S\setminus \de U = \Sigma$. The above display (since on $\de U(C\rho)$ we have $n = \eta$) reads:
    \begin{align}
        \frac{1}{\rho}|S| + \int_{\Sigma} |n-\eta|^2\,\mathrm{d}\hau^2 \leq C\epsilon\,.
    \end{align}
    Now we approximate $\Sigma$ as a Lipschitz normal graph over $\de U$ using standard theory of currents and functions of bounded variation (Note that $U$ is bi-Lipschitz equivalent to $\textbf{B}_{\rho}$ with constant $10$). For $\ell>0$ small, there exists a function $h:\de U \to \R$ such that:
    \begin{align}\label{estimate:bad-set-esitmate}
        \hau^2(\Sigma \setminus \{x + n(x) h(x) :\, x\in \de U\}) \leq C\epsilon/\ell\,,
    \end{align}
    with $\Lip(h)\leq C\ell$, by thresholding the maximal function the the tilt-excess. By the estimate of $|S|$ we also estimate that:
    \begin{align}
        \int_{\de U} |h| \,\mathrm{d}\hau^2 \leq C\rho \epsilon/\ell\,.
    \end{align}
    Hence we can see that:
    \begin{align}
        \hau^2(\{x\in\de U: h(x) \geq \rho/20 \}) \leq C\epsilon/\ell\,.
    \end{align}
    Moreover since $h$ is $C\ell$-Lipschitz, we gather that:
    \begin{align}\label{estimate:lipszhitz-upper-estimate}
        \hau^2(\{x+h(x)n(x): h(x) \geq \rho/20 \}) \leq C\epsilon/\ell\,.
    \end{align}
    Now we can see that combining \cref{estimate:lipszhitz-upper-estimate} and \cref{estimate:bad-set-esitmate}:
    \begin{align}
    \begin{aligned}
        \hau^2(\{x\in \Sigma:\, \dist(x,U)\geq \rho/4\}) \leq &\hau^2(\Sigma\setminus \{x+h(x)n(x):\, x\in \de U\}) \\ &+ \hau^2(\{x+h(x)n(x):\, h(x) \geq \rho/20\})
        \leq C\epsilon/\ell\,.
    \end{aligned}
    \end{align}
    Putting this together with \cref{estimate:remainder-area-diff} (fixing some $\ell_0>0$) we gather that:
    \begin{align}
        \hau^{2}\left(\left\{x\in {\Sigma_0}:\, \dist(x,U) \geq \diam(U)/8\right\}\right) \leq C(N,\lambda)\epsilon\,.
    \end{align}
\end{proof}
We also recall a compactness theorem for minimizing disks, which is essentially due to White \cite{W-1}:
\begin{lemma}[Compactness of minimizing disks \cite{W-1}]\label{lemma:disk-compactness}
    Let $A\subset N$ be a smooth open convex set with $\diam(A) \leq \rho_0$ for $\rho_0$ small enough. Let $M_k\in\mathcal{M}$ be a sequence of smooth embedded disks $M_k\subset A$ with simple closed curves as boundaries $\de M_k \subset \de A$ that minimize $\F$-area. Then the following is true (up to subsequences):
    \begin{enumerate}
        \item There exists a limit $\F$-minimal surface $M$ (possibly empty), smooth in the interior, such that for any compact set $K\subset A$ the disks $M_k\cap K$ smoothly converges (sub-sequentially) to $M$. In other words we have $\limsup_{k\to\infty} \|M_k\|_{C^{2,\alpha}(K)} < C$ where $C$ depends only on $K,A$ and $N$ and $\F$.
        \item If $M_k$ is made up of a collection of least $\F$-area disks (with $\de M_k$ being a collection of disjoint closed loops on $\de A$) with a uniform area bound along the sequence i.e. $\limsup_{k\to\infty} \hau^2(M_k) < \infty$, then the same conclusions hold.
        
        \item Assume moreover that there is an open portion of the boundary $S\subset \de A$ such that $$\limsup_{k\to\infty}\|\partial M_k\cap S\|_{C^{2,\alpha}_{\mathrm{loc}}} < \infty\,,$$ then, the limit $M$ is $C^{2,\beta}$ for any compact set $K\subset A\cup S$ and $M_k$ converges in $C^{2,\beta}$ to $M$ on $K$ for any $\beta<\alpha$ (though this limit may be empty in the set $S$ by \cite[example 4.3]{HS}). In other words we have $\limsup_{k\to\infty} \|M_k\|_{C^{2,\alpha}(K)} < C$ where $C$ depends only on $K,A$ and $N$ and $\F$ and $\limsup_{k\to\infty}\|\partial M_k\cap S\|_{C^{2,\alpha}_{\mathrm{loc}}}$.
        
        \item Take $S$ to be a connected open portion of the boundary $\de A$ and let $M_k$ be a single $\F$-minimizing disk. Moreover assume that there is a simple arc $\Gamma\subset S $ such that $\de M_k \cap S = \Gamma_k$ are a sequence of simple arcs with uniform $C^{2,\alpha}$ norms that converge smoothly to $\Gamma$. Then $M$ is $C^{2,\beta}$ up to its boundary in $S$ for any $\beta<\alpha$ and a component of $M$ extends smoothly up to $\Gamma$.
    \end{enumerate}
    Here $\alpha$ is equal or less than the regularity of the integrand $F$.
\end{lemma}
\begin{proof}
    The proof of the first three items is essentially contained in the proof of \cite[Section 2]{W-1} in the last two paragraphs of page $422$ and $423$. To prove the last point, we only need to check that $\Gamma$ is indeed part of the boundary in the limit. First it is evident that $\de M \subset \Gamma$, it remains to show the other inclusion. Indeed take any point $x\in\Gamma$ and apply \cref{lem:density-boundary-estimate} to see that for large enough $k$ we have:
    \begin{align}
        \frac{\hau^2(B_\epsilon(x)\cap M_k)}{\epsilon^2} \geq C\,,
    \end{align}
    for small enough $\epsilon(x,\Gamma,N,A)$. This means that $M\cap B_{\epsilon}(x_0)$ is non empty, hence $\Gamma = \de M \cap S$.

\end{proof}
\begin{rmk}
    Note that the limit might not be a disk in general, as demonstrated by \cite[example 3.4]{HS}.
\end{rmk}

\subsection{Replacement in balls and the filigree lemma}
In this subsection, we recall a local replacement result. Roughly speaking, \cref{prop:replacement-thm-AS-aniso} produces a replacement in a strongly convex ball which decomposes into stacked disks which are individually almost $\F$-minimizing. 
The result is an anisotropic analogue of the Almgren--Simon replacement theorem \cite[Theorem 1]{AS}, and should not be confused with the varifold replacement used in the Almgren--Pitts min--max theory. In the following, given an open set 
\(U\) and \(\theta>0\), we will use the notation
\begin{align}
    U_\theta := \{x\in U:\, \dist(x,\de U)\geq \theta\}\,.
\end{align}
\begin{proposition}[Almgren-Simon \cite{AS} replacement]\label{prop:replacement-thm-AS-aniso}
    Suppose $\theta>0$ is given and
    \begin{enumerate}
        \item $U\subset B_{\rho_0/2}$, $U\approx \mathbf{B}_1$ is a $C^2$ strongly convex open set in $N$, satisfying all the assumptions of \cref{lem:lemma3-anisotropic-MSY}.
        \item $M\in\mathcal{M}$, $\de M \setminus \de U$ is not contained in any $K_{\Sigma_0}$, where ${\Sigma_0}$ is any component of $M\setminus U$ with $\de M\cap{\Sigma_0} = \emptyset$ and $K_{\Sigma_0}$ is defined as in \cref{lem:lemma3-anisotropic-MSY}.
        \item $M$ intersects $\de U$ transversely; in case $\de M \cap \de U \not=\emptyset$, we will always take this to mean there is a $C^2$ (open) surface $\Sigma$ with $M\subset \Sigma$, with $(\Sigma\setminus M)\cap  U = \emptyset$ and with $\Sigma$ intersecting $\de U$ transversely.
        \item The bound \cref{lemma3:area-bound-assumption} in \cref{lem:lemma3-anisotropic-MSY} holds.
    \end{enumerate}
    Then there is $\tilde{M} \in \mathcal{M}$ such that
    \begin{enumerate}
        \setcounter{enumi}{4}
        \item $\de \tilde{M} = \de M$, $\tilde{M}\setminus U \subset M\setminus U$, $\tilde{M}\cap U_\theta \subset M \cap U_\theta$;
        \item $\tilde{M}$ intersects $\de U$ transversely, in the same sense as above.
        \item $\F(\tilde{M}) + \left(1-C(\lambda,\beta)\sqrt{\hau^2(\de U) + \hau^2(M) + |K_{\Sigma_0}|}\right)\F((M\setminus\tilde{M})\cap U_\theta) \leq \F(M)$.
        \item $\tilde{M}\cap \overline{U}$ is a disjoint union $\bigcup_{j=1}^k \Sigma_j$ of elements $\Sigma_j\in \mathcal{M}$ (and consequently $\tilde{M} \subset \overline{U}$ in case $\de M \subset \de U)$.
    \end{enumerate}
    If in addition to the hypotheses (1)-(4) we have
    \begin{enumerate}
        \setcounter{enumi}{8}
        \item $\F(M) \leq \F(P) + \theta$ for every $P\in\mathcal{M}$ with $\de P = \de M$, then there are non-negative numbers $\theta_1,\cdots,\theta_k$ with $\sum_{i=1}^k \theta_i \leq \theta$ and
        \item $\left(1-C(\lambda,\beta)\sqrt{\hau^2(\de U) + \hau^2(M) + |K_{\Sigma_0}|}\right) \F(\Sigma_j) \leq \F(P) + \theta_j$ for every $P\in\mathcal{M}$ with $\de P = \de \Sigma_j$ for $j=1,\cdots,k$.
    \end{enumerate}
    Moreover, if the principal curvature of $\de U$ is at least $(100\,\diam(U))^{-1}$ and $\rho_0>0$ is chosen small enough, we may remove the multiplier in point $(7)$ and $(10)$. In fact in this case the following is true:
    \begin{enumerate}
        \setcounter{enumi}{10}
        \item $\F(\tilde{M}) + \F((M\setminus\tilde{M})\cap U_\theta) \leq \F(M)$.
        \item $\F(\Sigma_j) \leq \F(P) + \theta_j$ for every $P\in\mathcal{M}$ with $\de P = \de \Sigma_j$ for $j=1,\cdots,k$.
    \end{enumerate}
    \begin{proof}
        The proof of \cite[Theorem 1]{AS} follows line by line, with the main modification that \cite[(3.1)]{AS} is replaced by inequality \cref{lemma3:aniso-main-result} in \cref{lem:lemma3-anisotropic-MSY}. In the rest of the proof, we may replace the isotropic $\hau^2$ by the anisotropic area $\F$.
    \end{proof}
\end{proposition}
In parallel with \cite[Remark 3.12]{AS}, we also record the following consequence.

\begin{rmk}\label{rmk:3-12AS}
    One can modify the sequence $\{M_k\}$ in \cref{thm:white-improvement} to a sequence $\{\tilde{M}_k\}$ so that
    $\tilde{M}_k \subset \mathbf{B}\cup \de M_k$,
    without changing the varifold limit.
\end{rmk}

\begin{rmk}
    In applications, we only use conclusions (11) and (12) of
    \cref{prop:replacement-thm-AS-aniso}, choosing the balls so that the principal curvatures of their boundaries are sufficiently large. 
    Moreover, the result also applies to balls intersecting the boundary of the ambient manifold. 
    In view of \cref{rmk:3-12AS}, no replacement is needed along the boundary of the ambient manifold, and the same conclusions remain valid.
\end{rmk}

We also adapt the \textit{filigree lemma} of \cite[Lemma 3]{AS} to the anisotropic setting. The main tool for the proof is the replacement theorem proved in \cref{prop:replacement-thm-AS-aniso}.

\begin{lemma}[Anisotropic filigree lemma]\label{lem:aniso-filigree-lemma}
    Suppose $\{Y_t\}_{t\in[0,1]}$ is an increasing collection of strongly convex sets, satisfying the hypothesis of \cref{lem:lemma3-anisotropic-MSY} (with constants $\beta,\theta$ independent of $t$), with $Y_t=\{x \in N: f(x) < t\}$, $t>0$ where $f$ is a non-negative function on some neighborhood $B_{\rho_0}\subset N$ which is $C^2$ on $B_{\rho_0} \setminus \overline{Y}_0$, $Df\not=0$ on $Y_1\setminus \overline{Y}_0$ and $\sup_{\overline{Y}_1\setminus \overline{Y}_0} |Df|\leq c_1$ for some constant $c_1>0$. Assume $M\in\mathcal{M}$ and $\epsilon>0$ are such that  for all $t\in(0,1)$ we have that $\de M \setminus \de Y_t$ is not contained in any $K_{\Sigma_0}$, where ${\Sigma_0}$ is any component of $M\setminus Y_t$ with $\de M \cap {\Sigma_0} = \emptyset$ and $K_{\Sigma_0}$ is as in \cref{lem:lemma3-anisotropic-MSY}, and
    \begin{align}
        \F(M) \leq \F(N) + \epsilon\,, \;\;\; \forall N\in\mathcal{M} \text{ with }\de N = \de M\,.
    \end{align}
    Then
    \begin{align}
        \F(M\cap Y_t)\leq 2\epsilon \mbox{\quad  whenever \quad } t\leq 1- c_1c_{\lambda}\sqrt{\beta}\sqrt{\F(M\cap Y_1)}\,.
    \end{align}
    Here $c_{\lambda}>0$ is a constant depending on $\lambda$.
\end{lemma}
\begin{proof}
    The proof of \cite[Lemma 3]{AS} follows verbatim, replacing the use of \cite[Theorem 1]{AS} with \cref{prop:replacement-thm-AS-aniso}. We point out also that the constants are different depending on $\beta,\lambda>0$.
\end{proof}

\subsection{Integrality of the limit varifold}
In this section we prove an anisotropic analogue of \cite[Theorem 2]{AS}. More precisely in \cref{thm:integrality-limit-varifold} we show that the limit varifold of any sequence of $\F$-minimizing disks, is in fact an $\F$-stationary integral varifold.

\begin{thm}[Integrality of limit varifold]\label{thm:integrality-limit-varifold}
    Suppose $A\subset B_{\rho_0}\subset N$ is open, $A\approx \mathbf{B}$, and $\{M_k\}$ is a sequence of surfaces in $\mathcal{M}$ with $\de M_k \subset \de A$, $(M_k\setminus\de M_k )\subset A$, and
    \begin{align}
        \F(M_k) \leq \F(M) + \epsilon_k,\;\;\; \forall M\in \mathcal{M} \text{ with } \de M = \de M_k\,,
    \end{align}
    where $\lim_{k\to\infty}\epsilon_k\to 0$. Suppose further that the varifold limit $W = \lim_{k\to\infty} \mathbf{v}(M_k)$ exists in $\mathbf{V}_2(N)$. Then $W$ is an $\F$-stationary integral varifold in $A$ with respect to anisotropic area $\delta_{\F}W|_{A}=0$. Moreover $W$ has the property that if $x_0\in \spt\|W\|$ and $W$ has a tangent varifold $C$ at $x_0$ with $\spt\|C\|\subset H$ for some plane $H$, then there is a $\rho>0$ such that
    \begin{align}
        W\lfloor_{G_2(B_{\rho}(x_0))} = n \mathbf{v}(M)\,,
    \end{align}
    where $M$ is a smooth oriented connected surface and $n>0$ is an integer.
\end{thm}
\begin{proof}
     We observe that $W\big|_{G_2(A)}$ is stationary. Indeed, for any open subset $U\subset A$ and any diffeomorphism leaving $\de U$ fixed, by almost minimality we know that $\F(M_k) \leq \F(h^{\#}M_k) + \epsilon_k$. Taking the limit as $k\to\infty$, we deduce that $\F(h^{\#} W|_{U}) \geq \F(W|_U)$.

    We now aim to show that the multiplicity is lower bounded by some universal constant
    \begin{align}
        \Theta_{*}^2(\|W\|,x_1) \geq c\,,
    \end{align}
    for any $x_1\in\spt(\|W\|)$. Let $c_2$ be the isoperimetric constant on $S^2$, then by the filigree \cref{lem:aniso-filigree-lemma}, (with $f(x) = \dist(x_1,x)/\rho$ and $c_1 = \frac{1}{\rho}$) we know that if $\hau^2(M_k\cap B_{\rho}(x_1))\leq \eta\rho^2$, for small enough $\eta$ (possibly depending on $\lambda$) then we have $\hau^2(M_k\cap B_{\rho/2}(x_1))\leq 2\epsilon_k$. Hence if for a subsequence $\{k'\}\subset \{k\}$ we have $\hau^2(M_{k'}\cap B_{\rho}(x_1))\leq \eta\rho^2$, then we would have $\spt(\|W\|)\cap B_{\rho/2}(x_1) = \emptyset$, which is a contradiction with $x_1\in\spt(\|W\|)$. This shows for all sufficiently large $k$ we have $\hau^2(M_k\cap B_{\rho}(x_1)) \geq \eta \rho^2$ which implies:
    \begin{equation}\label{eq:lb-density}
        \|W\|(G_2(B_{\rho}(x_1))) \geq c \rho^2\,.
    \end{equation}
    This concludes the density lower bound, which combined with \cite[Theorem 1.2]{DDG} and \eqref{H:ellipticity} asserts that $\spt(\|W\|)$ is rectifiable.

    In order to prove integrality, take $x_0,C,H$ as in the statement. Moreover let $\phi:B_{\rho_0}(x_0)\to \mathbf{B}_2\subset \R^3$ be a local chart generated by the exponential map as in \cite{MSY} (after a possible rotation in the euclidean target, identifying $H$ with $\R^2\times\{0\}$). Moreover we name $\tilde{W} = \phi^{\#}\left(W\cap G_2(B_{\rho_0}(x_0))\right)$, and we define $K_{s,t}:= D_s\times (-t,t)$. By the multiplicity lower bound, it is easy to see that given any $\sigma_0>0$ small there is a radius $r$ such that 
    \begin{align}
        \boldsymbol\mu_r(\spt\|\tilde W\|) \cap K_{1,1} \subset K_{1,\sigma_0/2}.
    \end{align}
    Here $\boldsymbol{\mu}_r(x)= \frac{x}{r}$. Take such $r$; now our goal is to basically reprove \cite[Theorem 2]{AS} with slights modifications. Define $\tilde{M}_k:=\phi(M_k\cap B_{\rho_0}(x_0))$, we have that
    \begin{align}
        \tilde\F(\boldsymbol\mu_r(\tilde{M}_k)) \leq \tilde\F(M) + \epsilon_k/r^2, \text{ for all } M\in\mathcal{M} \text{ with } \de M = \de \boldsymbol{\mu}_r(\tilde{M}_k)\,.
    \end{align}
    Here $\tilde{\F}$ is the push-forward of the anisotropic energy $\F$ via $\phi$:
    \begin{align}
        \tilde{F}\left(\phi(x),D\phi(x)[T]\right) = F(x,T)\,.
    \end{align}
    Note that $\tilde F$ satisfies  \eqref{H:ellipticity} with $\lambda/2$ in place of $\lambda$, as $\phi$ is $C^2$-close to the identity. Since $\tilde{M}_k$ converges to $\phi^{\#}W$ in the sense of varifolds, by the co-area formula we have that for almost all $\sigma\in(\sigma_0,1)$:
    \begin{align}
        \hau^{1}(\boldsymbol{\mu}_r(\tilde{M}_k) \cap (D_1\times\{-\sigma,\sigma\})) \to 0\,.
    \end{align}
    Hence for any small $\eta>0$, for large enough $k$ we can find $\sigma_k\in(\frac34\sigma_0,\sigma_0)$:
    \begin{align}
        \hau^{1}(\boldsymbol{\mu}_r(\tilde{M}_k) \cap (D_1\times\{-\sigma_k,\sigma_k\})) < \eta\,.
    \end{align}
    This means that, we can find $\rho_k \in (\frac34,1)$ such that:
    \begin{align}
        \boldsymbol{\mu}_r(\tilde{M}_k) \cap (\de D_{\rho_k}\times\{-\sigma_k,\sigma_k\}) = \emptyset\,.
    \end{align}
    Following the assertion above, we apply the replacement \cref{prop:replacement-thm-AS-aniso} to $\boldsymbol\mu_r(\tilde{M}_k)$ with $U=K_{1,\sigma_k}$ to find a family of disks $P_{k}^{1},\cdots,P_{k}^{R_k'},P_{k}^{R_k'+1},\cdots,P_{k}^{R_k}\subset \mathcal{M}$ for $R_k,R_k'$ integers depending on $k$. Moreover\footnote{We point out the typo in the equation right before \cite[(5.11)]{AS}, the definitions are swapped.}
    \begin{align}
        \de P_{k}^{1},\cdots,\de P_{k}^{R_k'} &\subset \de D_{\rho_k}\times(-\sigma_k,\sigma_k)\,,\\
        \de P_{k}^{R_k'+1},\cdots,\de P_{k}^{R_k} &\subset D_{\rho_k}\times\{-\sigma_k,\sigma_k\}\,,
    \end{align}
    such that:
    \begin{align}
        (1-C(\lambda)r)\tilde\F_r(P_k^{i}) \leq \tilde\F_r(P) + \epsilon_{k,i}, \forall P\in\mathcal{M}\text{ with } \de P = \de P_{k}^i\,,
    \end{align}
    with $\sum_{i}\epsilon_{k,i} \leq \epsilon_k/r^2$. Here $\tilde{F}_r(x,T) = \tilde{F}(rx,T)$ is the push-forward of $\tilde\F$ under zooming. By the virtue of (5),(7) in \cref{prop:replacement-thm-AS-aniso}, we have
    \begin{align}
        \boldsymbol{\mu}_r^{\#}(\tilde{W})\lfloor_{G_2(K_{1,1/2})} = \lim_{k\to\infty} \sum_{i=1}^{R_k} \mathbf{v}(P_k^i)\,.
    \end{align}
    Similar to \cite{AS}, we know that $P_{k}^{R_k'+1},\cdots, P_{k}^{R_k}$ can be discarded without changing the varifold limit. This is exactly the same argument, while replacing the use of filigree lemma by \cref{lem:aniso-filigree-lemma}. With the same exact argument, we can also discard elements of $P_{k}^{1},\cdots, P_{k}^{R_k'}$ whose boundary $\de P_{k}^{i}$ is null-homotopic in $\de D_{\rho_k}\times(-\sigma_k,\sigma_k)$, hence:
    \begin{align}
        \boldsymbol{\mu}_r^{\#}(\tilde{W})\lfloor_{G_2(K_{1,1/2})} = \lim_{k\to\infty} \sum_{i\in \Im_k} \mathbf{v}(P_k^i)\,,
    \end{align}
    where $\Im_k\subset \{1,\cdots,R_k'\}$ is the collection of those $P_k^i$s, $i\in\{1,\cdots,R_k'\}$, which $\de P_k^i$ is not null-homotopic in $\de D_{\rho_k}\times(-\sigma_k,\sigma_k)$. Here we depart slightly from the proof of \cite{AS}. Note that the uniform convexity assumption \eqref{H:ellipticity} means that:
    \begin{align}\label{updated-uniform-convexity}
    \tilde{G}_r(x,\nu) - \frac\lambda4|\overline{\nu}-\nu|^2 \geq \ang{D_\nu \tilde{G}_r(x,\overline{\nu}),\nu}\,.
    \end{align}
    Here $\tilde{G}_r(x,\lambda \nu) = |\lambda|\tilde{F}_r(x,\nu)$.
    
    Now fix $i\in\Im_k$ and take $A_k^i\subset \de D_{\rho_k}\times(-\sigma_k,\sigma_k)$ to be the component of $\de D_{\rho_k}\times [-\sigma_k,\sigma_k] \setminus P_{k}^i$ containing $\de D_{\rho_k}\times \{\sigma_k\}$, and define $\Omega_k^i$ to be the component of $K_{\rho_k,\sigma_k}\setminus P_k^i$ such that $\de\Omega_k^i = A_k^i \cup (D_{\rho_k}\times\{\sigma_k\})\cup P_i^k$. Then take the vector field $\eta = (0,0,-1)$, extension of the outward normal to $D_{\rho_k}\times\{\sigma_k\}$ and apply the divergence theorem to $D_\nu \tilde{G}_r(x,\eta)$, we see that:
    \begin{align}
    \begin{aligned}\label{pre-cacciopoli}
        -c\lambda\sigma_k\rho_k &\leq \int_{\Omega_k^i} \div(D_\nu \tilde{G}_r(x,\eta)) \\ &=  -\tilde{\F}_r(D_{\rho_k}\times\{\sigma_k\}) + \int_{A_k^i} \ang{D_\nu \tilde{G}_r(x,\eta),\nu} + \int_{P_k^i} \ang{D_\nu \tilde{G}_r(x,\eta),\nu}\\
        &\leq \tilde\F_r(P_k^i) -\tilde{\F}_r(D_{\rho_k}\times\{\sigma_k\}) -c_\lambda \int_{P_{k}^i}|\eta-\nu|^2 \,\mathrm{d}\hau^2+ c_\lambda\sigma_k\rho_k\,.
        \end{aligned}
    \end{align}
    Since $D_{\rho_k}\times\{\sigma_k\}\cup A_k^i$ is a competitor for $P_k^i$, we have
    \begin{align}\label{comparison-1}
        \tilde\F_r(P_k^i) \leq \tilde\F_r(D_{\rho_k}\times\{\sigma_k\}) + \tilde\F_r(A_k^i) + \epsilon_{k,i} + C(\lambda)r\,.
    \end{align}
    Using the above and \cref{pre-cacciopoli}, we arrive at
    \begin{align}\label{excess-bound}
        \int_{P_{k}^i}|\eta-\nu|^2 \,\mathrm{d}\hau^2 \leq \epsilon_{k,i} + c_\lambda \sigma_0 + C(\lambda) r\,.
    \end{align}
    This is basically an excess bound. First note that since for all $i\in\Im_k$, all $\de P_k^i$ are not null-homotopic in $\de D_{\rho_0}\times [-\sigma_k,\sigma_k]$, we have:
    \begin{align}\label{intergral-lower-bound}
        \hau^2(P_k^i\cap K_{\rho,1}) \geq \pi\rho^2,\;\;\;\text{ for all }\rho\in(0,\rho_k]\,.
    \end{align}
    This serves as the multiplicity lower bound. Moreover in the view of this lower bound and the multiplicity upper bound we get that $\sup_{k}|\Im_k| < \infty$ and we can take a further subsequence (not relabeled) such that $|\Im_k|$ is constant.
    
    Now, to get the upper bound, we realize that
    \begin{align}
        \begin{aligned}
        \left|\tilde{\F}_r(P_k^i) - \tilde{F}_r(0,\eta)\hau^2(P_k^i)\right| &\leq \int_{P_k^i} |\tilde{F}_r(x,\nu_{P_k^i}) - \tilde{F}_r(0,\eta)| \,\mathrm{d}\hau^2 \\
        &\leq c_\lambda \int_{P_k^i}|\nu-\eta|\,\mathrm{d}\hau^2 + \omega_{F}(r,x_0)\hau^2(P_k^i)\\
        &\leq c_\lambda \left(\sqrt{\epsilon_{k,i} + \sigma_0} + \omega_F(r,x_0)\right)\,.
        \end{aligned}
    \end{align}
    Here $\omega_F(r,x_0)$ is the modulus of continuity of $D_xF$ on $x_0$ at radius $r$. Since $F\in C^1$ by \eqref{H:regularity}, we know that $\lim_{r\to0} \omega_F(r,x_0) = 0$. With a similar argument we also assert that:
    \begin{align}
        \left|\tilde\F_r(D_{\rho_k}\times\{\sigma_k\}) - \tilde{F}_r(x_0,\eta)\pi\rho_k^2\right| \leq c_\lambda\omega_F(r,x_0)\,.
    \end{align}
    Combining the last two displays with \cref{comparison-1} we get that:
    \begin{align}\label{integral-upper-bound}
        \hau^2(P_k^i) \leq \pi\rho_k^2 + c_\lambda(\sigma_0 + \omega_F(r,x_0))
    \end{align}
    for large enough $k$. Then take $n=|\Im_k|$ and a further subsequence (if necessary) such that $\rho_k\to\rho_0\in[\frac34,1]$ for $i=1,\cdots,n$:
    \begin{align}
        \mathbf{v}(\boldsymbol{\mu}_{\rho_k}(P_k^i)) \to \tilde{W}_i\,,
    \end{align}
    with $\boldsymbol\mu_{\rho_0}^{\#}\tilde{W}\lfloor_{K_{1/2,1}} = \sum_{i=1}^n \tilde{W}_i\lfloor_{K_{1/2,1}}$ and from \cref{intergral-lower-bound,integral-upper-bound} we get that:
    \begin{align}
        \forall \rho\in(0,1):\;\|\tilde{W}_i\|(K_{\rho,1})\geq \pi \rho^2,\;\; \|\tilde{W}_i\|(K_{1,1}) \leq \pi + c_\lambda(\sigma_0+\omega_F(r,x_0))\,.
    \end{align}
    Moreover we have that $\spt\|\tilde{W}_i\|\subset K_{1,\sigma_0}$, hence:
    \begin{align}
        \|\tilde{W}_i\|(\B_1) \geq \|\tilde{W}_i\|(K_{1-\sigma_0,\sigma_0}) \geq \pi(1-\sigma_0)^2\,.
    \end{align}
    and
    \begin{align}
        \|\tilde{W}_i\|(\B_1) \leq \|\tilde{W}_i\|(K_{1,1}) \leq \pi + c_\lambda(\sigma_0+\omega_F(r,x_0))  
    \end{align}
    Since $r,\sigma_0$ is arbitrary and $\lim_{r\to0}\omega_F(r,x_0)=0$ since $F\in C^1$, we get that $\tilde{W}_i$ is integral, hence $\tilde{W}$ and consequently $W$ is integral. The $C^{1,\alpha}$ regularity of each $\tilde{W}_i$ then follows from the classical regularity theorem of Allard in \cite[Section 3 \& 3.6]{A-constancy} (since their multiplicity is close to one and they they are limits of almost minimizing disks, hence they are individually stationary). Their higher regularity follows by Schauder estimates and standard PDE arguments. Take $M^i$ to be such that $\tilde{W}_i = \mathbf{v}(M^i)$. Now, similar to the arguments of the paragraph right after \cite[(6.8)]{AS}, since $P_{k}^i$s are disjoint (for fixed $k$), the limit surfaces $M^i$s cannot be transversal, hence if they touch, by the maximum principle of \cite{SW-maximum-principle}, they must be the same. This concludes the proof.
\end{proof}

\subsection{Proof of \cref{thm:white-improvement}: regularity à la Pitts}\label{sec:regularity}

In this section we prove the anisotropic analogues of the regularity results of \cite[Section 5 \& 6]{AS}. However, rather than relying on the methods of \cite{AS}, we use Pitts' regularity scheme. First we use \cite{W-1} combined with \cite{HS} to show that this limit varifold, which is integral by \cref{thm:integrality-limit-varifold}, admits replacements in all annuli small enough, centered at both interior and boundary points. Differently from what happens in the min-max regularity theory, \emph{the smallness of the annuli is independent of their center point}. This is crucial, as it allows us to conclude that any limit varifold is in fact a smooth $\F$-stationary and $\F$-stable surface, bypassing the issue of removal of isolated singularities. 

\begin{proposition}[Replacements in annuli]\label{prop:replacements-in-annuli}
    Let $B_{\rho} \subset N$ be an open ball with $\rho \leq \rho_1$ for small enough $\rho_1(N)>0$ and $\{M_k\}\subset \mathcal{M}$ a sequence of surfaces (topological disks) with $\de M_k\subset \de B_{\rho}$ and $M_k\setminus \de M_k \subset B_\rho$ and for every $M\in\mathcal{M}$ with $\de M = \de M_k$ we have:
    \begin{align}
        \F(M_k) \leq \F(M) + \epsilon_k,\quad\text{with } \lim_{k\to\infty}\epsilon_k = 0\,,
    \end{align}
    Suppose further that the varifold limit $W = \lim_{k\to\infty} \mathbf{v}(M_k)$ exists in $\mathbf{V}_2(N)$. Then there exists a modified sequence $M_k'$ with the following properties:
    \begin{enumerate}
        \item We have $\de M_k' = \de M_k$.
        \item $\{M_k'\}$ is another minimizing sequence, meaning for every $M\in \mathcal{M}$ with $\de M = \de M_k'$:
        \begin{align}
            \F(M_k') \leq \F(M) + \epsilon_k\,.
        \end{align}
        \item The modified sequence has a varifold limit $W' = \lim_{k\to\infty} \mathbf{v}(M_k')$ (up to a subsequence) such that:
        \begin{align}
            W'|_{B_{\rho/2}} = W|_{B_{\rho/2}}\,.
        \end{align}
        \item We have that $W'|_{B_{\rho}\setminus B_{\rho/2}} = \sum_{i}\mathbf{v}(\Sigma_i)$, where $\Sigma_i$ is a smooth $\F$-stationary and $\F$-stable surface in any compact set $K\subset B_{\rho}\setminus \overline{B_{\rho/2}}$ with uniform curvature estimates depending only on $K,\F,N$.
        
        \item If $\limsup_{k\to\infty}\|\de M_k\|_{C^{2,\alpha}} <\infty$, then $W'|_{B_{\rho}\setminus B_{\rho/2}} = \sum_{i}\mathbf{v}(\Sigma_i)$, where $\Sigma_i$s are separately regular up to the outer boundary and uniform curvature estimates up to the outer boundary $\de B_\rho$ depending on $K,\F,N,\limsup_{k\to\infty}\|\de M_k\|_{C^{2,\alpha}}$.
    \end{enumerate}
\end{proposition}
\begin{rmk}
    Note that the interior curvature estimates do not depend on the mass or the multiplicity of the limiting varifold.
\end{rmk}
\begin{proof}
    We take inspiration from the proof of \cite[Theorem 4.1]{HS} and the tools developed in \cite{W-1} to build a replacement in the annulus.

    We cover the annulus with a countable collection of smooth open convex sets $\{Q_i\}_{i=1}^\infty$ with the following properties:
    \begin{itemize}
        \item $Q_i\subset B_{\rho}\setminus \overline{B_{\rho/2}}$ and $B_\rho \setminus \overline{B_{\rho/2}} = \bigcup_{i=1}^\infty Q_i$.
        \item $\diam(Q_i) \sim \dist(Q_i,B_{\rho/2})$ and for any two intersecting sets $Q_i\cap Q_j\not = \emptyset$ we have that $\frac 14\leq \frac{\diam(Q_i)}{\diam(Q_j)} \leq 4$, meaning that $Q_i$ is a Whitney covering with respect to the inner boundary $\de B_{\rho/2}$.
        \item For some universal number $n_0$, we have that $\de Q_i\cap \de B_\rho \not = \emptyset$ if and only if $1\leq i \leq n_0$. More importantly we also have $\de B_{\rho} \subset \bigcup_{i=1}^{n_0}\de Q_i$.
        \item Up to a slight enlarging the sets, we also assume that $M_k$ is transverse to each $\de Q_j$.
        \item The sets $\{Q_j\}$ and the ball $B_\rho$ are small and convex enough such that the conditions of \cref{lemma:tentacles-estimate} apply.
    \end{itemize}
    Now take $Q_1$ and consider the intersection $\de Q_1\cap M_k$. Moreover take the map diffeomorphism $\chi: D_1 \to M_k$ with $\cup_{j}\chi(\gamma_j) = \de Q_1\cap M_k$ where $\gamma_j\subset D_1$ is a simple closed curve. Take the connected components of $\cup_{j=1}^{n_1} \mathrm{int}(\gamma_j)$ to be $\Omega_i$ and take the outermost curves of each component to be $\gamma_i$ for $i\in I$. Now by the result of \cite{W-1}, there exists a surface $\Sigma_i\subset Q_1$ that minimizes the $\F$-area with the boundary $\chi(\gamma_i)$. Now we replace $M_k$ with $M_k^{(1)}$ defined as follows:
    \begin{align}
        M_k^{(1)} = \left(M_k \setminus \bigcup_{i\in I} \chi(\Omega_i)\right) \cup \bigcup_{i\in I }\Sigma_i\,.
    \end{align}
    This tells us that:
    \begin{align}
        \F(M_k^{(1)}) \leq \F(M_k)\,,
    \end{align}
    and by construction:
    \begin{align}
        \de M_k = \de M_k^{(1)}\,.
    \end{align}
    Hence $M_k^{(1)}$ is $\epsilon_k$-minimizing with respect to the original boundary. 

    We now apply the compactness \cref{lemma:disk-compactness} choosing \(A=Q_1\). Namely,
    \[
        M_k^{(1)}\cap Q_1=
    \bigcup_{i\in I}\Sigma_i,
    \]
    where each \(\Sigma_i\subset Q_1\) is a least \(\F\)-area disk with boundary \(\chi(\gamma_i)\subset\partial Q_1\). Thus \(M_k^{(1)}\cap Q_1\) is a finite collection of least \(\F\)-area disks. Moreover, by \eqref{H:comparability} and by the almost-minimizing property of \(M_k\),
    \[
        \begin{aligned}
        \mathcal H^2(M_k^{(1)}\cap Q_1)
        &\le
        C(\lambda)\F(M_k^{(1)}\cap Q_1)        \le
        C(\lambda)\F(M_k^{(1)})                \le
        C(\lambda)\F(M_k)                       \\&\le C(\lambda)\bigl(\F(S_k)+\epsilon_k\bigr)\le C,
        \end{aligned}
    \]
    where \(S_k\subset\partial B_\rho\) is a disk in \(\partial B_\rho\) with \(\partial S_k=\partial M_k\). Therefore, \cref{lemma:disk-compactness}(2) gives a subsequence such that \(M_k^{(1)}\cap Q_1\) converges smoothly on compact subsets of \(Q_1\) to a smooth \(\F\)-minimal surface.
    
    Now take ${\Sigma_0}\subset M_k\setminus Q_1$ to be the component (not necessarily connected) of $M_k$ such that ${\Sigma_0} \cap \de M_{k} = \emptyset$. Then as a direct consequence of \cref{lemma:tentacles-estimate} and by the definition and convexity of $Q_1$ we have that $(M_{k}\Delta M_k^{(1)})\cap B_{\rho/2} \subset {\Sigma_0}$. Since $\dist(Q_1,B_{\rho/2}) \sim \diam(Q_j)$, then the tentacles estimate \cref{lemma:tentacles-estimate} implies:
   \begin{align}
        \hau^2({\Sigma_0}\cap B_{\rho/2}) \leq C\epsilon_k\,.
   \end{align}
    This means that the varifold limit of $\{M_k^{(1)}\}$ up to a subsequence (not relabeled) $W^1 = \lim_{k\to \infty} \textbf{v}(M_k^{(1)})$ is identical to $W$ in $B_{\rho/2}$ and smooth in $Q_1$.
    
    From this we build the sequence $\{M_k^{(2)}\}$ by repeating the same argument for $Q_2$ and taking a further subsequence with a smooth limit in the interior of $Q_2$. Moreover with the same argument and the use of \cref{lemma:tentacles-estimate} (choosing now $A=Q_2$), we can also ensure that any varifold limit $M_k^{(2)}$, namely $W_2$ is again identical to $W$ in the interior of $B_{\rho/2}$ and smooth in the interior of $Q_2$ and $Q_1\setminus Q_2$. What remains, is to show that in the overlap region $\de Q_2\cap Q_1$, no bending occurs in the limit and the limit $W_2$ is smooth inside $Q_1\cup Q_2$. For this we use the strategy of \cite[Theorem 4.1]{HS}, almost verbatim, but with different tools suited to the anisotropic setting.
    
    First if $Q_1\cap Q_2 = \emptyset$ we have nothing to prove, which leaves $Q_1\cap Q_2 \not= \emptyset$. Now take $\lim_{k\to\infty} M_k^{(2)}\cap Q_2 = T_2$ and $\lim_{k\to\infty} M_k^{(2)} \cap(Q_1\setminus \overline{Q_2}) = T_1$, we aim to show that $T_1\cup T_2$ is smooth in the interior of $B_\rho\setminus \overline{B_{\rho/2}}$.
    
    As the convergence of $M_k^{(1)}$ was smooth inside $Q_1$ and $M_k^{(2)}$ is made up of a sub-collection of components of $M_k^{(1)}$, we assert that $T_1$ meets $\de Q_2\cap Q_1$ is a (possibly collection of) smooth arc(s). If this collection is empty, then it is clear that $T_1\cup T_2$ is smooth, so we consider a component $E$ of $T_1$ such that this component intersects $\de Q_2 \cap Q_1$ and we consider an arc $\Gamma$ in this intersection $E\cap \de Q_2\cap Q_1$. By construction, it is a limit of a series of arcs $\Gamma_k$ as the intersection of (a sub-collection of) least $\F$-area disks in $M_k^{(1)}\cap Q_1$ with $\de Q_2 \cap Q_1$, and they are also part of the boundary of a least $\F$-area disk $F_k$ in $M_k^{(2)}\cap Q_2$. Now let $A_k = \de F_k \setminus \Gamma_k$ and take a point $x\in \Gamma \setminus \de B_\rho$ away from the boundary of the original annulus with $\dist(x,\de B_\rho) > 0$.
    
    We will show that $T_1\cup T_2$ is a smooth surface in the neighborhood of $x$. First since $\Gamma_k$ was inside the intersection of $\de Q_2$ with $M_k^{(1)}\cap Q_1$ and since $M_k^{(1)}$ was smooth in the interior, we can assert that $\Gamma_k$ has uniform $C^k$ estimate in its interior, which passes to the limit $\Gamma$. Like in \cite{HS} first, we show that in a neighborhood of $x$, $T_1\cup T_2$ is a continuous surface, smooth except for possibly a bend along $\Gamma$. Then we would like to show that no bending can occur.

    \textit{Case 1.} There is a constant $\epsilon>0$ such that $B_\epsilon(x) \cap A_k=\emptyset$ for all large enough $k>0$, then the fourth item in \cref{lemma:disk-compactness} guarantees that $F_k$ converges in $Q_2$ to a surface which extends continuously up to the boundary $\de Q_2$.

    \textit{Case 2.} The sequence $A_k$ intersects any neighborhood of $x$ (meaning that some folding happens on the boundary which disappears in the limit).

    Let $C_k$ be a component of $A_k\cap Q_1$ which contains a point $y_k$ where $\lim_{k\to\infty}y_k = x$. Since $C_k$ and $\Gamma_k$ are both components of $M_k^{(2)}\cap Q_1\cap \de Q_2$ and $M_k^{(2)}$ is embedded, hence both must converge to $\Gamma$. This means that $\Gamma$ is at least with multiplicity two and since $E$ is part of the interior of $Q_2$, then $M_k^{(2)}$ must've converged with multiplicity at least two to $E$ in $Q_1\setminus Q_2$. In this case we have to consider how the disks $F_k$ converge in $Q_2$. It is clear that $F_k$ intersects $B_\epsilon(x)$ in disjoint disks.

    \textit{Subcase 2(i).} There is an $\epsilon$ such that $F_k$ only intersects $\de Q_2\cap Q_1$ in $\Gamma_k$, which in this case, we follow the strategy of Case 1.

    \textit{Subcase 2(ii).} There is a sequence of points $x_k$ on $\Gamma_k$ with $\lim_{k\to\infty}x_k = x$ and a sequence of point $y_k$ on $C_k$ also converging to $x$ such that there exist a sequence of paths $d_k$ connecting $x_k$ to $y_k$ on $F_k$ with $\lim_{k\to\infty} \hau^1(d_k) = 0$. To show that this is impossible, we use the same exact cut and paste argument as in \cite[Theorem 4.1]{HS}. Note that because of \eqref{H:comparability} in the anisotropic case the $\F$-area of \cite[Figure 4c]{HS} is still greater than the $\F$-area of \cite[Figure 4b]{HS} for small enough $\hau^1(d_k)$. The rest of the argument follows exactly as in \cite{HS}, which we will omit.

    Since the argument above can be applied to any point $x\in T_1\cap \de Q_2$ with $\dist(x,\de B_\rho)>0$, we conclude that $T_1\cup T_2$ is continuous and smooth in the interior of $B_{\rho}\setminus B_{\rho/2}$ except for a possible bend along the arc $T_1\cap \de Q_2$. Suppose that $T_1\cup T_2$ fails to be smooth along an arc $\Gamma$ of $T_1\cap \de Q_2$ at a point $x\in \Gamma \setminus \de B_{\rho}$. Choose a small ball $B\subset B_{\rho}\setminus \overline{B_{\rho/2}}$ centered at $x$ such that the component of $T_1\cup T_2$ which contains $\Gamma$ intersects $B$ in a 2-disk $H$. Since $H$ has a bend along $\Gamma$ and $\F$ is a convex integrand, for small enough $B$ depending on $\lambda$, $H$ cannot be a $\F$-minimizing disk, hence there exists another 2-disk $H'$ with the same boundary as $H$ in $B$ with less area. We can use this as a competitor in the sequence exactly as in \cite[Theorem 4.1]{HS} the last paragraph of page 104 and reach a contradiction. This means that $T_1\cup T_2$ is smooth in the interior of $B_{\rho}\setminus \overline{B_{\rho/2}}$.

    Now in regarding the (5)th conclusion, if we have uniform $C^{2,\alpha}$ estimates on the boundary of $M_k$, we can use the third conclusion of \cref{lemma:disk-compactness} (as an analogue to \cite[Lemma 6.9]{HS}) and deduce similarly that for those $Q_k$s where the intersect the boundary $\de B_\rho$, the limit near the boundary will be made up of component that are regular up to the boundary with uniform curvature estimates depending on $\F,N$ and the regularity of the boundary data.

    Here we obtain a varifold $W_2$ such that $W_2|_{B_{\rho/2}} = W$ which is also smooth inside $Q_1\cup Q_2$. We may repeat this process and obtain a sequence of stationary integer rectifiable varifolds $W_k$ with $W_k|_{B_{\rho/2}} = W|_{B_{\rho/2}}$, such that $W_k$ is smooth in $\bigcup_{j=1}^{k}Q_k$ with uniform curvature estimates. Taking a limit of these varifolds (up to a subsequence) we can find a varifold $W'$ which is smooth in $B_{\rho}\setminus \overline{B_{\rho/2}}$. Moreover in the case the assumptions of (5) hold, the $W'$ is smooth in $B_\rho\setminus \overline{B_{\rho/2}}$ up to the outer boundary $\de B_\rho$. Then using a diagonal argument and a slight smoothing we can choose a sequence $M_{\pi(k)}^{(k)}\in \mathcal{M}$ which is $\epsilon_{\pi(k)}$ minimizing with the limit being $W'$.
\end{proof}

Almost identical to \cref{prop:replacements-in-annuli}, we also state and sketch the proof of the boundary version. First we define smoothed balls with centers on the boundary as follows. Imagine a point $x\in\de A$ and consider the half ball $B_\rho(x)\cap A$ with sharp edges. Then one can smoothen the edges slightly to produce $\tilde{B}_\rho(x)$ with smooth and $\F$-convex boundary, if the boundary of $A$ is also $\F$-convex and $\rho$ is chosen small enough. This is a technical modification, which makes boundary regularity issues simpler to deal with.

\begin{proposition}[Replacement in boundary annuli]\label{prop:replacements-in-annuli-boundary}
    Take a set $A\subset N$ with $\F$-convex boundary and take $\tilde{B}_\rho(x)$ defined as above with $x\in\de A$ with $\rho \leq \rho_1$ and take a sequence $\{M_k\}\subset\mathcal{M}$ as in \cref{prop:replacements-in-annuli} with $\de M_k\subset \tilde{B}_{\rho}$. Then there exists a modified sequence $M_k'$ with the following properties:
    \begin{enumerate}
        \item We have $\de M_k' = \de M_k$.
        \item $\{M_k'\}$ is another minimizing sequence, meaning for every $M\in \mathcal{M}$ with $\de M = \de M_k'$:
        \begin{align}
            \F(M_k') \leq \F(M) + \epsilon_k\,.
        \end{align}
        \item The modified sequence has a varifold limit $W' = \lim_{k\to\infty} \mathbf{v}(M_k')$ (up to a subsequence) such that:
        \begin{align}
            W'|_{B_{\rho/2}(x)\cap A} = W|_{B_{\rho/2}(x)\cap A}\,.
        \end{align}
        \item We have that $W'|_{(\tilde{B}_{\rho}\setminus B_{\rho/2})\cap A} = \sum_{i}\mathbf{v}(\Sigma_i)$, where $\Sigma_i$ is a smooth $\F$-stationary and $\F$-stable surface in any compact set $K\subset (\tilde{B}_{\rho}\setminus \overline{B_{\rho/2}})\cap A$ with uniform curvature estimates depending only on $K,\F,N$.
        
        \item If $\limsup_{k\to\infty}\|\de M_k\cap \de A \|_{C^{2,\alpha}} <\infty$, then $W'|_{(\tilde{B}_{\rho}\setminus B_{\rho/2})\cap A} = \sum_{i}\mathbf{v}(\Sigma_i)$, where $\Sigma_i$s are separately regular up to the outer boundary $\de A \cap \tilde{B}_{\rho}\setminus \overline{B_{\rho/2}}$ with uniform curvature estimates depending on $K,\F,N,\limsup_{k\to\infty}\|\de M_k\cap\de A\|_{C^{2,\alpha}}$.

        \item If moreover $\limsup_{k\to\infty}\|\de M_k\cap \de A \|_{C^{2,\alpha}} + \|\de M_k\cap \de \tilde{B}_\rho\cap A \|_{C^{2,\alpha}} <\infty$ then the curvature estimates are uniform up to $\de(\tilde{B}_\rho\cap A)\setminus \overline{B_{\rho/2}}$.
    \end{enumerate}
\end{proposition}
\begin{proof}
    The proof goes almost identically as in \cref{prop:replacements-in-annuli}, with the Whitney covering done in $A\cap B_{\rho}\setminus \overline{B_{\rho/2}}$ with respect to the inner boundary $\de B_{\rho/2}\cap A$. The boundary uniform estimates comes from applying (3) in \cref{lemma:disk-compactness} in each ball and then gluing exactly as in \cref{prop:replacements-in-annuli} inspired from \cite[Theorem 4.1]{HS}.
\end{proof}

We are finally able to prove \cref{thm:white-improvement}, by using Pitts' regularity theory for min-max constructions.

\begin{proof}[Proof of \cref{thm:white-improvement}]
    (1) Let $\{M_k\}\subset \mathcal M$ be a sequence of smooth embedded disks such that
    \[
        \partial M_k\subset \partial \mathbf B,\qquad\mathbf F(M_k)\leq \mathbf F(P)+\varepsilon_k
    \]
    for every $P\in\mathcal M$ with $\partial P=\partial M_k$, where $\varepsilon_k\to 0$ and in view of \cref{rmk:3-12AS} we have $M_k\subset \textbf{B}\cup \de M_k$. Hence,  the $\mathbf F$-energies of $M_k$ are uniformly bounded. By \eqref{H:comparability}, the $\mathcal H^2$--areas of $M_k$ are uniformly bounded as well. Hence, after passing to a subsequence, we know that there exists a varifold $W\in \mathbf V_2(\mathbf{B})$ such that
    \[
        \lim_{k\to\infty}\mathbf v(M_k)=W.
    \]
    By \cref{thm:integrality-limit-varifold}, the varifold $W$ is integral and $\mathbf F$--stationary in $\mathbf B$.

    Fix now an open set $A$, satisfying $ K\subset A\subset\subset \mathbf B$. Then, it suffices to show that every point of $\spt\|W\|\cap A$ is a regular point. Indeed, once this is shown, it follows that the tangent varifold at any point is a subset of a plane, hence by \cref{thm:integrality-limit-varifold} $\spt\|W\|\cap A$ is a smooth embedded $\mathbf F$-stationary surface, say $\Sigma$, and since $W$ is integral there exists an integer-valued multiplicity function $\theta$ on $\Sigma$ such that 
    \[
        W\llcorner A=\theta\,\mathbf v(\Sigma)\llcorner A.
    \]
    Restricting to $K$ then gives the conclusion.

    Fix $x\in \spt\|W\|\cap A$ and set
    \[
        \rho=\rho(x):=\frac18\min\{\rho_1,\dist(x,\mathbb R^3\setminus A)\},
    \]
    where $\rho_1$ is given in \cref{prop:replacements-in-annuli}. Then
    \[
        \overline{B_{4\rho}(x)}\subset A\quad\text{and}\quad 4\rho<\rho_1.
    \]
    Without loss of generality, we can shrink $\rho_1$ if necessary such that every sphere $\partial B_r(x)$ with $0<r\le 2\rho$ is strictly $\mathbf F$-mean convex. Denote the annulus as 
    \[
        \operatorname{An}(x;s,t):=B_t(x)\setminus \overline{B_s(x)}.
    \]

    In the following, we will first construct a smooth embedded $\mathbf F$-stationary surface 
    \[
        \Sigma_x\subset B_{2\rho}(x)\setminus\{x\}
    \]
    such that
    \[
        \spt\|W\|\cap \bigl(B_\rho(x)\setminus\{x\}\bigr)\subset \Sigma_x,
    \]
    and then remove the puncture by overlapping punctured balls. To make the proof clearly, we divide it into four steps.

    \medskip
    \noindent\textbf{Step 1: set up.}
    
    Choose a smooth strongly convex domain $U_x\subset\subset B_{4\rho}(x)$ such that $\overline{B_{2\rho}(x)}\subset U_x$, and such that the hypotheses of \cref{prop:replacement-thm-AS-aniso} are satisfied in $U_x$ (after choosing $U_x$ sufficiently small and with boundary curvature in the regime needed there, so that the multiplier can be removed). Choosing $U_x$ generically and passing to a subsequence if necessary, we may assume that $M_k$ intersects $\partial U_x$ transversely for every $k$.

    Since $U_x\subset\subset \mathbf B$, we know that every admissible modification of $M_k$ inside $U_x$, keeping the outside part fixed, gives a global competitor with the same boundary $\partial M_k$. Hence, the sequence $\{M_k\}$ is almost minimizing in $U_x$ in the sense required by \cref{prop:replacement-thm-AS-aniso}. Applying that \cref{prop:replacement-thm-AS-aniso} in $U_x$, we obtain a modified sequence $\{\tilde M_k\}$ such that
    $$
        \partial \tilde M_k=\partial M_k,\qquad\tilde M_k\setminus U_x\subset M_k\setminus U_x,
    $$
    and
    \[
        \tilde M_k\cap \overline{U_x}=\bigcup_{j=1}^{\ell_k} D_{k,j},
    \]
    where the $D_{k,j}$ are pairwise disjoint embedded disks. Moreover, each $D_{k,j}$ is itself almost minimizing among embedded disks with the same boundary, and the total error tends to zero as $k\to\infty$.

    Now fix any annulus
    \[
        \operatorname{An}(x;s,t)\subset B_{2\rho}(x),\qquad 0<s<t\leq 2\rho.
    \]
    Since $\operatorname{An}(x;s,t)\subset\subset U_x$, the stacked disks above provide precisely the local input used in the proof of \cref{prop:replacements-in-annuli}. Therefore, arguing exactly as in \cref{prop:replacements-in-annuli}, we obtain a replacement of $W$ in $\operatorname{An}(x;s,t)$. 
    
    Hence we know that there exists a replacement of $W$ in $\operatorname{An}(x;\rho,2\rho)$, denoting by $W'$. Write
    \[
        W'\llcorner G_2(\operatorname{An}(x;\rho,2\rho))=\sum_{\alpha=1}^{N'} m'_\alpha\,\mathbf v(\Sigma'_\alpha),
    \]
    where the $\Sigma'_\alpha$ are pairwise disjoint connected smooth embedded $\mathbf F$--stationary and $\F$-stable surfaces. Set
    \[
        S':=\bigcup_{\alpha=1}^{N'}\Sigma'_\alpha.
    \]
    Choose $t\in(\rho,2\rho)$ so that every $\Sigma'_\alpha$ intersects $\partial B_t(x)$ transversely, and set
    \[
        \gamma:=S'\cap \partial B_t(x).
    \]
    Then $\gamma$ is a finite disjoint union of smooth embedded closed curves.

    For any $s\in(0,\rho)$, choosing a replacement $W''$ of $W'$ in the annulus $\operatorname{An}(x;s,t)$. Write
    \[
        W''\llcorner G_2(\operatorname{An}(x;s,t))=\sum_{\beta=1}^{N''} m''_\beta\,\mathbf v(\Sigma''_\beta),
    \]
    where the $\Sigma''_\beta\subset \operatorname{An}(x;s,t)$ are pairwise disjoint connected smooth embedded $\mathbf F$--stationary and $\F$-stable surfaces, smooth up to the outer sphere $\partial B_t(x)$. Set
    \[
        S'':=\bigcup_{\beta=1}^{N''}\Sigma''_\beta.
    \]

    \medskip
    \noindent\textbf{Step 2: smooth gluing across $\partial B_t(x)$.}

    Fix a point $y\in\gamma$. Since $S'$ is smooth near $y$ and transverse to $\partial B_t(x)$ at $y$, after shrinking $r>0$ if necessary we may assume that $S'\cap B_r(y)$ consists of a single embedded disk, denoted by $\Delta'_y$, and that
    \[
        \gamma\cap B_r(y)=\Delta'_y\cap \partial B_t(x)\cap B_r(y)
    \]
    is a smooth embedded arc.

    We now apply the same local gluing analysis as in the proof of \cref{prop:replacements-in-annuli}; compare also Step~1 in the proof of \cite[Proposition~4.14]{DePhilippisDeRosa}. Here only the regular case arises, since $\Delta'_y$ is smooth and $t$ has been chosen so that the intersection with $\partial B_t(x)$ is transverse. By construction of the outer replacement, near $y$ the disk $\Delta'_y$ is the smooth limit of a sequence of embedded $\mathbf F$-minimizing disks whose boundary arcs on $\partial B_t(x)\cap B_r(y)$ we denote by $\Gamma_j$. By construction of the inner replacement, one can choose a minimizing sequence of embedded disks $F_j\subset \operatorname{An}(x;s,t)\cap B_r(y)$ converging to the relevant local part of $S''$ such that the boundary portion of $F_j$ on $\partial B_t(x)\cap B_r(y)$ is precisely $\Gamma_j$. Since $\partial B_t(x)$ is strictly $\F$-mean convex at this scale, the boundary regularity and boundary curvature estimates apply to the inner minimizing sequence near $\Gamma_j$; see the proof of \cref{prop:replacements-in-annuli}. Running the same Case~1 and Case~2 Hass-Scott analysis as in the proof of \cref{prop:replacements-in-annuli}, one excludes both boundary folding and boundary vanishing. Hence, after shrinking $r$ if necessary, the connected component of $(S'\cup S'')\cap B_r(y)$ containing $\gamma\cap B_r(y)$ is a continuous embedded topological disk, smooth away from $\gamma\cap B_r(y)$, with at most a possible bend along $\gamma\cap B_r(y)$.

    We next exclude the bend exactly as in the last paragraph of the smooth gluing arguments in the proof of \cref{prop:replacements-in-annuli}. If a genuine bend occurred, then after shrinking $r$ once more the relevant component of $(S'\cup S'')\cap B_r(y)$ would be a topological disk $H$ with a bend along an arc of $\gamma$. Since $F$ is a convex elliptic integrand, such a bent disk cannot be $\mathbf F$--minimizing in $B_r(y)$ among disks with the same boundary. Hence there exists another embedded disk $H'\subset B_r(y)$ with
    \[
        \partial H'=\partial H\quad\text{and}\quad\mathbf F(H')<\mathbf F(H).
    \]
    Using $H'$ as a competitor in the minimizing sequence defining the inner replacement, exactly as in the proof of \cref{prop:replacements-in-annuli}, yields a contradiction. Therefore no bend occurs.
    
    We conclude that $S'$ and $S''$ glue smoothly across $\gamma$ near $y$. Since $y\in\gamma$ was arbitrary, the gluing is smooth all along $\gamma$.

    \medskip
    \noindent\textbf{Step 3: regularity in the punctured ball.}

    This step follows \cite[Proof of Proposition~6.3, Step~3]{ColdingDeLellisMinMax}, with the isotropic inputs replaced by their anisotropic counterparts; compare also \cite[Proof of Proposition~4.14, Step~2]{DePhilippisDeRosa}.

    We first show that every connected component $S$ of $S''$ meets $\gamma$:
    \begin{equation}\label{eq:component-meets-gamma}
        S\cap \gamma\neq\emptyset.
    \end{equation}
    Indeed, assume \eqref{eq:component-meets-gamma} fails for some component $S$. Since $\partial B_t(x)$ is strictly $\F$-mean convex and $S\subset B_t(x)\setminus \overline{B_s(x)}$ is $\mathbf F$-stationary, the anisotropic maximum principle implies that $S$ cannot be compactly contained in $B_t(x)\setminus \overline{B_s(x)}$. Hence, $S\cap \partial B_t(x)\neq\emptyset$. Fix $y\in S\cap \partial B_t(x)$. If \eqref{eq:component-meets-gamma} failed, then for some $\sigma>0$ one would have
    \[
        y\in \spt\|W''\|\cap \partial B_t(x),\qquad\bigl(B_\sigma(y)\cap \spt\|W''\|\bigr)\subset B_t(x),
    \]
    that is, $W''$ would touch $\partial B_t(x)$ from the inside at $y$, contradicting again the strict $\F$--mean convexity of $\partial B_t(x)$; for details we can refer to the arguments of the proof of equation~(4.48) in \cite[Proof of Proposition~4.14]{DePhilippisDeRosa}. This proves \eqref{eq:component-meets-gamma}.

    By Step~2 and \eqref{eq:component-meets-gamma}, for every $s<\rho$ the outer support $S'$ extends through $\partial B_t(x)$ to a smooth embedded $\mathbf F$-stationary surface
    \[
        \Sigma_{x,s}\subset \operatorname{An}(x;s,2\rho).
    \]
    Moreover, if $0<s_1<s_2<\rho$, then $\Sigma_{x,s_1}$ and $\Sigma_{x,s_2}$ coincide on the nonempty open annulus $\operatorname{An}(x;t,2\rho)$. Hence, by unique continuation for the anisotropic minimal surface equation,
    \[
        \Sigma_{x,s_1}=\Sigma_{x,s_2}\qquad\text{on }\operatorname{An}(x;s_2,2\rho).
    \]
    Therefore, the family is nested, and we may define
    \[
        \Sigma_x:=\bigcup_{0<s<\rho}\Sigma_{x,s}\subset B_{2\rho}(x)\setminus\{x\}.
    \]
    Then $\Sigma_x$ is a smooth embedded $\mathbf F$-stationary surface in the punctured ball.

    We next prove
    \begin{equation}\label{eq:punctured-support-inclusion}
        \spt\|W\|\cap \bigl(B_\rho(x)\setminus\{x\}\bigr)\subset \Sigma_x.
    \end{equation}
    For $y\in \spt\|W\|\cap B_\rho(x)$, we set $r(y):=|x-y|$, and let $Q_x$ be the set of points $y\in \spt\|W\|\cap B_\rho(x)$ such that
    \[
        TV(y,W)=\{\theta\,|\pi|\}\qquad\text{for some }\theta>0,
    \]
    where $\pi$ is a plane transversal to $\partial B_{r(y)}(x)$. By \cite[Lemma~5.6]{DePhilippisDeRosa}, the set $Q_x$ is dense in $\spt\|W\|\cap B_\rho(x)$.

    We claim that every $y\in Q_x$ belongs to $\Sigma_x$. Fix such a point, set $r:=r(y)$, and let $\widetilde W$ be a replacement of $W'$ in $\operatorname{An}(x;r,t)$. Since $r<\rho$ and $W'=W$ in $B_\rho(x)$, we have
    \[
        \widetilde W=W\qquad\text{on }B_r(x).
    \]
    By the first part of this step, the support of $\widetilde W$ in $\operatorname{An}(x;r,2\rho)$ is smooth, and by nestedness it agrees there with $\Sigma_x$.

    Assume by contradiction that $y\notin \Sigma_x$. Then there exists $a>0$ such that
    \begin{equation}\label{eq:no-outer-support}
        \bigl((\spt\|\widetilde W\|)\setminus B_r(x)\bigr)\cap B_a(y)=\emptyset.
    \end{equation}
    Flatten $\partial B_r(x)$ at $y$, so that in local coordinates
    \[
        \partial B_r(x)=\{p_1=0\},\qquad B_r(x)=\{p_1\le 0\}.
    \]
    Let $C\in TV(y,\widetilde W)$ be any tangent cone of $\widetilde W$ at $y$. By \eqref{eq:no-outer-support}, the support of $C$ is contained in the closed half-space $\{p_1\le 0\}$. On the other hand,
    \[
        \widetilde W\llcorner B_r(x)=W\llcorner B_r(x),
    \]
    and by the choice $y\in Q_x$ we have $TV(y,W)=\{\theta\,|\pi|\}$ for a plane $\pi$ transverse to $\{p_1=0\}$. Hence, the restriction of $C$ to $\{p_1<0\}$ agrees with the restriction of $\theta\,|\pi|$ there, and therefore necessarily, up to a rotation,
    \[
        C=\theta\,|\pi\cap\{p_1\le 0\}|.
    \] 
    In particular, $C\neq 0$. Since $\delta_F\widetilde W=0$, the cone $C$ is stationary for the frozen anisotropic integrand at $y$. This gives a nonzero stationary cone for the frozen integrand supported in a proper half-space, contradicting the anisotropic constancy theorem; see, for instance, \cite[Proposition~5]{DePhilippisDeRosaHirsch}. Therefore, $y\in\Sigma_x$ for every $y\in Q_x$. By density of $Q_x$ we obtain \eqref{eq:punctured-support-inclusion}.

    \medskip
    \noindent\textbf{Step 4: removal of the puncture by overlapping punctured balls.}

    The previous three steps were performed for one fixed point $x\in \spt\|W\|\cap A$. Since $x$ was arbitrary, we may repeat the construction for every point $y\in \spt\|W\|\cap A$. Therefore, for each such $y$ there exist $\rho(y)>0$ and a smooth embedded $\mathbf F$-stationary surface
    \[
        \Sigma_y\subset B_{2\rho(y)}(y)\setminus\{y\}
    \]
    such that
    \[
        \spt\|W\|\cap \bigl(B_{\rho(y)}(y)\setminus\{y\}\bigr)\subset \Sigma_y.
    \]
    We now prove that every point of $\spt\|W\|\cap A$ is regular.

    Fix $p\in \spt\|W\|\cap A$. Since $W$ is a $2$--dimensional integral varifold, the Radon measure $\|W\|$ has no atoms. Hence no point of $\spt\|W\|$ can be isolated. Set
    \[
        d_0:=\dist(p,\mathbb R^3\setminus A)>0,\qquad a:=\min\{\rho_1,d_0\}.
    \]
    Choose
    \[
        q\in \spt\|W\|\cap A\qquad\text{such that}\qquad 0<|p-q|<\frac{a}{32}.
    \]
    Then
    \[
        \dist(q,\mathbb R^3\setminus A)\ge d_0-|p-q|\ge \frac{d_0}{2},
    \]
    and therefore
    \[
        \rho(q):=\frac18\min\{\rho_1,\dist(q,\mathbb R^3\setminus A)\}\ge \frac{a}{16}.
    \]
    In particular,
    \[
        |p-q|<\frac{a}{32}<\rho(q),
    \]
    so $p\in B_{\rho(q)}(q)\setminus\{q\}.$
    Applying Steps~1--3 with center $q$, we obtain a smooth embedded $\mathbf F$--stationary surface $\Sigma_q$ in $B_{2\rho(q)}(q)\setminus\{q\}$ such that
    \[
        \spt\|W\|\cap \bigl(B_{\rho(q)}(q)\setminus\{q\}\bigr)\subset \Sigma_q.
    \]
    Since $p\in B_{\rho(q)}(q)\setminus\{q\}$,
    it follows that $p$ has a neighborhood on which $\spt\|W\|$ is a smooth embedded $\mathbf F$--stationary surface.

    Since $p\in \spt\|W\|\cap A$ was arbitrary, every point of $\spt\|W\|\cap A$ is regular. Hence $\spt\|W\|\cap A$ is a smooth embedded $\mathbf F$--stationary surface, say $\Sigma$. Because $W$ is integral, there exists an integer-valued multiplicity function $\theta$ on $\Sigma$ such that
    \[
        W\llcorner A=\theta\,\mathbf v(\Sigma)\llcorner A.
    \]
    Restricting to $K$ gives the conclusion of \cref{thm:white-improvement}\textup{(1)}.

    (2) Assume in addition that $\Gamma_k:=\partial M_k\subset \partial \mathbf B$ converges smoothly in $\partial \mathbf B$ to a simple closed curve $\Gamma$.

    Exactly as in the proof of \cref{thm:white-improvement}\textup{(1)}, after passing to a subsequence, we have
    \[
        \mathbf v(M_k)\to W\in \mathbf V_2(\mathbb R^3\setminus \Gamma),
    \]
    where $W$ is an integral $\mathbf F$--stationary varifold in $\operatorname{int}(\R^3\setminus \Gamma)$. By \cref{thm:white-improvement}\textup{(1)}, in particular $W$ is regular in every closed neighborhood in $\R^3\setminus \Gamma$. Thus, it remains to analyze the behavior near $\partial \Gamma$. (This in particular treats every point outside of $\Gamma$ and in the support of $W$ as an interior point)

    Fix $x\in \Gamma$. Choose $\rho=\rho(x)>0$ so small that $\Gamma\cap B_{4\rho}(x)$ is a single smooth embedded arc, every sphere $\partial B_r(x)$ with $0<r\le 2\rho$ is strictly $\mathbf F$-mean convex, and the hypotheses of \cref{prop:replacements-in-annuli-boundary} are satisfied in every boundary annulus
    \[
        \operatorname{An}_{\partial}(x;s,t):=\mathbf B\cap\bigl(\tilde{B}_t(x)\setminus \overline{B_s(x)}\bigr),\qquad 0<s<t\leq 2\rho.
    \]
    Here we use the smoothed ball defined right before \cref{prop:replacements-in-annuli-boundary}. Arguing exactly as in Steps~1--3 of the proof of \cref{thm:white-improvement}\textup{(1)}, with ordinary annuli replaced by boundary annuli and \cref{prop:replacements-in-annuli} replaced by \cref{prop:replacements-in-annuli-boundary}, we obtain that for every $s\in(0,\rho)$ there exists a smooth embedded $\mathbf F$--stationary surface-with-boundary
    \[
        \Sigma_{x,s}\subset \operatorname{An}_{\partial}(x;s,2\rho),
    \]
    smooth up to $\partial \mathbf B\cap(\tilde{B}_{2\rho}(x)\setminus \overline{B_s(x)})$, such that
    \[
        \partial \Sigma_{x,s}\cap \partial \mathbf B\subset\Gamma\cap \bigl(\tilde{B}_{2\rho}(x)\setminus \overline{B_s(x)}\bigr),
    \]
    To show that equality in the above display, we only need to assert the reverse inclusion. However note that $\Gamma_k$ converges smoothly to $\Gamma$ , hence the conditions of \cref{lem:density-boundary-estimate} are satisfied and the density lower bound guarantees the reverse inclusion:
    \[
        \partial \Sigma_{x,s}\cap \partial \mathbf B= \Gamma\cap \bigl(\tilde{B}_{2\rho}(x)\setminus \overline{B_s(x)}\bigr),
    \]
    and the family is nested: if $0<s_1<s_2<\rho$, then $\Sigma_{x,s_1}=\Sigma_{x,s_2}$ on $\operatorname{An}_{\partial}(x;s_2,2\rho)$. Here the only additional point, compared with the interior proof, is that if $S$ is a connected component of the inner replacement meeting $\partial \mathbf B$, then $\partial S\cap \partial \mathbf B$ is a union of connected components of $\Gamma\cap \bigl(\tilde{B}_t(x)\setminus \overline{B_s(x)}\bigr)$, so $S$ must also meet the gluing interface $\gamma$ in this case.

    We may therefore set
    \[
        \Sigma_x:=\bigcup_{0<s<\rho}\Sigma_{x,s}\subset (\mathbf B\cap \tilde{B}_{2\rho}(x))\setminus\{x\}.
    \]
    Then $\Sigma_x$ is a smooth embedded $\mathbf F$--stationary surface in the punctured half-ball, smooth up to $\partial \mathbf B\cap \tilde{B}_{2\rho}(x)\setminus\{x\}$, and
    \begin{equation}\label{eq:boundary-local-inclusion-final}
        \Bigl(\spt\|W\|\cap \operatorname{int}(\mathbf B)\cap B_\rho(x)\Bigr)\cup\Bigl(\Gamma\cap B_\rho(x)\setminus\{x\}\Bigr)\subset \Sigma_x .
    \end{equation}
    For points of $\Gamma\setminus\{x\}$, this is immediate from the boundary identity above. For points of $\spt\|W\|\cap \operatorname{int}(\mathbf B)\cap B_\rho(x)$ it is exactly the same half-space argument as in the last part of Step~3 of the proof of \cref{thm:white-improvement}\textup{(1)}, applied to a replacement in $\operatorname{An}_{\partial}(x;r,t)$ with $r=|x-y|$.


    Moreover, by the property of the replacements in \cref{prop:replacements-in-annuli-boundary}, we have 
    $$\mathbf F\bigl(\Sigma_x\cap (\mathbf{B}\cap B_\rho(x))\bigr)=\mathbf F\bigl(W\llcorner (\mathbf{B}\cap B_\rho(x))\bigr).$$
    Hence, we know that $\spt\|W\|=\Sigma_x$ on $ \mathbf B\cap B_\rho(x)\setminus\{x\}$.

    Now fix $p\in \Gamma$, choose $q\in \Gamma$, $q\neq p$, so close that $p\in B_{\rho(q)}(q)\setminus\{q\}$, and apply the previous construction with center $q$. Hence, we have $\spt\|W\|=\Sigma_q$ on $ \mathbf B\cap B_{\rho(q)}(q)\setminus\{q\}$. Choose a neighborhood $U$ of $p$ so small that $\overline{\Sigma_q}\cap \partial \mathbf B\cap U=\Gamma\cap U$. Since $\spt\|W\|\cap \partial \mathbf B=\Gamma$, it follows that $\spt\|W\|\cap U=\Sigma_q\cap U$. Therefore, $p$ is a regular boundary point of $\spt\|W\|$. As $p\in \Gamma$ was arbitrary, every point in $\Gamma$ is regular.

    We have thus shown that $\Sigma := \spt\|W\|$ is a compact, smooth, embedded surface-with-boundary in $\mathbf{B}$ and that $\partial\Sigma = \Gamma$. The boundary convergence is with multiplicity one, so at most one component of \(\Sigma\) can contain the prescribed boundary. If another component were present, it would be a compact \(\F\)-minimal surface without boundary compactly contained in $\mathbf{B}$. This is impossible by the anisotropic maximum principle, then \(\Sigma\) is connected; in particular, it is a single smooth embedded surface-with-boundary in $\mathbf{B}$ and $W=\mathbf{v}(\Sigma)$. It remains to show that $\Sigma$ is a disk, which follows from an argument very similar to that in \cite[Section~8]{AS}. For completeness, we sketch it in the following.

    Since $\Sigma$ is a compact $C^2$ embedded surface-with-boundary, there exist $\delta>0$ and a $C^2$ tubular retract
    \[
        \mu:\{z\in \mathbb R^3:\dist(z,\Sigma)\le \delta\}\to\Sigma.
    \]
    We now argue exactly as in \cite[Section~8]{AS}, with \cref{prop:replacement-thm-AS-aniso} in place of the local isotropic replacement theorem. Using a finite cubical decomposition of a tubular neighborhood of $M$, together with the local anisotropic replacement procedure in small uniformly convex sets, one modifies some $M_k$ finitely many times, for $k$ sufficiently large, to an embedded disk $N$ such that
    \[
        \partial N=\Gamma\quad\text{and}\quad N\subset \{z\in\mathbb R^3:\dist(z,\Sigma)\le \delta\}.
    \]
    Here, if necessary, we first compose $M_k$ with an ambient diffeomorphism of $\mathbf B$ sending $\Gamma_k$ to $\Gamma$ and converging to the identity in $C^2$.
    
    Let $\chi_1:\mathbf D\to N$ be a $C^2$ diffeomorphism. Then
    \[
        \chi:=\mu\circ \chi_1:\mathbf D\to \Sigma
    \]
    is a $C^2$ map, and $\chi|_{\partial\mathbf D}$ is a diffeomorphism onto $\Gamma$. Thus, $\Gamma$ is null-homotopic in $\Sigma$. Exactly as in \cite[Section~8]{AS}, the existence of such a map implies that $\Sigma$ is orientable. By the classification of compact connected orientable surfaces with one boundary component, either $\Sigma$ is a disk or else it is a disk with $h\ge 1$ handles. The latter is impossible, because in that case the unique boundary curve is not null-homotopic. Therefore, $\Sigma$ is a disk and hence $\Sigma\in\mathcal M$.

    Finally, if $P\in \mathcal M$ satisfies $\partial P=\Gamma$, then choose ambient diffeomorphisms $\Phi_k:\mathbf B\to \mathbf B$ such that $\Phi_k(\Gamma)=\Gamma_k$ and $\Phi_k\to \mathrm{Id}$ in $C^2$, and set $P_k:=\Phi_k(P)$. Then we have $\partial P_k=\Gamma_k=\partial M_k$ and $\mathbf F(P_k)\to \mathbf F(P)$, so
    \[
        \mathbf F(M_k)\le \mathbf F(P_k)+\varepsilon_k.
    \]
    Passing to the limit gives
    \[
        \mathbf F(\Sigma)\le \mathbf F(P).
    \]
    Hence $\Sigma$ minimizes $\mathbf F$ among all smooth embedded disks spanning $\Gamma$. This proves \cref{thm:white-improvement}\textup{(2)}.

    (3) Assume that $S\subset \partial \mathbf B$ is relatively open and $\Gamma_k\cap S$ converges smoothly to a simple arc $\Gamma$, for any point $x\in\Gamma$, after choosing $\rho=\rho(x)>0$ sufficiently small we may assume that $\Gamma\cap B_{4\rho}(x)$ is a single smooth arc contained in $S$. Then the boundary-annulus replacement argument used in the proof of \cref{thm:white-improvement}\textup{(2)} applies verbatim in $\mathbf B\cap B_{2\rho}(x)$, and yields a smooth embedded $\mathbf F$--stationary surface $\Sigma_x\subset (\mathbf B\cap B_{2\rho}(x))\setminus\{x\}$, smooth up to $\partial \mathbf B\cap B_{2\rho}(x)\setminus\{x\}$
    such that
    \[
        W=\mathbf v(\Sigma_x)\qquad\text{on }(\mathbf B\cap B_\rho(x))\setminus\{x\}.
    \]
    Similarly, \cref{lem:density-boundary-estimate} is used to rule out the cancellation of the boundary in the limit varifold.
    Using the same covering argument as in \cref{thm:white-improvement}\textup{(2)}, we can remove the puncture and concludes that every point of $\Gamma$ is a regular boundary point of $\spt\|W\|$. Therefore, $\spt\|W\|$ has a component that extends smoothly up to the smooth part $\Gamma$.
\end{proof}


\subsection{Proof of \cref{thm:Meeks-Simon-Yau-without-boundary} and \cref{thm:Meeks-Simon-Yau-with-boundary}: $\gamma$-reduction}\label{sec:anis-gamma-reduction}

In this subsection, we introduce the anisotropic analogue of the $\gamma$-reduction machinery in \cite[Section 3]{MSY}. Fix $\varepsilon_0>0$ so that \cref{lem:anis-1} holds and let
\[
0<\gamma<\varepsilon_0^2/9.
\]
By abuse of notation, we continue to write $\mathbf F(\cdot)$ for the anisotropic area of compact embedded surfaces with boundary, and more generally of compact countably $2$-rectifiable sets. For compact sets $E,F\subset N$, we write
\[
E\Delta F:=(E\setminus F)\cup(F\setminus E).
\]

\begin{definition}[Anisotropic $\gamma$-reduction]
\label{def:anis_gamma_reduction}
Let $\Sigma_1,\Sigma_2\in\mathcal C_1$. We say that $\Sigma_2$ is an anisotropic $\gamma$-reduction of $\Sigma_1$, and write $\Sigma_2\ll_{\gamma,F}\Sigma_1$, if the following conditions are satisfied:
\begin{enumerate}
    \item There exists a closed annulus $A \subset \Sigma_1$ (diffeomorphic to $\{x \in \mathbb{R}^2 : 1/2 \leq |x| \leq 1\}$) and two disjoint discs $D_1, D_2$ each diffeomorphic to $\textbf{D}$ such that:
    \[
    \overline{\Sigma_1\setminus \Sigma_2}=A,\qquad
    \overline{\Sigma_2\setminus \Sigma_1}=D_1\cup D_2,
    \qquad
    \partial A=\partial D_1\cup\partial D_2.
    \]
    \item The surgery region has small anisotropic area:
    \[
    \mathbf F(A)+\mathbf F(D_1)+\mathbf F(D_2)<2\gamma.
    \]
    \item There exists a compact set $\Upsilon\subset N$ such that
    \[
    A\cup D_1\cup D_2=\partial\Upsilon, \quad \Upsilon\text{ homeomorphic to }\mathbf{B}\text{ and } (\Upsilon\setminus\partial\Upsilon)\cap(\Sigma_1\cup\Sigma_2)=\emptyset.
    \]
    \item Let $\Sigma_1^\ast$ be the connected component of $\Sigma_1$ containing $A$. If $\Sigma_1^\ast\setminus A$ is disconnected, then each connected component of $\Sigma_1^\ast\setminus A$ is either not simply connected or else has anisotropic area at least $\varepsilon_0^2/2$.
\end{enumerate}
\end{definition}

\begin{rmk}
\label{rmk:anis_gamma_basic}
The arguments in \cite[(3.1)--(3.5)]{MSY} are purely topological and carry over verbatim to the present setting. In particular, if $\Sigma_2\ll_{\gamma,F}\Sigma_1$, then every two-sided component of $\Sigma_1$ gives rise to two-sided components of $\Sigma_2$, and
\[
\operatorname{genus}(\Sigma_2)\le \operatorname{genus}(\Sigma_1).
\]
Moreover, for every $G,\Lambda>0$ there exists
$c=c(\varepsilon_0,G,\Lambda/\varepsilon_0^2)$
such that if
\[
\Sigma_k\ll_{\gamma,F}\Sigma_{k-1}\ll_{\gamma,F}\cdots\ll_{\gamma,F}\Sigma_1,
\qquad
\operatorname{genus}(\Sigma_1)\le G,\quad \mathbf F(\Sigma_1)\le\Lambda,
\]
then
\[
k\le c
\quad\text{and}\quad
\mathbf F(\Sigma_k\Delta\Sigma_1)\le 2c\gamma.
\]
\end{rmk}

A surface $\Sigma\in\mathcal C_1$ is called anisotropic $\gamma$-irreducible if there does not exist any $\widetilde\Sigma\in\mathcal C_1$ such that
$\widetilde\Sigma\ll_{\gamma,F}\Sigma$.
Then the anisotropic analogue of \cite[Remark 3.6]{MSY} also holds: 
\begin{rmk}
\label{rmk:anis_gamma_disc_criterion}
$\Sigma\in\mathcal C_1$ is anisotropic $\gamma$-irreducible if and only if whenever $\Delta\subset N$ is an embedded disc satisfying
$\partial\Delta=\Delta\cap\Sigma$ and
$\mathbf F(\Delta)<\gamma$,
there exists a disc $\widetilde{\Delta}\subset\Sigma$ such that
$\partial \widetilde{\Delta}=\partial\Delta
$ and $
\mathbf F(\widetilde{\Delta})<\varepsilon_0^2/2$.
The rough idea of the proof is that one replaces a thin annular neighborhood of $\partial\Delta$ in $\Sigma$ by two nearby push-offs of $\Delta$. The only additional point is that, since the anisotropic integrand is smooth, the anisotropic area of the thin annulus can be made arbitrarily small and the anisotropic areas of the push-offs can be made arbitrarily close to $\mathbf F(\Delta)$.
\end{rmk}

\begin{definition}[Weak anisotropic $\gamma$-reduction]
We say $\Sigma_2$ is a weak anisotropic $\gamma$-reduction of $\Sigma_1$, denoted $\Sigma_2 <_{\gamma, F} \Sigma_1$, if there exists $\widetilde{\Sigma}_1 \in \mathcal{G}(\Sigma_1)$ such that:
\[
\mathbf F(\widetilde\Sigma_1\Delta\Sigma_1)<\gamma
\quad\text{and}\quad
\Sigma_2\ll_{\gamma,F}\widetilde\Sigma_1.
\]
\end{definition}

\begin{rmk}
\label{rmk:weak_anis_gamma}
The anisotropic analogue of \cite[(3.8)--(3.10)]{MSY} holds as well. Namely, for every $G,\Lambda>0$ there exists
$c=c(\varepsilon_0,G,\Lambda/\varepsilon_0^2)$
such that if
\[
\Sigma_k<_{\gamma,F}\Sigma_{k-1}<_{\gamma,F}\cdots<_{\gamma,F}\Sigma_1,
\qquad
\operatorname{genus}(\Sigma_1)\le G,\quad \mathbf F(\Sigma_1)\le\Lambda,
\]
then
\[
k\le c
\quad\text{and}\quad
\mathbf F(\Sigma_k\Delta\Sigma_1)\le 3c\gamma.
\]
Indeed, each weak $\gamma$-reduction consists of an isotopic perturbation of anisotropic area $<\gamma$ followed by a genuine anisotropic $\gamma$-reduction.
\end{rmk}

\begin{definition}[Strongly anisotropic $\gamma$-irreducible]
A surface $\Sigma \in \mathcal{C}_1$ is strongly anisotropic $\gamma$-irreducible if there does not exist any $\widetilde{\Sigma} \in \mathcal{C}_1$ such that $\widetilde{\Sigma} <_{\gamma, F} \Sigma$.
\end{definition}
\begin{rmk}
\label{rmk:strongly_anis_gamma}
If $\Sigma$ is strongly anisotropic $\gamma$-irreducible, then $\Sigma$ is anisotropic $\gamma$-irreducible. Moreover, exactly as in \cite[Remark 3.11]{MSY}, if $\Sigma_1\in\mathcal G(\Sigma)$ and
$\mathbf F(\Sigma_1\Delta\Sigma)<\theta<\gamma$,
then strong anisotropic $\gamma$-irreducibility of $\Sigma$ implies strong anisotropic $(\gamma-\theta)$-irreducibility of $\Sigma_1$. Finally, by \cref{rmk:weak_anis_gamma}, every $\Sigma\in\mathcal C_1$ admits, after finitely many weak anisotropic $\gamma$-reductions, a strongly anisotropic $\gamma$-irreducible representative.
\end{rmk}

These definitions and remarks are direct analogues of those in \cite[pp.~628--630]{MSY}, with the isotropic area replaced by $\mathbf{F}$. 
For further details, we refer the reader to \cite[pp.~628--630]{MSY}.

With these notions in place, we are able to derive an anisotropic version of Theorem~2 in \cite{MSY}.
\begin{proposition}[Anisotropic analogue of {\cite[Theorem 2]{MSY}}]
\label{prop:anis_thm2}
Suppose $A\subset N$ is diffeomorphic to\/ $\mathbf B$, suppose $\Sigma\in\mathcal C_1$,
$E(\Sigma) := \mathbf{F}(\Sigma) - \inf_{\tilde{\Sigma} \in \mathcal{G}(\Sigma)} \mathbf{F}(\tilde{\Sigma}) \leq \gamma/4$,
$\Sigma$ is strongly anisotropic $\gamma$-irreducible, $\Sigma$ intersects $\partial A$ transversely, and for each component $\Gamma$ of $\Sigma\cap\partial A$, let $F_\Gamma\subset\partial A$ be one of the two discs bounded by $\Gamma$, chosen so that $\mathbf F(F_\Gamma)$ is minimal among the two choices. Suppose furthermore that if
$\Gamma_1,\cdots,\Gamma_q$
denote the components of $\Sigma\cap\partial A$, and if we write $F_j:=F_{\Gamma_j}$, then
$\sum_{j=1}^q \mathbf F(F_j)\le \gamma/8.$
Let $\Sigma_0$ denote the union of all components $\Lambda$ of $\Sigma$ such that
$\Lambda\subset K_\Lambda$ and $\partial K\cap\Sigma=\emptyset$
for some compact $K_\Lambda\subset N$ diffeomorphic to\/ $\mathbf B$.

Then $\mathbf F(\Sigma_0)\le E(\Sigma)$,
and there exist pairwise disjoint closed discs $D_1,\cdots,D_p\subset \Sigma\setminus\Sigma_0$ with $\partial D_i\subset\partial A$ such that
\[
\left(\bigcup_{i=1}^p D_i\right)\cap (A\setminus\partial A)
=
(\Sigma\setminus\Sigma_0)\cap (A\setminus\partial A),
\qquad
\sum_{i=1}^p \mathbf F(D_i)\le \sum_{j=1}^q \mathbf F(F_j)+E(\Sigma),
\]
and with
\[
\bigcup_{i=1}^{p}\left(\varphi_1(D_i\setminus\partial D_i)\right)\subset A\setminus\partial A.
\]
for some isotopy $\varphi=\{\varphi_t\}_{0\le t\le 1}$ of $N$ such that $\varphi_t(x)=x$ for all $(x,t)\in W\times[0,1]$,
where $W$ is a neighborhood of $(\Sigma\setminus\Sigma_0)\setminus \bigcup_{i=1}^p (D_i\setminus\partial D_i)$.
\end{proposition}

\begin{proof}
    The proof follows closely that of \cite[Proof of Theorem~2]{MSY}, with the isotropic area replaced by the anisotropic energy. For completeness, we provide the details here.
    
    As in \cite[Remark 3.14]{MSY}, it suffices to prove the theorem in the special case $\Sigma_0=\emptyset$. The general case then follows by the same topological argument. Thus we shall assume, in the proof, that 
    \begin{equation}\label{eq:emptySigma_0}
        \Sigma_0=\emptyset.
    \end{equation}
    We proceed by induction on $q$, under the stronger hypotheses 
    \begin{equation}\label{eq:strogerH}
        \sum_{j=1}^q \mathbf F(F_j)\le \gamma/8,\qquad E(\Sigma)\le \gamma/2-2\sum_{j=1}^q \mathbf F(F_j),
    \end{equation}
    and, writing $\tilde{\gamma}:=\gamma/4+4\sum_{j=1}^q \mathbf F(F_j)+E(\Sigma)$, we assume moreover that $\Sigma$ is strongly anisotropic $\tilde{\gamma}$-irreducible.

    Notice that the hypotheses of the theorem imply \eqref{eq:strogerH}, and since
    \[
        \tilde{\gamma}=\gamma/4+4\sum_{j=1}^q \mathbf F(F_j)+E(\Sigma)\le \gamma,
    \]
    strong anisotropic $\gamma$-irreducibility implies strong anisotropic $\tilde{\gamma}$-irreducibility.

    If $q=0$, then $\Sigma\cap\partial A=\emptyset$. Since $A$ is diffeomorphic to $\mathbf B$ and \eqref{eq:emptySigma_0} holds, no component of $\Sigma$ can be contained in $A$. Hence $\Sigma\cap A=\emptyset$, and the conclusion is trivial.

    Assume that the theorem has been proved with $q-1$ in place of $q$. Relabeling if necessary, we may assume that $\Gamma_q$ is innermost on $\partial A$, so that $F_q\cap\Sigma=\Gamma_q$. Since $\Sigma$ is strongly anisotropic $\tilde{\gamma}$-irreducible and $\mathbf F(F_q)\le \tilde{\gamma}$, \cref{rmk:anis_gamma_disc_criterion} gives a disk $D\subset\Sigma$ such that $\partial D=\Gamma_q$ and $\mathbf F(D)<\varepsilon_0^2/2$. Also, $\mathbf F(F_q)\le\gamma/8<\varepsilon_0^2/72<\varepsilon_0^2/2$, hence $\mathbf F(D\cup F_q)<\varepsilon_0^2$. Therefore, by \cref{lem:anis-1}, the sphere $D\cup F_q$ bounds a region $U$ homeomorphic to $\mathbf B$. Let $\Lambda$ be the component of $\Sigma$ containing $D$, and consider the possibility that $\Lambda\setminus D\subset U$. Since $\overline{U}\cong\mathbf B$, this would imply $\Lambda\subset\Sigma_0$, contradicting \eqref{eq:emptySigma_0}. Hence we must have
    \begin{equation}\label{eq:anis-319}
        \Sigma\setminus D\cap U=\emptyset.
    \end{equation}

    Write $\Sigma^\ast:=(\Sigma\setminus D)\cup F_q$, and $F_{q,\varepsilon}:=\{x\in N:\dist(x,F_q)<\varepsilon\}$. For each $\varepsilon>0$, exactly as in \cite[Theorem 2]{MSY}, using \eqref{eq:anis-319}, we can select a continuous isotopy $y=\{y_t\}_{0\le t\le 1}$ such that
    \[
        y_t(F_{q,\varepsilon})\subset F_{q,\varepsilon},\qquad y_t(x)=x\ \text{for }x\notin F_{q,\varepsilon},
    \]
    \[
    y_1(\Sigma^\ast\cap F_{q,\varepsilon})\cap \partial A=\emptyset,\qquad \Sigma_\varepsilon^\ast:=y_1(\Sigma^\ast)\in\mathcal C_1.
    \]
    For $\varepsilon$ small enough, we may also arrange that
    \begin{align}\label{eq:anis-thm2-i}
        &\Sigma_\varepsilon^\ast\in\mathcal G(\Sigma),
        \qquad
        \Sigma_\varepsilon^\ast\cap\partial A=\Gamma_1\cup\cdots\cup\Gamma_{q-1},\\\label{eq:anis-thm2-ii}
        &\mathbf F\big((\Sigma_\varepsilon^\ast\setminus\Sigma)\cup(\Sigma\setminus\Sigma_\varepsilon^\ast)\big)<\mathbf F(D)+\mathbf F(F_q)+\varepsilon,\\\label{eq:anis-thm2-iii}
        &\mathbf F(\Sigma_\varepsilon^\ast)<\mathbf F(\Sigma)+\mathbf F(F_q)-\mathbf F(D)+\varepsilon.
    \end{align}
    Notice that \eqref{eq:anis-thm2-iii} is equivalent to
    \begin{equation}\label{eq:anis-thm2-iii-prime}
        E(\Sigma_\varepsilon^\ast)<E(\Sigma)+\mathbf F(F_q)-\mathbf F(D)+\varepsilon.
    \end{equation}
    Taking $\varepsilon\le \mathbf F(F_q)$, we deduce from \eqref{eq:anis-thm2-i}, \eqref{eq:anis-thm2-ii}, and \cref{rmk:strongly_anis_gamma} that $\Sigma_\varepsilon^\ast$ is strongly anisotropic $\gamma^\ast$-irreducible, where
    \[
        \gamma^\ast=\gamma/4+E(\Sigma)+4\sum_{j=1}^{q-1}\mathbf F(F_j)+2\mathbf F(F_q)-\mathbf F(D).
    \]
    By \eqref{eq:anis-thm2-iii-prime}, we thus have $\gamma^\ast>\gamma/4+E(\Sigma_\varepsilon^\ast)+4\sum_{j=1}^{q-1}\mathbf F(F_j)$. Furthermore, using \eqref{eq:anis-thm2-iii-prime}, \eqref{eq:strogerH}, and $\varepsilon\le \mathbf F(F_q)$,
    \[
        E(\Sigma_\varepsilon^\ast)<E(\Sigma)+2\mathbf F(F_q)\le\gamma/2-2\sum_{j=1}^{q}\mathbf F(F_j)+2\mathbf F(F_q)=\gamma/2-2\sum_{j=1}^{q-1}\mathbf F(F_j).
    \]
    Thus $\Sigma_\varepsilon^\ast$ satisfies the inductive hypotheses with $q-1$ in place of $q$. Hence there exist pairwise disjoint closed disks $\widetilde{\Delta}_1,\cdots,\widetilde{\Delta}_p\subset \Sigma_\varepsilon^\ast$ with $\partial \widetilde{\Delta}_i\subset\partial A$ such that
    \[
        \left(\bigcup_{i=1}^p \widetilde{\Delta}_i\right)\cap(A\setminus\partial A)=\Sigma_\varepsilon^\ast\cap(A\setminus\partial A),
   \qquad
        \sum_{i=1}^p \mathbf F(\widetilde{\Delta}_i)\le\sum_{j=1}^{q-1}\mathbf F(F_j)+E(\Sigma_\varepsilon^\ast),
    \]
    and $\widetilde\varphi_1(\widetilde{\Delta}_i\setminus\partial \widetilde{\Delta}_i)\subset A\setminus\partial A$
    for some isotopy $\widetilde\varphi=\{\widetilde\varphi_t\}_{0\le t\le 1}$ which leaves a neighborhood of $\Sigma_\varepsilon^\ast\setminus\bigcup_{i=1}^p (\widetilde{\Delta}_i\setminus\partial \widetilde{\Delta}_i)$ fixed. It follows that there are pairwise disjoint closed disks $\Delta_1,\cdots,\Delta_p\subset \Sigma^\ast=(\Sigma\setminus D)\cup F_q$ such that
    \[
        \left(\bigcup_{i=1}^p \Delta_i\right)\cap(A\setminus\partial A)=\Sigma^\ast\cap(A\setminus\partial A),\qquad \partial\Delta_i=\partial\widetilde{\Delta}_i,
    \]
    \begin{equation}\label{eq:anis-321}
        \sum_{i=1}^p \mathbf F(\Delta_i)\le\sum_{j=1}^{q-1}\mathbf F(F_j)+E(\Sigma)+\mathbf F(F_q)-\mathbf F(D)+\varepsilon,
    \end{equation}
    and such that $\psi_1(\Delta_i\setminus\partial \Delta_i)\subset A\setminus\partial A$ for some isotopy $\psi=\{\psi_t\}_{0\le t\le 1}$ which leaves a neighborhood of $\Sigma^\ast\setminus\bigcup_{i=1}^p (\Delta_i\setminus\partial \Delta_i)-F_q$ fixed. In fact, up to an arbitrarily small $C^1$-perturbation, one may take $A_i=y_1^{-1}(A_i)$, $\psi=\widetilde\varphi \circ y$.

    Now let $\beta=\{\beta_t\}_{0\le t\le 1}$ be a continuous isotopy such that
    \[
        \beta_t(U)\subset U,\qquad \beta_t|_{\Sigma\setminus D}=\mathrm{id}_{\Sigma\setminus D},\qquad \beta_1(D)=F_q.
    \]
    Such an isotopy exists because $U\cong\mathbf B$ and \eqref{eq:anis-319} hold. As in \cite[Theorem 2]{MSY}, we consider two cases.

    \emph{Case (i):} \(F_q\subset \bigcup_{i=1}^p \Delta_i\).
    
    Then there is a unique index \(i_0\) such that \(F_q\subset \Delta_{i_0}\). We define
    \[
        D_{i_0}:=(\Delta_{i_0}\setminus F_q)\cup D,\qquad D_i:=\Delta_i\quad (i\neq i_0).
    \]
    Then the \(D_i\) are pairwise disjoint closed discs in \(\Sigma\), and by \eqref{eq:anis-321},
    \[
        \sum_{i=1}^p \mathbf F(D_i)=\sum_{i=1}^p \mathbf F(\Delta_i)-\mathbf F(F_q)+\mathbf F(D)\le\sum_{j=1}^{q}\mathbf F(F_j)+E(\Sigma)+\varepsilon.
    \]
    We define a continuous isotopy \(\varphi\) by \(\varphi=\psi \circ \beta\); smoothing as in \cite[Theorem 2]{MSY}, we obtain the required isotopy.

    \emph{Case (ii):} \(F_q\not\subset \bigcup_{i=1}^p \Delta_i\).
    
    We define
    \[
        D_i:=\Delta_i\quad (i=1,\cdots,p),\qquad D_{p+1}:=D.
    \]
    Then the \(D_i\) are pairwise disjoint closed disks in \(\Sigma\), and by \eqref{eq:anis-321},
    \[
        \sum_{i=1}^{p+1}\mathbf F(D_i)=\sum_{i=1}^{p}\mathbf F(\Delta_i)+\mathbf F(D)\le\sum_{j=1}^{q}\mathbf F(F_j)+E(\Sigma)+\varepsilon.
    \]
    Exactly as in \cite[Theorem 2]{MSY}, there is a neighborhood \(W\) of \(\partial D(=\Gamma_q)\) such that $W\cap D\subset A$. Otherwise, we would have \(W\supset \partial D\) with $W\cap(\Sigma\setminus D)\subset A\setminus\partial A$,
    which, by the covering property above, would imply \(F_q\subset \bigcup_{i=1}^p \Delta_i\), contrary to the assumption of Case (ii). Hence one defines a continuous isotopy \(\varphi\) by $\varphi=\beta \circ (\psi|_{N\setminus W})$, and smoothing as in \cite[Theorem 2]{MSY}, one obtains the required isotopy.

    In either case, since \(\varepsilon>0\) was arbitrary, the required area estimate follows, and the proof by induction is complete.
\end{proof}

After this point, one argues exactly as on \cite[pp.~634--635]{MSY}, replacing $\gamma$-reduction, weak $\gamma$-reduction, and strong $\gamma$-irreducibility by their anisotropic counterparts. Starting from a minimizing sequence $\{\Sigma_k\}$ for $\mathbf F$, one extracts, as in \cite[pp.~634--635]{MSY}, a number $\gamma_0\in (0,\varepsilon_0^2/9)$, a subsequence $\{\Sigma_{k'}\}$, and surfaces $\widetilde\Sigma_{k'}\in\mathcal C_1$ such that each $\widetilde\Sigma_{k'}$ is strongly anisotropic $\gamma_0$-irreducible,
\[
\mathbf F(\widetilde\Sigma_{k'}\Delta \Sigma_{k'})\to 0,
\]
the associated varifolds have the same limit,
\[
\mathbf v(\widetilde\Sigma_{k'})-\mathbf v(\Sigma_{k'})\to 0
\quad\text{weakly},
\]
and
\[
E(\widetilde\Sigma_{k'})\to 0.
\]
Thus, after passing to a subsequence and relabeling, we may assume that the minimizing sequence itself is strongly anisotropic $\gamma_0$-irreducible and has the same varifold limit as the original one. This is the sequence to which we apply the anisotropic analogue of \cite[Theorem~2]{MSY}, i.e. \cref{prop:anis_thm2}.

In the following, we give the proof of \cref{thm:Meeks-Simon-Yau-without-boundary}.

\begin{proof}[Proof of \cref{thm:Meeks-Simon-Yau-without-boundary}]
    By the discussion immediately following \cref{prop:anis_thm2}, after passing to a subsequence and relabeling, we may assume that there exists $\gamma_0\in (0,\varepsilon_0^2/9)$ such that $\{\Sigma_k\}\subset \mathcal C_1$ is a strongly anisotropic $\gamma_0$-irreducible minimizing sequence,
    \[
        E(\Sigma_k):=\mathbf F(\Sigma_k)-\inf_{\widetilde\Sigma\in\mathcal G(\Sigma_k)}\mathbf F(\widetilde\Sigma)\to 0,
    \]
    and $\mathbf v(\Sigma_k)\to V$ for some $\F$-stationary varifold $V$. Moreover, this is the same varifold limit as that of the original minimizing sequence. Exactly as in \cite[Remark 3.14]{MSY}, let $\widetilde{\Sigma}_k$ be obtained from $\Sigma_k$ by deleting all connected components $\Lambda$ such that $\Lambda\subset K_\Lambda$ and $\partial K_\Lambda\cap \Sigma_k=\emptyset$ for some compact $K_\Lambda\subset N$ diffeomorphic to $\mathbf B$. Then $E(\widetilde{\Sigma}_k)\to 0$ and $v(\widetilde{\Sigma}_k)\to V$. Furthermore, we know that, for all sufficiently large $k$, $\widetilde{\Sigma}_k$ is strongly anisotropic $(3\gamma_0/4)$-irreducible. In particular, after discarding finitely many terms, we may assume that each $\widetilde{\Sigma}_k$ is strongly anisotropic $(\gamma_0/2)$-irreducible.

    We first prove that
    \begin{align}
        V=\sum_{i=1}^R n_i\,\mathbf v(\Sigma^{(i)})
    \end{align}
    for some positive integers $R,n_1,\cdots,n_R$ and pairwise disjoint connected smooth embedded $\F$-minimal surfaces $\Sigma^{(1)},\cdots,\Sigma^{(R)}$.

    Let $x_0\in \spt\|V\|$. Using the anisotropic first variation formula together with the usual cutoff-radial vector field argument, exactly as in the isotropic case one obtains $\|V\|(\{x_0\})=0$. Choose $\rho_0>0$ so small that every geodesic sphere $\partial B_\rho(x_0)$ with $0<\rho\le \rho_0$ is uniformly $\F$-convex. Since $\|V\|(\{x_0\})=0$, we may choose $0<\rho<\rho_0$ so that $\|V\|\bigl(B_\rho(x_0)\setminus B_{3\rho/4}(x_0)\bigr)$ is as small as we want and $\|V\|(\partial B_\rho(x_0))=\|V\|(\partial B_{3\rho/4}(x_0))=0$. Hence
    \[
        \mathcal H^2\!\left(\widetilde{\Sigma}_k\cap \bigl(B_\rho(x_0)\setminus B_{3\rho/4}(x_0)\bigr)\right)\to\|V\|\bigl(B_\rho(x_0)\setminus B_{3\rho/4}(x_0)\bigr).
    \]

    Similar to \cite[(5.1)-(5.3)]{MSY}, we apply the coarea formula to the distance function from $x_0$ on the smooth surface $\widetilde{\Sigma}_k$, and we can choose $\rho_k\in (3\rho/4,\rho)$ such that $\widetilde{\Sigma}_k$ intersects $\partial B_{\rho_k}(x_0)$ transversely satisfying
    \begin{equation}\label{eq:anis-5-3}
        \operatorname{length}(\widetilde{\Sigma}_k\cap \partial B_{\rho_k}(x_0))\le\frac{C}{\rho}\,\mathcal H^2\!\left(\widetilde{\Sigma}_k\cap \bigl(B_\rho(x_0)\setminus B_{3\rho/4}(x_0)\bigr)\right).
    \end{equation}
    For each component $\Gamma$ of $\widetilde{\Sigma}_k\cap \partial B_{\rho_k}(x_0)$, let $F_\Gamma\subset \partial B_{\rho_k}(x_0)$ be the smaller of the two discs bounded by $\Gamma$. Since $\partial B_{\rho_k}(x_0)$ is a small geodesic sphere, the two-dimensional isoperimetric inequality on $\partial B_{\rho_k}(x_0)$, together with the uniform ellipticity of $F$, yields
    \[
        \sum_\Gamma \mathbf F(F_\Gamma)\le C\,\rho_k\,\operatorname{length}(\widetilde{\Sigma}_k\cap \partial B_{\rho_k}(x_0))
    \]
    for a constant $C$ independent of $k$. Hence, by \eqref{eq:anis-5-3},
    \[
        \sum_\Gamma \mathbf F(F_\Gamma)\le C\,\mathcal H^2\!\left(\widetilde{\Sigma}_k\cap \bigl(B_\rho(x_0)\setminus B_{3\rho/4}(x_0)\bigr)\right).
    \]
    Choosing $\rho$ sufficiently small and then $k$ sufficiently large, we may therefore assume that
    \[
        \sum_\Gamma \mathbf F(F_\Gamma)\le \gamma_0/16.
    \]
    Since also $E(\widetilde{\Sigma}_k)\to 0$, we may further assume that
    \[
        E(\widetilde{\Sigma}_k)\le \gamma_0/8.
    \]
    Therefore \cref{prop:anis_thm2} applies to $\widetilde{\Sigma}_k$ in the ball $B_{\rho_k}(x_0)$ with $\gamma_0/2$ in place of $\gamma$. Thus there exist pairwise disjoint disks $D_k^{(1)},\cdots,D_k^{(q_k)}\subset \widetilde{\Sigma}_k$ and isotopies $\varphi^{(k)}=\{\varphi_t^{(k)}\}_{0\le t\le 1}$ of $N$ such that
    \begin{equation}\label{eq:anis-5-4}
        \left\{\begin{aligned}
            &\partial D_k^{(i)}\subset \partial B_{\rho_k}(x_0),\quad \widetilde{\Sigma}_k\cap B_{\rho_k}(x_0)=\left(\bigcup_{i=1}^{q_k}D_k^{(i)}\right)\cap B_{\rho_k}(x_0),\\&\varphi_1^{(k)}\bigl(D_k^{(i)}\setminus \partial D_k^{(i)}\bigr)\subset B_{\rho_k}(x_0)\setminus \partial B_{\rho_k}(x_0),\quad \sum_{i=1}^{q_k}\mathbf F(D_k^{(i)})\le\sum_\Gamma \mathbf F(F_\Gamma)+E(\widetilde{\Sigma}_k).
        \end{aligned}\right.
    \end{equation}
    Hence, by shrinking $\rho$ and choosing $k$ sufficiently large if necessary, we may assume that \cref{lemma3:area-bound-assumption} holds for $B_{\rho_k}(x_0)$, $D_k^{(i)}$, $i=1,\cdots,q_k$. Then, exactly as in \cite[(5.5)--(5.7)]{MSY}, and using \cref{prop:replacement-thm-AS-aniso} in place of the isotropic replacement theorem, we obtain
    \[
        \partial \widetilde{D}_k^{(i)}=\partial D_k^{(i)},\qquad \widetilde{D}_k^{(i)}\setminus\partial \widetilde{D}_k^{(i)}\subset B_{\rho_k}(x_0), \qquad\mathbf F(\widetilde{D}_k^{(i)})\le \mathbf F(D_k^{(i)}),
    \]
    and
    \begin{equation}\label{eq:anis-5-7}
        \mathbf F(D_k^{(i)})\leq\mathbf F(\widetilde{D}_k^{(i)})+\varepsilon_{k,i}\leq\inf\{\mathbf F(\Delta): \Delta\in \mathcal D_{i,k}\}+2\varepsilon_{k,i},\quad i=1,\cdots,q_k,
    \end{equation}
    where $\mathcal D_{i,k}$ denotes the set of all discs in $N$ with boundary $\partial D_k^{(i)}$ and where $\sum_{i=1}^{q_k}\varepsilon_{k,i}\to 0$ as $k\to\infty$.

    By \eqref{eq:anis-5-4} and \eqref{eq:anis-5-7}, \cref{lem:aniso-filigree-lemma} applies exactly as in \cite[(5.7)--(5.9)]{MSY}. Hence there exists an integer $\ell$, independent of $k$, such that for all sufficiently large $k$,
    \[
        \mathbf F(D_k^{(i)}\cap B_{\rho/2}(x_0))>4\varepsilon_{k,i}
    \]
    for at most $\ell$ integers $i\in\{1,\cdots,q_k\}$. Thus, relabeling if necessary, we have
    \begin{align}
        \left\{
        \begin{aligned}
            &\left(\lim_{k\to\infty}\mathbf{v}\left(\sum_{i=1}^{\ell}D_k^{(i)}\right)\right)\llcorner B_{\rho/2}(x_0)=V\llcorner B_{\rho/2}(x_0),\\
            &\left(\lim_{k\to\infty}\mathbf{v}\left(\sum_{i=\ell+1}^{q_k}D_k^{(i)}\right)\right)\llcorner B_{\rho/2}(x_0)=0.
        \end{aligned}\right.
    \end{align}
    Fix $i\in\{1,\cdots,\ell\}$. Since $\{D_k^{(i)}\}$ is a sequence of embedded discs, almost minimizing among discs with the same boundary, and since for $\rho$ sufficiently small the geodesic sphere $\partial B_{\rho_k}(x_0)$ is uniformly $\F$-convex, \cref{thm:white-improvement}(1) applies in local coordinates. Therefore, after passing to a subsequence, $D_k^{(i)}$ converges in the varifold sense on compact subsets of $B_{\rho/4}(x_0)$ to a properly embedded stable $\F$-minimal surface $M^{(i)}\subset B_{\rho/4}(x_0)$. Grouping together the coincident limits, we obtain
    \[
        V\llcorner B_{\rho/4}(x_0)=\sum_{j=1}^{\ell_0}m_j\,\mathbf v(M^{(j)})
    \]
    for some positive integers $\ell_0$ and $m_j$.

    Since $x_0\in \spt\|V\|$ was arbitrary, and since the local limit surfaces are embedded and satisfy the maximum principle for $\F$-minimal surfaces (see \cite{SW-maximum-principle}), the same patching argument as in \cite[p.~641]{MSY} shows that there exist pairwise disjoint connected smooth embedded $\F$-minimal surfaces $\Sigma^{(1)},\cdots,\Sigma^{(R)}$ and positive integers $n_1,\cdots,n_R$ such that
    \begin{equation}\label{eq:V-decom}
        V=\sum_{i=1}^R n_i\,\mathbf v(\Sigma^{(i)}).
    \end{equation}
    In particular,
    \[
        \mathbf v({\Sigma}_k)\to \sum_{i=1}^R n_i\,\mathbf v(\Sigma^{(i)}),
    \]
    i.e. \eqref{eq:S_k-converges} holds. It remains to prove the genus estimate \eqref{eq:genus-bound} and the stability assertion. This follows closely the remaining part of \cite[Section~5]{MSY}, as the argument is largely topological. For the reader's convenience, we provide a sketch below.

    By \eqref{eq:V-decom}, there exists a sequence $h_k\downarrow 0$ such that $\widetilde{\Sigma}_k$ intersects each level set $\{x\in N:\dist(x,\Sigma^{(i)})=h_k\}$ transversely and such that
    \begin{equation}\label{eq:anis-5-12}
        \begin{aligned}
            \operatorname{length}&\left(\widetilde{\Sigma}_k\cap \left\{x:\dist\left(x,\bigcup_{i=1}^{R}\Sigma^{(i)}\right)=h_k\right\}\right)\\&\qquad+\mathcal H^2\left(\widetilde{\Sigma}_k\cap \left\{x:h_{k}^{-1}>\dist\left(x,\bigcup_{i=1}^{R}\Sigma^{(i)}\right)>h_k\right\}\right)\to 0
        \end{aligned}
    \end{equation}
    as $k\to\infty$. Fix $x_0\in \Sigma^{(i)}$ and choose $\rho>0$ sufficiently small so that $B_\rho(x_0)\cap \Sigma^{(j)}=\emptyset$ $(j\neq i)$, and $\Sigma^{(i)}\cap B_\rho(x_0)$ is a single embedded disc. At this stage, since \eqref{eq:V-decom} is already established and $\Sigma^{(i)}$ is smooth, the local area bounds used in \cite[(5.1)--(5.3)]{MSY} are automatic. Repeating the preceding argument, but now in the thin tubular region
    \[
        B_{\rho_k}(x_0)\cap \left\{x:\dist\left(x,\bigcup_{i=1}^{R}\Sigma^{(i)}\right)<h_k\right\},
    \]
    and using \eqref{eq:anis-5-12} in place of \cite[(5.12)--(5.13)]{MSY}, one obtains exactly as in \cite[(5.14)--(5.19)]{MSY} discs $D_k^{(1)},\cdots,D_k^{(n_i)}$ such that, for all sufficiently large $k$,
    \[
        \partial D_k^{(r)}\subset \partial B_{\rho_k}(x_0),\quad\partial D_k^{(r)} \text{ is not null-homotopic in }\partial B_{\rho_k}(x_0)\cap \{x:\dist(x,\Sigma^{(i)})<h_k\},
    \]
    \[
    \mathbf v(D_k^{(r)})\to \mathbf v(\Sigma^{(i)}\cap B_\rho(x_0))\qquad (r=1,\cdots,n_i),
    \]
    and the remaining part of $\widetilde{\Sigma}_k\cap B_{\rho_k}(x_0)$ has vanishing boundary length and vanishing area. Here one again uses \cref{thm:white-improvement}(1) in place of the isotropic interior regularity theory from \cite[Sections~5--6]{AS}.

    Once this local description is available, the argument on \cite[pp.~643--644]{MSY} is purely topological and carries over verbatim. Namely, one covers $\bigcup_{i=1}^R\Sigma^{(i)}$ by finitely many small balls $B_{\sigma/2}(y_1),\cdots,B_{\sigma/2}(y_P)$, chooses auxiliary vertex balls around the finitely many intersection points of the spherical boundaries, and applies the above local construction first in the vertex balls and then in the balls centered at the $y_m$. This yields, after a further isotopic modification of $\widetilde{\Sigma}_k$, that for all sufficiently large $k$,
    \begin{equation}\label{eq:anis-5-23}
        \widetilde{\Sigma}_k\in \mathcal G\!\left(S_k^{(0)}\cup \bigcup_{i=1}^R S_k^{(i)}\right),
    \end{equation}
    where
    \[
        \mathcal H^2(S_k^{(0)}\cap K)\to 0\qquad\text{for every compact }K\subset N,
    \]
    and
    \[
        S_k^{(i)}=
        \begin{cases}
            \displaystyle \bigcup_{r=1}^{m_i}\left\{x\in N:\dist(x,\Sigma^{(i)})=\frac{r}{k}\right\},& n_i=2m_i,\\[1.2em]\displaystyle \Sigma^{(i)}\cup\bigcup_{r=1}^{m_i}\left\{x\in N:\dist(x,\Sigma^{(i)})=\frac{r}{k}\right\},& n_i=2m_i+1.
        \end{cases}
    \]

    The genus estimate \eqref{eq:genus-bound} now follows exactly as in \cite[p.~644]{MSY}. Indeed, \eqref{eq:anis-5-23} expresses $\widetilde{\Sigma}_k$ as a surface isotopic to a union of parallel sheets around the limit surfaces, together with a remainder of vanishing area. Counting the contributions of the sheets, one obtains
    \[
        \sum_{i\in \mathcal U}\frac12 n_i(g_i-1)+\sum_{i\in \mathcal O}n_i g_i\le \operatorname{genus}(\widetilde{\Sigma}_k),
    \]
    where $\mathcal U$ denotes the one-sided limit surfaces and $\mathcal O$ the two-sided ones. Since $\operatorname{genus}(\widetilde{\Sigma}_k)\le \operatorname{genus}(\Sigma_k)$, this gives \eqref{eq:genus-bound}.

    Finally, the stability assertion is proved exactly as on \cite[pp.~644--645]{MSY}, replacing isotropic area by $\mathbf F$. Let
    \[
        d:=\inf\{\dist(x,y): x\in \Sigma^{(i)},\ y\in \Sigma^{(j)},\ i\neq j\},
    \]
    with the convention $d=\infty$ if $R=1$. Then \eqref{eq:anis-5-23} and the almost minimizing property imply that each family $\{S_k^{(i)}\}$ satisfies
    \[
        \mathbf F(S_k^{(i)})\le\inf\{\mathbf F(\Sigma): \Sigma\in \mathcal G(S_k^{(i)}),\ \Sigma\subset \{x:\dist(x,\Sigma^{(i)})<d/2\}\}+o(1).
    \]
    If $\Sigma^{(i)}$ is two-sided, then this almost minimizing property yields the corresponding second variation inequality, and hence $\Sigma^{(i)}$ is $\F$-stable.
    If, in addition, every $\Sigma_k$ is two-sided and $\Sigma^{(i)}$ is one-sided, then every component of $S_k^{(i)}$ is two-sided, hence $n_i$ must be even. Choosing any connected component $T_k^{(i)}$ of $S_k^{(i)}$, we obtain a two-sided double cover of $\Sigma^{(i)}$ which is still almost minimizing in the tubular neighborhood $\{x:\dist(x,\Sigma^{(i)})<d/2\}$. Therefore the orientable double cover of $\Sigma^{(i)}$ is $\F$-stable, i.e. $\Sigma^{(i)}$ is $\F$-stable in the usual one-sided sense. This proves the final assertion.
\end{proof}

The remainder of this subsection is devoted to the proof of \cref{thm:Meeks-Simon-Yau-with-boundary}. 
Many of the arguments are similar to those used in the proofs of \cref{thm:white-improvement} and \cref{thm:Meeks-Simon-Yau-without-boundary}, with appropriate modifications to the setting of \cref{thm:Meeks-Simon-Yau-with-boundary}.
\begin{proof}[Proof of \cref{thm:Meeks-Simon-Yau-with-boundary}]
    (1) The proof is essentially identical to that of \cref{thm:Meeks-Simon-Yau-without-boundary}, with the modification that isotopies are now taken in the class $\mathcal G(\Sigma_0)$ relative to the fixed boundary $\Gamma$. After passing to a subsequence and relabeling, we may assume that there exists $\gamma_0\in (0,\varepsilon_0^2/9)$ such that $\{\Sigma_k\}$ is a strongly anisotropic $\gamma_0$-irreducible minimizing sequence and $\mathbf{v}(\Sigma_k)\to V$ for some $\mathbf{F}$-stationary varifold $V$. Moreover, this coincides with the varifold limit of the original minimizing sequence.

    Exactly as in the proof of \cref{thm:Meeks-Simon-Yau-without-boundary}, let $\widetilde\Sigma_k$ be obtained from $\Sigma_k$ by deleting all connected components $\Lambda$ such that $\Lambda\subset K_\Lambda$ and $\partial K_\Lambda\cap \Sigma_k=\emptyset$ for some compact $K_\Lambda\subset N$ diffeomorphic to $\mathbf B$. Then $\{\widetilde{\Sigma}_k\}$ is still a minimizing sequence, and $\mathbf{v}(\widetilde\Sigma_k)\to V$ as $k\to\infty$, and, after discarding finitely many terms, we may assume that each $\widetilde\Sigma_k$ is strongly anisotropic $(\gamma_0/2)$-irreducible.

    Fix now a compact set $K\Subset A$. We claim that $V\llcorner K$ is represented by a smooth (possibly disconnected, and with integer multiplicities) $\F$-minimal and $\F$-stable surface in $K$. 
    
    Let $x_0\in \spt\|V\|\cap K$. Since $K\Subset A$, we can choose $\rho_0>0$ so small that $B_{\rho_0}(x_0)\Subset A$, and every geodesic sphere $\partial B_\rho(x_0)$, $0<\rho\leq \rho_0$, is uniformly $\F$-convex. Since $B_\rho(x_0)\Subset A$, the fixed boundary $\Gamma=\partial \Sigma_k$ does not meet $B_\rho(x_0)$. Therefore the local argument in the proof of \cref{thm:Meeks-Simon-Yau-without-boundary} applies verbatim inside $B_\rho(x_0)$, with all isotopies supported away from $\Gamma$. More precisely, repeating exactly the argument from \eqref{eq:anis-5-3} to \eqref{eq:V-decom}, using \cref{prop:anis_thm2}, \cref{prop:replacement-thm-AS-aniso}, \cref{lem:aniso-filigree-lemma}, and \cref{thm:white-improvement}(1), we obtain that
    \[
        V\llcorner B_{\rho/4}(x_0)=\sum_{j=1}^{\ell_0} m_j\, \mathbf{v}(M^{(j)})
    \]
    for some positive integers $m_1,\cdots,m_{\ell_0}$, where $M^{(1)},\cdots,M^{(\ell_0)}\subset B_{\rho/4}(x_0)$ are pairwise disjoint connected smooth embedded $\F$-minimal surfaces. Moreover, each $M^{(j)}$ is $\F$-stable, since it arises as the local varifold limit of embedded discs which are almost minimizing among discs with the same boundary.

    Since $x_0\in \spt\|V\|\cap K$ was arbitrary, a finite covering of $K$ by such balls, together with the same patching argument as in the proof of \cref{thm:Meeks-Simon-Yau-without-boundary} (see also \cite[p.~641]{MSY}), shows that there exist pairwise disjoint connected smooth embedded $\F$-minimal surfaces $\Sigma^{(1)},\cdots,\Sigma^{(R)}$ defined in a neighborhood of $K$, and positive integers $n_1,\cdots,n_R$, such that
    \[
        V\llcorner K=\sum_{i=1}^R n_i\, \mathbf{v}(\Sigma^{(i)})\llcorner K.
    \]
    If we denote by
    \[
        \Sigma:=\sum_{i=1}^R n_i\,\Sigma^{(i)}
    \]
    this weighted union, then $\Sigma$ is smooth in $K$ (possibly with several connected components and integer multiplicities), and $\Sigma_k\to \Sigma$ in the sense of varifolds on $K$. Since each $\Sigma^{(i)}$ is $\F$-minimal and $\F$-stable, it follows that $\Sigma$ is $\F$-minimal and $\F$-stable in $K$. As $K\Subset A$ was arbitrary, the proof of (1) is complete.

    (2) Similar to the proof of \cref{thm:Meeks-Simon-Yau-with-boundary}(1), after passing to a subsequence and relabeling, we may assume that $\{\Sigma_k\}$ is a strongly anisotropic $\gamma_0$-irreducible minimizing sequence in the relative isotopy class $\mathcal G(\Sigma_0)$ fixing boundary $\Gamma$, and $\mathbf{v}(\Sigma_k)\rightharpoonup V$ for some $\F$-stationary varifold $V$ in $A$. Moreover, by \cref{thm:Meeks-Simon-Yau-with-boundary}(1), for every compact set $K\Subset A$, the restriction $V\llcorner K$ is represented by a smooth embedded $\F$-minimal and $\F$-stable surface, possibly disconnected and with integer multiplicities. Thus, only the behavior at the boundary remains to be analyzed.

    First, it's easy to see that $\Gamma\subset \spt\|V\|$. Indeed, if $p\in \Gamma$ and $r>0$ is sufficiently small, then \cref{lem:density-boundary-estimate} gives $\mathcal H^2(\Sigma_k\cap A\cap B_r(p))\ge c r^2$     for all sufficiently large $k$, where $c>0$ is independent of $k$. Passing to the limit yields $\|V\|(A\cap B_r(p))\ge c r^2>0$, hence $p\in \spt\|V\|$.

    We next show that every point of $\Gamma$ is a regular boundary point of $\spt\|V\|$. Fix $p\in \Gamma$. Choose $\rho>0$ so small that $\Gamma\cap B_{4\rho}(p)$ is a single smooth embedded arc, $\partial A$ is strictly $\F$-convex in this neighborhood, and boundary normal coordinates are defined in $B_{4\rho}(p)$. Smoothing the corners of the sets $A\cap B_t(p)$ in these coordinates, we obtain a smooth one-parameter family of rounded domains $U_t$, $t\in(0,2\rho]$ such that each $U_t$ is diffeomorphic to the unit ball, $\partial U_t$ is a smooth closed surface uniformly $\F$-convex for $t\in (0,2\rho]$, $S_t:=\partial U_t\cap \partial A$ is a topological disk, $\beta_t:=\Gamma\cap U_t$ is a single smooth arc contained in $S_t$, and $C_t:=\partial U_t\setminus S_t$ is a smooth cap contained in the interior of $A$.

    For $k$ sufficiently large and for almost every $t\in(0,2\rho)$, the surface $\Sigma_k$ meets $\partial U_t$ transversely. Since $\Sigma_k\subset A$ and $\partial\Sigma_k=\Gamma$, there is exactly one connected component $\alpha_{k}$ of $\Sigma_k\cap \partial U_t$ such that $\alpha_{k}\cap \operatorname{int}(S_t)=\beta_t$ while every other connected component of $\Sigma_k\cap \partial U_t$ is contained in $C_t$. Using coarea in the shell $U_{2\rho}\setminus U_\rho$, exactly as in the proof of \cref{thm:Meeks-Simon-Yau-without-boundary}, we choose $t_k\in(\rho,2\rho)$ such that
    \[
        \operatorname{length}(\Sigma_k\cap \partial U_{t_k})\le C\,\mathcal H^2\bigl(\Sigma_k\cap (U_{2\rho}\setminus U_\rho)\bigr),
    \]
    and therefore, after choosing $\rho$ sufficiently small, the total $\F$-area of the filling disks in $\partial U_{t_k}$ for the boundary components of $\Sigma_k\cap \partial U_{t_k}$ is smaller than $\gamma_0/8$. Since the rounded domains $U_{t_k}$ satisfy the same geometric hypotheses as in \cref{prop:anis_thm2}, the argument of \cref{prop:anis_thm2} applies to $\Sigma_k\cap U_{t_k}$, thus we have that there are pairwise disjoint embedded disks $B_k,\ D_k^{(1)},\cdots,D_k^{(m_k)}\subset \Sigma_k$ such that $\partial B_k=\alpha_{k}$, each $\partial D_k^{(j)}\subset C_{t_k}$, and after discarding components contained in small balls disjoint from $\Gamma$, the interiors of these disks cover $\Sigma_k\cap (U_{t_k}\setminus \partial U_{t_k})$.
    Exactly as in the proof of \cref{thm:Meeks-Simon-Yau-without-boundary}, using \cref{prop:replacement-thm-AS-aniso}, we may replace these disks by disks with the same boundaries which are almost minimizing among disks with the same boundary. For the disks $D_k^{(j)}$, whose boundaries lie in $C_{t_k}$, \cref{thm:white-improvement}(1) applies. For the distinguished boundary disk $B_k$, \cref{thm:white-improvement}(3) applies and shows that the limit support has a component extending smoothly up to the boundary arc, showing that $p$ is a regular boundary point of $\spt\|V\|$. Since $p\in \Gamma$ was arbitrary, every point in $\Gamma$ is a regular boundary point in $\spt\|V\|$.

    Using similar arguments as the proof of \cref{thm:white-improvement}(2), we can show that $\spt\|V\|\cap (\partial A\setminus \Gamma)=\emptyset$. Hence, combining with $\Gamma\subset \spt\|V\|$ we know $\spt\|V\|\cap \partial A=\Gamma$. Therefore, we obtain that $\Sigma:=\spt\|V\|$ is a smooth embedded surface-with-boundary in $A$ with $\partial\Sigma=\Gamma$. Write $\Sigma=\Sigma^{(1)}\cup\cdots\cup \Sigma^{(R)}$ for its connected components. Arguing as in the proof of \cite[Corollary~8.2]{Delellis-Pellandini}, we show that each $\Sigma^{(i)}$ appears in $V$ with multiplicity one and that each $\Sigma^{(i)}$ meets $\Gamma$. Indeed, either $\partial\Sigma^{(i)}=\emptyset$ or $\partial\Sigma^{(i)}$ is the union of some connected components of $\Gamma$. In the latter case, the multiplicity of the boundary forces the multiplicity of $\Sigma^{(i)}$ to be $1$. On the other hand, the case $\partial\Sigma^{(i)}=\emptyset$ cannot occur, otherwise $\Sigma^{(i)}$ would be a smooth embedded $\mathbf{F}$-minimal surface without boundary contained in an $\mathbf{F}$-convex domain of a Riemannian manifold, contradicting the anisotropic maximum principle.

    Since every connected component of $\Sigma$ has nonempty boundary and occurs with multiplicity one, then by attaching some spherical caps to the boundary of each $\Sigma^{(i)}$ and a diffeomorphism we can turn them into a smooth embedded surface in $\R^3$. Then by the Jordan-Brouwer separation theorem \cite{Jordan-Brouwer-thm}, this resulting surface separates $\R^3$ into an inside and an outside, and then by \cite[Theorem 4.5]{Hirsch-DT} we conclude that $\Sigma^{(i)}$ is orientable and two-sided. Moreover, exactly as at the end of the proof of \cref{thm:Meeks-Simon-Yau-without-boundary}, the local almost minimizing property in a tubular neighborhood of each $\Sigma^{(i)}$ yields the corresponding second variation inequality, and therefore each $\Sigma^{(i)}$ is $\F$-stable.

    Once the boundary regularity, the multiplicity-one property and the two-sidedness have been established, the rest of the argument is purely topological and follows the last part of the proof of \cref{thm:Meeks-Simon-Yau-without-boundary}, then we conclude that for all sufficiently large $k$, there is
    \[
        \sum_{i=1}^R \operatorname{genus}(\Sigma^{(i)})\le \operatorname{genus}(\Sigma_k).
    \]
    This completes the proof of (2).

    (3) This is a purely local version of \cref{thm:Meeks-Simon-Yau-with-boundary}(2), so we just sketch it below. Let $S\subset \partial A$ be the relatively open subset where $\partial A$ is strictly $\F$-convex. Fix $p\in \Gamma\cap S$. Repeating exactly the local argument in the proof of \cref{thm:Meeks-Simon-Yau-with-boundary}(2), we choose rounded boundary balls $U_t$ centered at $p$, perform the same local reduction in $U_t$, and obtain a distinguished boundary disk $B_k$ whose boundary contains $\Gamma\cap U_t$. \cref{thm:white-improvement}(3) applies to the corresponding sequence of boundary disks and shows that the limit support extends smoothly up to $\Gamma$ at $p$. The same density lower bound estimate and the $\F$-convexity of $S$ give us that $\spt\|V\|\cap S=\Gamma\cap S$, hence the limit varifold is smooth up to the convex part of the boundary.

    The uniform curvature estimate follows from the proof of \cref{thm:white-improvement}(2) and \cref{thm:white-improvement}(3) in local coordinates, applied to the distinguished boundary disks constructed above. Since (3) is only a local boundary regularity statement, no genus estimate is asserted.
\end{proof}


\section{Anisotropic Simon-Smith min-max theory with genus bounds}\label{sec:min-max}

In this section, we prove the anisotropic Simon--Smith--Ketover theorem stated in \cref{thm:aniso-simon-smith-ketover}. More precisely, we first formulate the Simon--Smith min--max construction for the anisotropic energy $\mathbf F$ and then extend Ketover's optimal genus bound to the anisotropic setting.

The local ingredients required for the min--max construction have already been established in the preceding sections. In particular, the necessary anisotropic replacement theory is provided by
\cref{thm:white-improvement}, \cref{thm:Meeks-Simon-Yau-without-boundary}, and \cref{thm:Meeks-Simon-Yau-with-boundary},
proved in \cref{sec:aniso-MSY}. These results constitute the anisotropic analogue of the Meeks--Simon--Yau theory underlying the classical Simon--Smith construction.

The global min--max argument follows the same general scheme as in the isotropic theories of Simon--Smith \cite{smith} and Colding--De Lellis \cite{ColdingDeLellisMinMax}, together with the anisotropic adaptations introduced in \cite{DePhilippisDeRosa}. The present setting is somewhat simpler than the constant anisotropic mean curvature setting considered in \cite{DePhilippisDeRosa}, since we work with the pure anisotropic energy functional, without an enclosed-volume constraint and hence with zero prescribed anisotropic mean curvature.

Having established the anisotropic Simon--Smith min--max theorem, we then turn to the topology of the limiting surface. Using the same replacement and surgery framework, together with an anisotropic version of Ketover's lifting argument, we prove the optimal genus bound for the min--max limit. The topological part of Ketover's proof carries over verbatim. The only new analytic ingredient is an anisotropic no-folding estimate in the overlap region between consecutive replacements, where the minimal surface equation $H=0$ is replaced by the quasilinear elliptic equation satisfied by $\mathbf F$-stationary surfaces.

We record only those statements needed for the proof of \cref{thm:aniso-simon-smith-ketover}. Whenever an argument is identical to its isotropic counterpart after replacing the area functional by $\mathbf F$, we simply indicate the relevant reference and point out where the anisotropic input enters.

\subsection{Generalized family and width}\label{sec:generalized-family-width}

We begin by recalling the notions of generalized smooth families and width used in the Simon--Smith theory. We keep the formulation close to \cite{ColdingDeLellisMinMax,Delellis-Pellandini}.

\begin{definition}[Generalized family]
A \emph{generalized family of surfaces} in $N$ is a one-parameter family $\{\Sigma_t\}_{t\in[0,1]}$ of closed subsets of $N$ for which there exists a finite set $T\subset [0,1]$ such that:
\begin{enumerate}
    \item if $t\notin T$, then $\Sigma_t$ is a smooth embedded closed surface;
    \item if $t\in T$, then there exists a finite set $P_t\subset N$ such that $\Sigma_t\setminus P_t$ is a smooth embedded surface;
    \item the map $t\mapsto \F(\Sigma_t)$ is continuous and $\Sigma_t\to \Sigma_{t_0}$ in the Hausdorff topology whenever $t\to t_0$.
\end{enumerate}
\end{definition}

\begin{definition}[Saturated family and width]
Let $\Lambda$ be a collection of generalized families. We say that $\Lambda$ is \emph{saturated} if for every $\{\Sigma_t\}\in \Lambda$ and every smooth map
\[
\Psi \in C^{\infty}([0,1]\times N,N),
\qquad
\Psi(t,\cdot)\in \mathrm{Diff}_0(N)\quad \forall\, t\in[0,1],
\]
the deformed family $\{\Psi(t,\Sigma_t)\}_{t\in[0,1]}$ still belongs to $\Lambda$. 

We will also tacitly assume the standard bounded-singularity condition from the Simon--Smith theory: for every saturated family $\Lambda$ under consideration there exists $C(\Lambda)<\infty$ such that, for every sweepout in $\Lambda$, each singular set $P_t$ contains at most $C(\Lambda)$ points.

For every $\{\Sigma_t\}\in \Lambda$ we define
\[
\mathcal{F}(\{\Sigma_t\}) := \max_{t\in[0,1]} \F(\Sigma_t),
\]
and the \emph{anisotropic width} of $\Lambda$ is
\[
m_0(\Lambda):=\inf_{\Lambda}\, \mathcal{F}=\inf_{\{\Sigma_t\}\in \Lambda}\left[\max_{t\in[0,1]} \F(\Sigma_t)\right].
\]
\end{definition}
A sequence $\{\Sigma_t^j\}\subset \Lambda$ is called \emph{minimizing} if
\[
\lim_{j\to\infty}\max_{t\in[0,1]}\F(\Sigma_t^j)=m_0(\Lambda),
\]
and a sequence $\Sigma^j:=\Sigma_{t_j}^j$ is called a \emph{min-max sequence} if $\{\Sigma_t^j\}$ is minimizing and
\[
\lim_{j\to\infty}\F(\Sigma_{t_j}^j)=m_0(\Lambda).
\]

\begin{proposition}[Positivity of the width]\label{prop:positive-width-aniso}
Let $N$ be a closed 3-manifold with Riemannian metric and let $\{\Sigma_t\}$ be the level-set family of a Morse function. Then the smallest saturated family containing $\{\Sigma_t\}$ has $m_0(\Lambda)>0$.
\end{proposition}

\begin{proof}
This is exactly the argument of \cite[Proposition~1.4]{ColdingDeLellisMinMax} or \cite[Proposition 3.1]{DePhilippisDeRosa}. Since $\F$ and $\mathcal H^2$ are uniformly comparable by \eqref{H:comparability}, positivity of the isotropic width immediately implies positivity of the anisotropic width.
\end{proof}

\subsection{Pull-tight and almost minimizing min-max sequences}

The pull-tight construction is the same as in \cite[Section~4 and Errata]{ColdingDeLellisMinMax} and \cite[Section~4.1]{DePhilippisDeRosa}: one deforms non-stationary slices by a pseudo-gradient for the first variation of the functional. In the present setting the functional is simply $\F$, so the argument is formally simpler than in \cite[Section~4.1]{DePhilippisDeRosa}, where one must also keep track of the volume term.

\begin{proposition}[Pull-tight]\label{prop:aniso-pull-tight}
For every saturated family $\Lambda$ there exists a minimizing sequence $\{\Sigma_t^j\}\subset \Lambda$ such that every associated min-max sequence $\Sigma_{t_j}^j$ converges in the varifold sense, up to subsequences, to a $\F$-stationary varifold $V\in \mathbf{V}_2(N)$.
\end{proposition}

\begin{proof}
The proof is the same as in \cite[Proposition~4.1 and Errata]{ColdingDeLellisMinMax}; see also \cite[Proposition~4.1]{DePhilippisDeRosa} for the anisotropic variant with prescribed mean curvature. The only modification is that the isotropic first variation is replaced everywhere by the anisotropic first variation $\delta_{\F}$.
\end{proof}

As above, $\Lambda$ is a fixed saturated set of $1$–parameter families $\{\Sigma_{t}\}$ in $N$. In \cref{prop:aniso-pull-tight} we showed that there exists a family
$\{\Sigma_{t}\}$ such that every min-max sequence clusters towards $\F$-stationary varifolds, but they are not necessarily regular. Hence, in order to prove regularity we introduce the notion of almost minimizing varifolds as \cite[Definition 3.2]{ColdingDeLellisMinMax}.

\begin{definition}[Almost minimizing]
Let $U\subset N$ be open and let $\eps>0$. A surface $\Sigma\subset N$ is said to be \emph{$\eps$-almost minimizing} in $U$ if there \emph{does not} exist any isotopy $\psi\in C^{\infty}([0,1]\times N,N)$ supported in $U$ such that
\begin{align}
\F(\psi(t,\Sigma))&\leq \F(\Sigma)+\eps/8, \quad \forall\, t\in[0,1],\\
\F(\psi(1,\Sigma))&\leq \F(\Sigma)-\eps.
\end{align}
A sequence $\{\Sigma^j\}$ is called \emph{almost minimizing} in $U$ if each $\Sigma^j$ is $\eps_j$-almost minimizing in $U$ for some $\eps_j\downarrow 0$.
\end{definition}

For $x\in N$ and $0<s<t$ we set
\[
\operatorname{An}(x;s,t):=B_t(x)\setminus \overline{B_s(x)},
\qquad
\mathcal{AN}_r(x):=\bigl\{\operatorname{An}(x;s,t):0<s<t<r\bigr\}.
\]
The next proposition is the anisotropic analogue of the almost minimizing selection principle in the Simon--Smith theory, with the strengthened trichotomy from \cite[Proposition~4.3]{DePhilippisDeRosa}.

\begin{proposition}[Almost minimizing min-max sequence]\label{prop:aniso-am-minmax}
There exist a min-max sequence $\Sigma^j=\Sigma_{t_j}^j$ as in \cref{prop:aniso-pull-tight}, converging to a $\F$-stationary varifold $V$, and a function $r:N\to(0,\infty]$ such that for every $x\in N$ and every annulus $\operatorname{An}(x;s,t)\in \mathcal{AN}_{r(x)}(x)$ there exists a non-relabeled subsequence $\Sigma^j$ which is almost minimizing in $\operatorname{An}(x;s,t)$. Moreover, one of the following alternatives holds:
\begin{enumerate}
    \item there exists $R\in (0,\operatorname{Inj}(N)/18)$ such that $r\equiv R$ on $N$, and there exists $y\in N$ such that the sequence is almost minimizing in $N\setminus B_{18R}(y)$;
    \item $r\equiv \operatorname{Inj}(N)/18$ on $N$;
    \item there exists $p\in N$ such that $r(x)=\dist(x,p)$ for every $x\neq p$ and $r(p)=\infty$.
\end{enumerate}
\end{proposition}
\begin{proof}
The proof is the same combinatorial argument as in \cite[Section~5]{ColdingDeLellisMinMax}, combined with the strengthened covering alternative in \cite[Proposition~4.3]{DePhilippisDeRosa}. In the present setting one works with generalized families of embedded surfaces rather than boundaries of Caccioppoli sets, and the interpolation step is exactly the one in \cite[Section~5]{ColdingDeLellisMinMax}. Since no volume constraint is present, the anisotropic changes are, in fact, slightly simpler than in \cite{DePhilippisDeRosa}.
\end{proof}

\subsection{Replacements and regularity: smooth gluing via induction}

Next, we construct replacements for the $\F$-stationary limit $V$. We use the same definition as in \cite[Definition~6.1]{ColdingDeLellisMinMax}.

\begin{definition}[Replacement]
Let $V\in \mathbf{V}_2(N)$ be 
$\F$-stationary, and let $U\subset N$ be an open set. A $\F$-stationary varifold $V'\in \mathbf{V}_2(N)$ is called a \emph{replacement} for $V$ in $U$ if
\[
V'=V \quad \text{on } G_2(N\setminus \overline{U}),
\quad
\F(V')=\F(V),
\]
and $V'\llcorner G_2(U)$ is induced by a smooth embedded $\F$-stationary and $\F$-stable surface $\Sigma$ with $\overline{\Sigma}\setminus\Sigma\subset\partial U$.
\end{definition}

\begin{proposition}[Existence of replacements]\label{prop:aniso-minmax-replacements}
Let $V$ be the $\F$-stationary limit given by \cref{prop:aniso-am-minmax}. Then for every $x\in N$ and every annulus $\operatorname{An}(x;s,t)\in \mathcal{AN}_{r(x)}(x)$, the varifold $V$ admits a replacement in $\operatorname{An}(x;s,t)$. Moreover, the same conclusion holds for every further replacement, so that $V$ has the good replacement property in the sense of \cite[Definition~6.2]{ColdingDeLellisMinMax}.
\end{proposition}

\begin{proof}
The construction of replacements follows closely the squeezing lemma in \cite[Section~7.2]{ColdingDeLellisMinMax}; see also \cite[Section~4.4]{DePhilippisDeRosa} for the corresponding anisotropic CMC argument. For the reader's convenience, we sketch the main ideas below.

Fix an annulus $\operatorname{An}\in \mathcal{AN}_{r(x)}(x)$
and choose the almost minimizing subsequence given by \cref{prop:aniso-am-minmax}. As in \cite[Section~7.2]{ColdingDeLellisMinMax}, one considers the local isotopy-minimization problem inside $\operatorname{An}$: namely, among all isotopies supported in that domain whose intermediate slices do not increase the $\F$-energy by more than the allowed almost minimizing threshold, one minimizes the final value of $\F$. Up to this point the argument is purely formal and is identical to the isotropic one. The only place where \cite{ColdingDeLellisMinMax} uses the isotropic Meeks--Simon--Yau theory is in upgrading such a local isotopy-minimizing sequence to a smooth embedded stable replacement. In the present setting this step is replaced by the anisotropic Meeks--Simon--Yau theorem proved in \cref{sec:aniso-MSY}, namely \cref{thm:Meeks-Simon-Yau-without-boundary} or 
\cref{thm:Meeks-Simon-Yau-with-boundary}. These results produce, from the local isotopy-minimizing sequence, a smooth embedded $\F$-stationary and $\F$-stable surface which agrees with the original sequence outside the chosen domain and therefore yields the required replacement.

Passing to the limit exactly as in \cite[Section~7.2]{ColdingDeLellisMinMax} and \cite[Section~4.4]{DePhilippisDeRosa}, one obtains a replacement of $V$ in $\operatorname{An}$. The argument for further replacements is the same: one repeats the same local isotopy-minimization procedure for the first replacement and again invokes \cref{thm:Meeks-Simon-Yau-without-boundary} or 
\cref{thm:Meeks-Simon-Yau-with-boundary} in place of the isotropic Meeks--Simon--Yau theorem. Hence $V$ has the good replacement property.
\end{proof}

In the following, we use \cref{prop:aniso-minmax-replacements} to obtain \cref{thm:aniso-simon-smith-one-singularity}.
\begin{thm}[Regularity away from at most one point]
\label{thm:aniso-regularity}
Let $V$ be as in \cref{prop:aniso-am-minmax} and, consequently, as in
\cref{prop:aniso-minmax-replacements}. Then there exists a set
$\mathcal S\subset N$ with $\#\mathcal S\leq 1$ such that
$\operatorname{spt}V\setminus \mathcal S$ is a smooth embedded
$\F$-minimal surface. More precisely, there exist positive integers
$n_1,\cdots,n_L$ and pairwise disjoint connected smooth embedded
$\F$-minimal surfaces
\[
    \Sigma^{(1)},\cdots,\Sigma^{(L)} \subset N\setminus \mathcal S
\]
which are locally $\F$-stable in $N\setminus\mathcal S$, such that
\[
    V=\sum_{i=1}^L n_i\,\mathbf{v}(\Sigma^{(i)}).
\]
Moreover,
\[
    \overline{\Sigma^{(i)}}\setminus \Sigma^{(i)} \subset \mathcal S
    \qquad\text{for each } i=1,\cdots,L .
\]
\end{thm}
\begin{proof}
We follow the proof of \cite[Proposition~6.3]{ColdingDeLellisMinMax};
see also \cite[Proposition~4.14]{DePhilippisDeRosa} for comparison.
Once \cref{prop:aniso-minmax-replacements} is established, all steps
proceed as in those references, except for the analogue of Step~2 in the
proof of \cite[Proposition~6.3]{ColdingDeLellisMinMax}, namely the
smooth gluing of two consecutive replacements across the intermediate
sphere in the anisotropic setting.

Roughly speaking, our approach to the smooth gluing combines the thin
convex-annulus construction from
\cite[Proof of Theorem~5.1, Step~1, Case~1]{DDL} with an induction on
the multiplicity of the first replacement. The key point is that, after
inserting a third replacement in a convex annulus, one can adapt the
one-sided wedge argument in \cite[Lemmas~5.2--5.3]{DDL} to show that the
multiplicity at the auxiliary point $z$ cannot exceed that of the first
replacement. Once this is established, a nested fourth-replacement
argument around $z$, together with the induction hypothesis and the
interior regularity of the replacements, forces the third replacement to
have only one branch near $z$. The remainder of the argument then follows
the same propagation mechanism as in
\cite[Proof of Theorem~5.1, Step~1, Case~1]{DDL}. We refer to
\cref{lem:aniso-smooth-gluing} for the details.

Accordingly, we only sketch the proof of this theorem. Given any two
consecutive replacements as in \cref{lem:aniso-smooth-gluing}, the lemma
shows that they glue smoothly across the intermediate sphere. This gives
the anisotropic analogue of Step~2 in
\cite[Proof of Proposition~6.3]{ColdingDeLellisMinMax}. The rest of the
proof then proceeds as in \cite[Section~6]{ColdingDeLellisMinMax}: by
\cref{prop:aniso-minmax-replacements}, the varifold $V$ admits good
replacements in all admissible annuli, and the regularity theory together
with \cite[Lemma~4.15]{DePhilippisDeRosa} implies that $V$ is integral
and smooth in the corresponding punctured balls.

The trichotomy in \cref{prop:aniso-am-minmax} allows one to absorb all
punctures by overlapping balls, except possibly for a single point.
Therefore there exists a set $\mathcal S\subset N$ with
$\#\mathcal S\leq 1$ such that $\operatorname{spt}V\setminus\mathcal S$
is a smooth embedded $\F$-minimal surface. Since the local pieces are
obtained from stable replacements, this regular part is locally
$\F$-stable in $N\setminus\mathcal S$.

Finally, decomposing $\operatorname{spt}V\setminus\mathcal S$ into its
connected components gives pairwise disjoint connected smooth embedded
$\F$-minimal surfaces
\[
    \Sigma^{(1)},\cdots,\Sigma^{(L)}\subset N\setminus\mathcal S .
\]
The integrality of $V$ gives positive integer multiplicities
$n_1,\cdots,n_L$, and hence
\[
    V=\sum_{i=1}^L n_i\,\mathbf v(\Sigma^{(i)}).
\]
Moreover, the only possible failure of smooth extension occurs at points
of $\mathcal S$, so
\[
    \overline{\Sigma^{(i)}}\setminus\Sigma^{(i)}\subset \mathcal S
    \qquad\text{for each } i=1,\cdots,L .
\]
This proves the theorem.
\end{proof}

\begin{lemma}[Smooth gluing across the intermediate sphere]
\label{lem:aniso-smooth-gluing}
Fix $x\in \spt\|V\|$, where $V$ is as in \cref{prop:aniso-am-minmax}. Choose $2\rho\le r(x)$ so that the replacement
construction of \cref{prop:aniso-minmax-replacements} is available in all
admissible annuli centered at $x$. Let $V'$ be a replacement of $V$ in
$\operatorname{An}(x;\rho,2\rho)$, and let $\Sigma'$ be the smooth stable $\F$-minimal surface
induced by $V'$ there. Choose $t\in (\rho,2\rho)$ such that $\Sigma'$ intersects
$\partial B_t(x)$ transversely. For $s<\rho$, let $V''$ be a replacement of $V'$ in
$\operatorname{An}(x;s,t)$, and let $\Sigma''$ be the corresponding smooth stable $\F$-minimal
surface in this annulus. Then for every $y\in \Sigma'\cap \partial B_t(x)$ there
exists $r>0$ such that
\[
\Sigma''\cap B_t(x)\cap B_r(y)=\Sigma'\cap B_t(x)\cap B_r(y).
\]
\end{lemma}

\begin{proof}
Fix $y\in \Sigma'\cap \partial B_t(x)$. After passing to exponential coordinates and
using the local nature of the statement, we may work in $\mathbb R^3$ exactly as in
\cite[Proof of Theorem~5.1, Step~1]{DDL}. Let $m\ge 1$ be the multiplicity of the
first replacement near $y$, so that locally
\[
V'=m\,\mathbf{v}(\Sigma').
\]
We prove the conclusion by induction on $m$.

\smallskip
\noindent\textbf{Base case: $m=1$.}
This is precisely the multiplicity-one situation treated in
\cite[Proof of Theorem~5.1, Step~1, Case~1]{DDL}. One chooses an auxiliary point
$z\in \Sigma'\setminus \operatorname{An}(x;s,t)$ close to $y$, constructs a thin convex annulus
$C$ tangent to $\Sigma'$ at $z$, and takes a third replacement $V'''$ in $C$. The
one-sided wedge argument in \cite[Lemmas~5.2--5.3]{DDL} shows that a blow-up of
$V'''$ at $z$ is the plane $T_z\Sigma'$ with multiplicity one. The graphicality
argument in \cite[Proof of Proposition~4.14, Step~1, Case~2]{DePhilippisDeRosa}
then implies that $V'''$ and $V'$ glue smoothly at $z$. Propagating this
information from $z$ back to $y$ as in
\cite[Proof of Theorem~5.1, Step~1, Case~1]{DDL}, one gets $TV(y,V'')=\{\mathbf{v}(T_y\Sigma')\}$,
and the same graphicality argument yields smooth gluing of $V''$ and $V'$ across
$\partial B_t(x)$ at $y$.

\smallskip
\noindent\textbf{Induction step.}
Assume the statement has been proved whenever the multiplicity of the first
replacement is at most $m-1$, and suppose that near $y$ the first replacement has
multiplicity exactly $m$.

Fix $0<\varepsilon\ll 1$. Since $y$ is a regular point of $\Sigma'$, after
shrinking the neighborhood we may assume that
\[
\dist(T_z\Sigma',T_y\Sigma')<\varepsilon
\qquad\text{for every }z\in \Sigma'\cap B_r(y).
\]
Choose
\[
z\in \Sigma'\cap B_r(y)\setminus \operatorname{An}(x;s,t).
\]
Following \cite[Proof of Theorem~5.1, Step~1, Case~1]{DDL}, let $\widetilde C$ be
the convex domain bounded by two spherical caps with common boundary circle $S$,
chosen so that $z\in S$, $T_zS=T_z(\Sigma'\cap\partial\widetilde C)$, and the two
caps meet at angle $3\varepsilon$. Then
\[
\Sigma'\cap B_r(y)\setminus \operatorname{An}(x;s,t)\subset \widetilde C,
\]
and $T_z\widetilde C$ is a wedge of opening angle $3\varepsilon$. Set
\[
C:=\Big((\widetilde C-y)\setminus \frac12(\widetilde C-y)\Big)+y,
\]
and let $V'''$ be a replacement of $V''$ in $C$.

We first bound the multiplicity of $V'''$ at $z$.

\smallskip
\noindent\emph{Claim 1.}
If the opening of $C$ is sufficiently small, then every tangent varifold of $V'''$
at $z$ has total multiplicity at most $m$.

\smallskip
\noindent\emph{Proof of Claim 1.}
Since $z\notin \operatorname{An}(x;s,t)$, we have
\[
V'''=V''=V'=m\,\mathbf{v}(\Sigma')
\qquad\text{near }z\text{ outside }C.
\]
After translating $z$ to the origin and rotating coordinates, we may assume
\[
P:=T_z\Sigma'=\{p_3=0\},\qquad
\ell:=T_zS=\{p_1=p_3=0\}=\operatorname{span}(e_2),
\]
and
\[
T_zC=\{|p_3|\le \tan(3\varepsilon)\,p_1\}.
\]
Exactly as in \cite[Lemma~5.2]{DDL}, any tangent varifold $W\in TV(z,V''')$
satisfies
\[
W\llcorner \{p_1\le 0\}=m\,\mathbf{v}(P)\llcorner \{p_1\le 0\},
\qquad
\spt(W)\cap \{p_1\ge 0\}\subset \{|p_3|\le \tan(3\varepsilon)\,p_1\}.
\]

Because $V'''$ is induced by a smooth $\F$-minimal surface in $C$,
$\spt(W)\cap\{p_1\ge 0\}$ is a union of half-planes. Moreover, since
$G_z(\nu):=G(z,\nu)$ is smooth and uniformly elliptic, there exists
$\theta_0=\theta_0(G_z)>0$ such that any nontrivial stationary junction for the
induced one-dimensional anisotropic length on $\partial B_1$ contains two incident
directions making an angle at least $\theta_0$. Hence, for $\varepsilon$ small
enough, no interior junction can occur in $\partial B_1$, and therefore
\[
W\llcorner \{p_1>0\}=\sum_{j=1}^N q_j\,\mathbf{v}(K_j),
\]
where $q_j\in\mathbb N$ and each $K_j$ is a $G_z$-stationary cone sector with
boundary line $\ell$.

For each \(K_j\), let \(\beta_{G_z}(K_j)\in \ell^\perp\) denote the
boundary-flux vector along \(\ell\), defined with respect to the inward
conormal of \(K_j\) along \(\ell\), so that
\[
\delta_{G_z}\mathbf{v}(K_j)(X)
=
-\int_\ell X\cdot \beta_{G_z}(K_j)\,\mathrm{d}\mathcal H^1
\]
for every compactly supported vector field \(X\); compare with
\cite[(10)]{DePhilippisDeRosaHirsch}. Since $W$ is $G_z$-stationary, the boundary
contributions along $\ell$ cancel:
\[
m\,\beta_{G_z}(P^-)+\sum_{j=1}^N q_j\,\beta_{G_z}(K_j)=0,
\qquad
P^-:=P\cap\{p_1\le 0\}.
\]
Taking scalar product with $e_1$, and writing
\[
b_0:=-\beta_{G_z}(P^-)\cdot e_1>0,
\qquad
b_j:=\beta_{G_z}(K_j)\cdot e_1,
\]
we obtain
\[
m\,b_0=\sum_{j=1}^N q_j\,b_j.
\]

As $\varepsilon\to 0$, every admissible sector $K_j$ converges, after passing to a
subsequence, to the flat half-plane
\[
P^+:=P\cap\{p_1\ge 0\},
\]
and the corresponding boundary-flux vectors converge as well; in particular,
$b_j\to b_0$ uniformly in $j$. Since $m$ is fixed, we may therefore choose
$\varepsilon=\varepsilon(m,G_z)>0$ so small that
\[
b_j>\frac{m}{m+1}\,b_0
\qquad\text{for every }j.
\]
Substituting this into the flux identity yields
\[
m\,b_0=\sum_{j=1}^N q_j\,b_j
>
\frac{m}{m+1}\,b_0\sum_{j=1}^N q_j,
\]
hence
\[
\sum_{j=1}^N q_j<m+1.
\]
Since the left-hand side is an integer, it follows that
\[
\sum_{j=1}^N q_j\le m.
\]
This proves Claim~1.
\qed

\smallskip
We next show that only one branch of $V'''$ enters the convex annulus $C$.

\smallskip
\noindent\emph{Claim 2.}
Near $z$, the varifold $V'''$ has a single smooth branch inside $C$.

\smallskip
\noindent\emph{Proof of Claim 2.}
In the above coordinates, the exterior side of $C$ carries the multiplicity-$m$
sheet
\[
V'''=V''=V'=m\,\mathbf{v}(\Sigma')
\qquad\text{near }z\text{ on }\{p_1<0\}.
\]
After shrinking the neighborhood if necessary, we may write
\[
V'''\llcorner (B_r(z)\cap C)=\sum_{i=1}^q m_i\,\mathbf{v}(\Gamma_i),
\]
where the $\Gamma_i$ are pairwise disjoint connected smooth embedded $\F$-minimal
surfaces and, by Claim~1,
\[
\sum_{i=1}^q m_i\le m.
\]
We claim that $q=1$.
Assume instead that $q\ge 2$. Then necessarily $m_i\le m-1$ for every $i$. Set
\[
\Gamma^-:=\Sigma'\cap B_r(z)\setminus C,
\qquad
\gamma:=\overline{\Gamma^-}\cap \overline C.
\]
Thus $\Gamma^-$ is the exterior multiplicity-$m$ sheet and $\gamma$ is its trace on
the interface with $C$.

Fix $\rho\in(0,r)$ sufficiently small so that the annulus $\operatorname{An}(z;\rho/2,\rho)$ is
admissible for replacements, and let $V^{(4)}_\rho$ be a replacement of $V'''$ in
$\operatorname{An}(z;\rho/2,\rho)$. We may also assume that each $\Gamma_i$ meets
$\partial B_\rho(z)\cap C$ transversely. Define
\[
\gamma_\rho:=\Gamma^-\cap \partial B_\rho(z),
\qquad
\Lambda_\rho:=\bigl(\partial B_\rho(z)\cap C\bigr)\setminus \gamma_\rho.
\]

Fix $x\in \Lambda_\rho\cap \spt\|V'''\|$. Then $x$ lies on one of the branches
$\Gamma_i$, whose multiplicity satisfies $m_i\le m-1$. Since $x\notin \gamma_\rho$,
the only local sheet of $V'''$ meeting $\partial B_\rho(z)$ at $x$ is
$m_i\,\mathbf{v}(\Gamma_i)$. The induction hypothesis therefore applies to this
branch and the corresponding local piece of $V^{(4)}_\rho$, so there exist an open
neighborhood $U_x\subset \Lambda_\rho$ of $x$ and a unique connected smooth embedded
$\F$-minimal surface $M_x\subset \spt\|V^{(4)}_\rho\|$ such that $M_x$ and
$\Gamma_i$ glue smoothly across $x$.

Since this holds for every $x\in \Lambda_\rho\cap \spt\|V'''\|$ and every
sufficiently small $\rho$, we obtain, for each small $\eta>0$, a smooth
$\F$-minimal surface $\Sigma_\eta$ which coincides with $V'''$ in
\[
\Omega_\eta:=B_r(z)\setminus \mathcal N_\eta(\Gamma^-\cup\gamma),
\]
where $\mathcal N_\eta(\Gamma^-\cup\gamma)$ denotes the connected component of the
$\eta$-tubular neighborhood of $\Gamma^-\cup\gamma$ containing $\Gamma^-\cup\gamma$.
The domains $\Omega_\eta$ exhaust $B_r(z)\setminus(\Gamma^-\cup\gamma)$ as
$\eta\to 0$, and each $\Sigma_\eta$ is smooth in the interior by regularity of
replacements. Letting $\eta\to 0$, we conclude that every branch $\Gamma_i$
extends smoothly up to $\gamma\setminus\{z\}$.

This forces two distinct smooth branches to touch along a nontrivial segment of
$\gamma$ away from $z$, contradicting the strong maximum principle for anisotropic
minimal surfaces. Hence $q=1$, proving Claim~2.
\qed

\smallskip
By Claim~2, there exist a single smooth branch $\Sigma^{'''}$ near $z$ and an
integer $m_*\le m$ such that
\[
V'''=m_*\,\mathbf{v}(\Sigma^{'''})
\qquad\text{near }z.
\]
On the exterior side of $C$ we already have
\[
V'''=V''=V'=m\,\mathbf{v}(\Sigma').
\]
Therefore the connected surface $\Sigma^{'''}$ contains an open subset of $\Sigma'$
outside $C$. By the constancy theorem, the multiplicity along this connected smooth
branch is constant, hence $m_*=m$ and
\[
\Sigma^{'''}=\Sigma'
\qquad\text{near the exterior side of }C.
\]
In particular, every tangent varifold of $V'''$ at $z$ is exactly
\[
m\,\mathbf{v}(T_z\Sigma').
\]
The graphicality argument in
\cite[Proof of Proposition~4.14, Step~1, Case~2]{DePhilippisDeRosa} then shows that
$V'''$ and $V'$ glue smoothly at $z$.

Finally, one propagates this information from $z$ back to $y$ exactly as in
\cite[Proof of Theorem~5.1, Step~1, Case~1]{DDL}, obtaining
\[
TV(y,V'')=\{m\,\mathbf{v}(T_y\Sigma')\}.
\]
Once the tangent varifold of $V''$ at $y$ is known to be the multiplicity-$m$ plane
$T_y\Sigma'$, the same graphicality argument yields smooth gluing of $V''$ and
$V'$ across $\partial B_t(x)$ at $y$. This completes the induction and the proof.
\end{proof}

Combining
\cref{prop:aniso-pull-tight}, \cref{prop:aniso-am-minmax}, \cref{prop:aniso-minmax-replacements}, and \cref{thm:aniso-regularity},
we first obtain the anisotropic Simon--Smith min-max theorem with at most
one possible singular point. The removable singularity results proved in
\cref{sec:removable-singularity-bernstein} or \cref{sec:remove-singularity-separation}
then allow us to remove this possible singularity under either
\eqref{assumption1} or \eqref{assumption2}. Consequently, we obtain the
following full regularity statement.

\begin{thm}[Anisotropic Simon--Smith min-max]
\label{thm:aniso-simon-smith}
Assume that $F$ satisfies either \eqref{assumption1} or
\eqref{assumption2}. Let $N$ be a closed $3$-manifold with a Riemannian
metric. For any saturated set of generalized families of surfaces
$\Lambda$, there exists a min-max sequence obtained from $\Lambda$ that
converges in the sense of varifolds to a smooth embedded
$\F$-minimal surface with anisotropic width $m_0(\Lambda)$
(multiplicity is allowed).
\end{thm}

\subsection{Genus bound}

Throughout this subsection, we assume that $F$ satisfies either
\eqref{assumption1} or \eqref{assumption2}. Hence, by
\cref{thm:aniso-simon-smith-one-singularity} together with the removable
singularity results in
\cref{sec:removable-singularity-bernstein,sec:remove-singularity-separation},
the possible isolated singular point is removable. In particular, the
min--max varifold obtained from the anisotropic Simon--Smith construction
is induced by a smooth embedded closed $\F$-minimal surface, possibly with
integer multiplicities.

Here we prove the optimal genus bound for the limiting surfaces produced
by this fully regular anisotropic Simon--Smith construction. We follow
Ketover's proof of the improved lifting lemma
\cite[Proposition~2.2 and Section~4]{K}. The topological part of the
argument is nearly unchanged. The only analytic point which requires
modification is the no-folding estimate in the overlap wedge between two
consecutive replacements. In the isotropic case this estimate uses the
identity \(H=0\); in the anisotropic case it is replaced by the elliptic
relation coming from \(\F\)-stationarity.

For a smooth embedded surface \(\Sigma\), we denote by \(T_\epsilon(\Sigma)\) its tubular
\(\epsilon\)-neighborhood, and by
\[
    \pi:T_\epsilon(\Sigma)\to \Sigma
\]
the nearest point projection, whenever \(\epsilon>0\) is sufficiently small.
The projection in item (3) is
always understood with respect to the tubular neighborhood of the limiting component
\(\Sigma^{(\ell)}\).

\begin{proposition}[Optimal lifting lemma]\label{prop:optimal-curve-lifting}
    Let $\Sigma_j$ be a $1/j$-almost $\F$-minimizing min-max sequence of surfaces in $N$ arising from \cref{thm:aniso-simon-smith} such that $\Sigma_j\to \sum_{k=1}^{m} n_k\Sigma^{(k)}$, where $\Sigma^{(k)}$ are embedded disjoint minimal surfaces. Let $\{\gamma_i\}_{i=1}^{n}$ be a collection of simple closed curves on $\Sigma:= \cup_{k=1}^{m} \Sigma^{(k)}$ such that each $\Sigma^{(k)}$ containing at least one curve, also contains a point $p_k$ where any two curves $\gamma_r,\gamma_s\in\Sigma^{(k)}$ share exactly the point $p_k$, namely $\gamma_r\cap \gamma_s = \{p_k\}$. Then there exists $\epsilon_0(\Sigma)>0$ such that for all $\epsilon < \epsilon_0$ there exists curves $\{\tilde{\gamma}_i\}$ in $\Sigma$ and a subsequence $\Sigma_j$ (not relabeled) and surfaces $\tilde{\Sigma}_j$ given by finitely many $\gamma$-reductions (neck pinch surgeries) from $\Sigma_j$ with the following properties:
    \begin{enumerate}
        \item Each $\tilde\gamma_i$ is homotopic to $\gamma_i$ and lies in $T_\epsilon(\gamma_i)$.
        \item $\tilde{\Sigma}_j\to \sum_{k=1}^{m}n_k\Sigma^{(k)}$ as varifolds.
        \item For each $1\leq i\leq n$, if $\tilde{\gamma}_i \subset\Sigma^{(\ell)}$ we have $\pi^{-1}(\tilde{\gamma}_i)\cap T_\epsilon(\Sigma^{(\ell)})\cap \tilde\Sigma_j$ is a union of $n_\ell$ closed curves, each projecting onto $\tilde\gamma_i$ with degree one, or it is a union of $n_\ell/2$ closed curves, each projection onto $\tilde\gamma_i$ with degree two, in which case $\Sigma^{(\ell)}$ is non-orientable and $\pi^{-1}(\tilde{\gamma}_i)$ is a M\"obius band.
    \end{enumerate}
\end{proposition}

The proof of \cref{prop:optimal-curve-lifting} is a modification of Ketover's proof of \cite[Proposition~2.2]{K}. Since the
argument is rather long, the full proof is given in \cref{app:optimal-lifting-proof}. 

Once \cref{prop:optimal-curve-lifting} is established, the optimal genus bound for the min--max $\mathbf{F}$-minimal surface arising from \cref{thm:aniso-simon-smith} follows immediately.

\begin{thm}[Optimal genus bound]\label{thm:aniso-genus-bound}
Assume that $F$ satisfies either \eqref{assumption1} or
\eqref{assumption2}.
Let $\{\Sigma_t^j\}\subset \Lambda$ be the minimizing sequence given by
\cref{thm:aniso-simon-smith}, and let
\(\Sigma^j:=\Sigma_{t_j}^j\) be a min-max sequence converging to
\[
    V=\sum_{i=1}^L n_i\,\mathbf{v}(\Gamma^i),
\]
where the \(\Gamma^i\)'s are pairwise disjoint smooth embedded \(\F\)-minimal surfaces.
Denote $\mathcal O$ and
$\mathcal U$ the set of indices \(i\) of orientable and non-orientable connected
components among the $\Gamma^i$'s, respectively. Then
\[
    \sum_{i\in \mathcal O}
    n_i\,\operatorname{genus}(\Gamma^i)
    +
    \frac12
    \sum_{i\in \mathcal U}
    n_i\bigl(\operatorname{genus}(\Gamma^i)-1\bigr)
    \leq
    \liminf_{j\to\infty}\,\liminf_{\tau\to t_j}
    \operatorname{genus}(\Sigma_\tau^j).
\]
\end{thm}

\begin{proof}
The proof is topological once \cref{prop:optimal-curve-lifting} is available.
It follows Ketover's proof of \cite[Theorem~1.2]{K}.

We will apply \cref{prop:optimal-curve-lifting} to the almost minimizing min-max sequence
\(\Sigma^j=\Sigma^j_{t_j}\). For each component \(\Gamma^i\) which is not homeomorphic
to \(S^2\), choose a finite collection of simple closed curves
$\alpha^i_1,\ldots,\alpha^i_{q_i}\subset \Gamma^i$
such that all curves meet only at one prescribed point \(p_i\), and
\[
    \Gamma^i\setminus \bigcup_{\ell=1}^{q_i}\alpha^i_\ell
\]
is a disk. If \(\Gamma^i\) is orientable of genus \(g_i\), we choose the standard \(2g_i\)
generators based at \(p_i\). If \(\Gamma^i\) is non-orientable of genus \(g_i\), we choose
the standard \(g_i\) one-sided generators based at \(p_i\); see \cite[Lemma~5.1]{K}.
For each component \(\Gamma^i\simeq S^2\), choose one auxiliary simple closed curve
\(\alpha^i_1\subset \Gamma^i\). This curve is only used to apply
\cref{prop:optimal-curve-lifting}; the corresponding component contributes zero to the
genus bound.

Choose \(\epsilon>0\) sufficiently small so that \(\epsilon/4<\epsilon_0\), where $\epsilon_0$ is the constant in
\cref{prop:optimal-curve-lifting}, the tubular neighborhoods \(T_\epsilon(\Gamma^i)\)
are pairwise disjoint, and the nearest point projections
$\pi_i:T_\epsilon(\Gamma^i)\to\Gamma^i$
are well-defined. We apply \cref{prop:optimal-curve-lifting} with scale \(\epsilon/4\).
After passing to a subsequence, we obtain curves
\(\tilde\alpha^i_\ell\subset \Gamma^i\), homotopic to \(\alpha^i_\ell\), and surfaces
\(\tilde\Sigma_j\) obtained from \(\Sigma^j\) by finitely many \(\gamma\)-reductions, such
that
\[
    \tilde\Sigma_j
    \to
    \sum_{i=1}^L n_i\,\mathbf v(\Gamma^i)
\]
as varifolds. Moreover, the lifts of every \(\tilde\alpha^i_\ell\) occur with full
multiplicity, which is in the annular case there are \(n_i\) degree-one lifts, and in the
M\"obius case there are \(n_i/2\) degree-two lifts. In particular, \(n_i\) is even for
non-orientable components.

We next perform surgeries on \(\tilde\Sigma_j\) so that the resulting surfaces are
contained in the tubular neighborhoods, without changing the lifted curves. Since the
lifting was obtained at scale \(\epsilon/4\), all lifted curves lie in the inner
neighborhoods \(T_{\epsilon/4}(\Gamma^i)\). Set
$\Lambda_i:=T_\epsilon(\Gamma^i)\setminus T_{\epsilon/2}(\Gamma^i).$
By varifold convergence,
\[
    \mathcal H^2(\tilde\Sigma_j\cap \Lambda_i)\to 0.
\]
Hence, by the coarea formula, for each \(i\) and all large \(j\) one can choose
\(\sigma_{i,j}\in(\epsilon/2,\epsilon)\) such that \(\tilde\Sigma_j\) intersects
\(\partial T_{\sigma_{i,j}}(\Gamma^i)\) transversely and
\[
    \mathcal H^1\bigl(\tilde\Sigma_j\cap \partial T_{\sigma_{i,j}}(\Gamma^i)\bigr)
\]
is as small as desired. Taking this length smaller than the relevant systolic scale of
\(\partial T_{\sigma_{i,j}}(\Gamma^i)\), every component of the intersection is a short
circle bounding a small disk in \(\partial T_{\sigma_{i,j}}(\Gamma^i)\).

As in Ketover's proof, we surger along these circles, starting with the innermost ones,
by gluing in the corresponding small disks and removing the cylindrical pieces of
\(\tilde\Sigma_j\) between nearby level surfaces. We then discard all connected components
of the resulting surface which are not contained in
\[
    \bigcup_i T_{\sigma_{i,j}}(\Gamma^i).
\]
These operations do not increase genus. They are supported in
\(T_\epsilon(\Gamma^i)\setminus T_{\epsilon/2}(\Gamma^i)\), and therefore do not affect
the lifted curves contained in \(T_{\epsilon/4}(\Gamma^i)\). We do not relabel the
surgered surfaces; thus
\[
    \tilde\Sigma_j\subset \bigcup_i T_{\sigma_{i,j}}(\Gamma^i).
\]

Fix first a non-spherical component \(\Gamma^i\). Cut \(T_{\sigma_{i,j}}(\Gamma^i)\) open
along
\[
    \pi_i^{-1}\left(\bigcup_{\ell=1}^{q_i}\tilde\alpha^i_\ell\right).
\]
Since the curves \(\tilde\alpha^i_\ell\) cut \(\Gamma^i\) into a disk, the resulting
three-manifold \(N_i\) is homeomorphic to a 3-ball. Equivalently, \(N_i\) may be written
as \(P_i\times[-\sigma_{i,j},\sigma_{i,j}]\), where \(P_i\) is a polygon, and
\(T_{\sigma_{i,j}}(\Gamma^i)\) is recovered from \(N_i\) by identifying pairs of vertical
faces of \(\partial P_i\times[-\sigma_{i,j},\sigma_{i,j}]\).

By the full lifting property, after cutting open, the boundary of
\(\tilde\Sigma_j\cap T_{\sigma_{i,j}}(\Gamma^i)\) in \(N_i\) consists, up to isotopy in
\(\partial N_i\), of the same parallel boundary curves as the corresponding stacked
model. In particular, one obtains \(n_i\) boundary curves isotopic to
\(\partial P_i\times\{0\}\). In the non-orientable case, the \(n_i/2\) degree-two lifts
split into \(n_i\) boundary curves after the cut. Applying the surgery lemma in a 3-ball,
\cite[Lemma~5.2]{K}, with \(U=N_i\), we may perform further neck-pinch surgeries in
\(N_i\) so that the cut-open pieces become \(n_i\) disks whose boundaries are isotopic in
\(\partial N_i\) to these \(n_i\) parallel curves.

After re-gluing the faces of the cut-open tubular neighborhood, we obtain the following
topological model. If \(\Gamma^i\) is orientable, each disk gives one copy of
\(\Gamma^i\) after the boundary identifications. Hence the contribution of this component
is
$n_i\,\operatorname{genus}(\Gamma^i).$
If \(\Gamma^i\) is non-orientable, then after the boundary identifications the disks pair
off to form \(n_i/2\) copies of the orientable double cover of \(\Gamma^i\). The orientable
double cover of a non-orientable surface of genus \(\operatorname{genus}(\Gamma^i)\) has
genus \(\operatorname{genus}(\Gamma^i)-1\). Hence the contribution of this component is
$ {n_i}\bigl(\operatorname{genus}(\Gamma^i)-1\bigr)/{2}.$

If \(\Gamma^i\simeq S^2\), its contribution to the left-hand side is zero. For completeness,
one repeats Ketover's spherical case: cut \(T_{\sigma_{i,j}}(\Gamma^i)\) along the
preimage of the auxiliary curve \(\tilde\alpha^i_1\). The complement consists of two
3-balls, and the same 3-ball surgery lemma applies to each of them. After re-gluing, only
spherical components are produced.

Adding the contributions over all components and using that neck-pinch surgeries,
isotopies, and discarding connected components do not increase genus, we obtain
\[
    \sum_{i\in \mathcal O}
    n_i\,\operatorname{genus}(\Gamma^i)
    +
    \frac12
    \sum_{i\in \mathcal U}
    n_i\bigl(\operatorname{genus}(\Gamma^i)-1\bigr)
    \le
    \liminf_{j\to\infty}\operatorname{genus}(\Sigma^j).
\]
Finally, by the definition of the generalized family, the genus of the slice \(\Sigma^j=\Sigma^j_{t_j}\) is bounded above by the
lower limiting genus of nearby smooth slices, hence
\[
    \operatorname{genus}(\Sigma^j)
    \le
    \liminf_{\tau\to t_j}\operatorname{genus}(\Sigma^j_\tau).
\]
Therefore, we have
\[
    \sum_{i\in \mathcal O}
    n_i\,\operatorname{genus}(\Gamma^i)
    +
    \frac12
    \sum_{i\in \mathcal U}
    n_i\bigl(\operatorname{genus}(\Gamma^i)-1\bigr)
    \le
    \liminf_{j\to\infty}\,\liminf_{\tau\to t_j}
    \operatorname{genus}(\Sigma^j_\tau).
\]
This proves the theorem.
\end{proof}

\section{Removing singularities via Bernstein}\label{sec:removable-singularity-bernstein}

We now turn to the removal of isolated singularities for stable
anisotropic minimal surfaces. Let
\(\Sigma\subset \B_1\setminus\{0\}\) be a properly embedded smooth
surface which is \(\F\)-stationary and \(\F\)-stable away
from the origin, which may represent a possible singularity. For the area functional, the monotonicity formula immediately implies that the density ratio remains bounded at this point. This implies that the point has
zero \(W^{1,2}\)-capacity on the surface. A standard logarithmic cutoff argument then allows the stationarity and stability inequalities, initially valid for test functions supported away from the point, to be
extended to test functions whose supports cross the point. Thus, the closure of the surface is stationary and stable across the point. The
Schoen--Simon--Yau curvature estimates, or equivalently the Schoen--Simon regularity theory in this two-dimensional setting, then
imply that the isolated singularity is removable.

In the anisotropic setting, however, no such monotonicity formula is available, and this strategy breaks down. Thus, the main issue is to recover, by other means, enough
control near the puncture to run the same capacity and curvature-estimate
argument. We give two independent criteria. In this section, we use a
Bernstein-type argument: under an explicit ellipticity-ratio bound for
the frozen integrand at the singular point, blow-ups of \(\Sigma\) have
flat leaves, which yields the required density bound. In the next section, we prove a different criterion based on the separation between nearby
sheets and a lower bound for the anisotropic Jacobi operator applied to
\(|x|\).

Recall that for \(x_0\in \B_1\), let \(G_{x_0}(\nu):=G(x_0,\nu)\), we denote
\[
    \langle \Psi_F(x_0,\nu)v,w\rangle
    =
    D^2_{\nu\nu}G_{x_0}(\nu)[v,w],
    \qquad v,w\in\nu^\perp .
\]
as the ellipticity matrix associated with the
frozen integrand \(G_{x_0}\), restricted to \(\nu^\perp\).

\begin{thm}[Removing singularities via Bernstein]
\label{thm:removable-singularity-via-bernstein}
Let \(F\in C^{2,\alpha}(G_2(\B_1))\) be a positive uniformly elliptic
integrand. Let
$\Sigma\subset \B_1\setminus\{0\}\subset\mathbb R^3$
be a smooth properly embedded \(\F\)-stationary surface which is
\(\F\)-stable in \(\B_1\setminus\{0\}\). Assume that
\(0\in\overline{\Sigma}\) and
\begin{align}\label{eq:ellipticity-bound-bernstein}
    \frac{\max_{\nu\in\mathbb S^2}\lambda_{\max}\Psi_F(0,\nu)}
    {\min_{\nu\in\mathbb S^2}\lambda_{\min}\Psi_F(0,\nu)}
    < 8 .
\end{align}
Then the origin is a removable singularity. More precisely,
\(\overline{\Sigma}\cap\B_1\) is a smooth embedded
\(\F\)-stationary surface in \(\B_1\).
\end{thm}

We use a \textit{lamination}
approach, inspired by the ideas of
\cite{MPR-removable-singularity} and \cite{FcS}. The result is not tied
to the min--max construction and may be applied independently in other
settings.

First we prove that under the condition \cref{eq:ellipticity-bound-bernstein} a complete $\F$-stationary and $\F$-stable surface in $\R^{3}\setminus\{0\}$ must be flat.
\begin{thm}[{\cite[Lemma 3.3]{MPR-removable-singularity}}]\label{thm:anisotropic-bernstein}
    Let $F(\cdot)$ be a fixed elliptic anisotropic integrand on $G_2(\R^3)$ with the ellipticity bound \cref{eq:ellipticity-bound-bernstein} and let $\Sigma\subset \R^3\setminus \{0\}$ be an $\F$-stationary, $\F$-stable, immersed surface which is complete outside of the origin. Then $\overline{\Sigma}$ is a plane.
\end{thm}
\begin{proof}
    Since $\Sigma$ is $\F$-stable outside the origin,
    the second variation inequality (see, for example, \cite[(A.4)]{CL}) gives
    \begin{equation}\label{eq:ineq-stability-A}
        \int_{\Sigma}
    \left\{
        \langle \Psi_F\nabla_\Sigma \phi,\nabla_\Sigma\phi\rangle
        - \tr_\Sigma(\Psi_F A_\Sigma^2)\phi^2
    \right\}\geq 0 \quad \text{for all }\phi \in C^{\infty}_c(\Sigma\setminus\{0\})\,.
    \end{equation}
    Since $\Sigma$ is $\F$-stationary, it satisfies the equation $\tr_\Sigma(\Psi_FA_\Sigma) = 0$. On the other hand, the Cayley--Hamilton identity for the two-dimensional endomorphism \(A_\Sigma\) is
    \[
        A_\Sigma^2 - H_\Sigma A_\Sigma + K_\Sigma I=0,
    \]
    where \(H_\Sigma=\tr_\Sigma A_\Sigma\) and \(K_\Sigma=\det A_\Sigma\). Composing with \(\Psi_F\) and taking the trace gives
    \[
        \tr_\Sigma(\Psi_F A_\Sigma^2)
    - H_\Sigma \tr_\Sigma(\Psi_F A_\Sigma)
    + K_\Sigma \tr_\Sigma(\Psi_F)=0.
    \]
    Using the stationary equation, the middle term vanishes, and hence
    \[
        \tr_\Sigma(\Psi_F A_\Sigma^2)
    = - \tr_\Sigma(\Psi_F) K_\Sigma.
    \]
    Substituting this identity into \eqref{eq:ineq-stability-A} yields
    \begin{align}
        \int_{\Sigma} \left\{\ang{\nabla_\Sigma\phi,\Psi_F\nabla_\Sigma\phi} + \tr_\Sigma(\Psi_F) K_\Sigma \phi^2\right\}\geq 0\quad \text{for all }\phi \in C^{\infty}_c(\Sigma\setminus\{0\})\,.\label{eq:ineq-stability-K}
    \end{align}
    Moreover, the stationary equation also implies \(K_\Sigma\leq 0\). Indeed, since \(\Psi_F\) is positive definite, the symmetric endomorphism \(\Psi_F^{1/2}A_\Sigma \Psi_F^{1/2}\) has trace
    \[
        \tr\bigl(\Psi_F^{1/2}A_\Sigma \Psi_F^{1/2}\bigr)
    = \tr_\Sigma(\Psi_F A_\Sigma)=0.
    \]
    Therefore its two eigenvalues are of the form \(\mu\) and \(-\mu\), so its
    determinant is non-positive. Since
    \[
        \det\bigl(\Psi_F^{1/2}A_\Sigma \Psi_F^{1/2}\bigr)
    = \det(\Psi_F)\det(A_\Sigma)
    = \det(\Psi_F)K_\Sigma
    \]
    and \(\det(\Psi_F)>0\), we obtain \(K_\Sigma\leq 0\). 
    
    Set
\[
    m_F:=\min_{\nu\in\mathbb S^2}\lambda_{\min}(\Psi_F(0,\nu)),
    \qquad
    M_F:=\max_{\nu\in\mathbb S^2}\lambda_{\max}(\Psi_F(0,\nu)).
\]
Then \(M_F\geq m_F>0\), and the ellipticity assumption \eqref{eq:ellipticity-bound-bernstein} becomes
\[
    \frac{M_F}{m_F}<8.
\]
In particular,
\[
    \langle \Psi_F\nabla_\Sigma\phi,\nabla_\Sigma\phi\rangle
    \leq M_F|\nabla_\Sigma\phi|^2,
    \qquad
    \tr_\Sigma(\Psi_F)\geq 2m_F.
\]
Using \(K_\Sigma\leq 0\), we have pointwise
\[
    \langle \Psi_F\nabla_\Sigma\phi,\nabla_\Sigma\phi\rangle
    + \tr_\Sigma(\Psi_F)K_\Sigma\phi^2
    \leq
    M_F |\nabla_\Sigma\phi|^2
    + 2m_F K_\Sigma\phi^2 .
\]
Combining this with \eqref{eq:ineq-stability-K} yields
\[
    \int_\Sigma
    \left\{
        M_F |\nabla_\Sigma\phi|^2
        + 2m_F K_\Sigma\phi^2
    \right\}\geq 0 .
\]
Dividing by \(M_F\), we obtain
\[
    \int_\Sigma
    \left\{
        |\nabla_\Sigma\phi|^2
        + c_0 K_\Sigma \phi^2
    \right\}\geq 0,
    \qquad
    c_0:=\frac{2m_F}{M_F}.
\]
Since \(M_F/m_F<8\), one has \(c_0>1/4\). Hence
\begin{equation}\label{interm1}
    \int_{\Sigma}
    \left\{
        |\nabla_\Sigma\phi|^2
        + c_0 K_\Sigma \phi^2
    \right\}
    \geq 0\quad
\text{for all }\phi\in C_c^\infty(\Sigma\setminus\{0\}),\text{ where
}c_0>1/4.
\end{equation}

    Let \(g\) be the induced metric on \(\Sigma\), and set
\[
    \widetilde g:=\frac{1}{r^2}g,
    \qquad r=|x|.
\]
The ambient conformal metric
\[
    \widehat g:=\frac{1}{r^2}\langle\cdot,\cdot\rangle
\]
on \(\mathbb R^3\setminus\{0\}\) is isometric to
\(\mathbb S^2\times\mathbb R\). Since \(\Sigma\) is complete outside the
origin, \((\Sigma,\widetilde g)\) is complete. It is also non-compact.
By \cite[Theorem A]{C}, the inequality
\eqref{interm1} with \(c_0>1/4\) implies that the conformal
type of \(\Sigma\) is either \(\mathbb C\) or
\(\mathbb C\setminus\{0\}\).

The quadratic form in \eqref{interm1} lifts to the
conformal cover. By \cite[Theorem 1(iii)]{FcS}, there exists a positive
solution \(u\) of 
\[-\Delta_\Sigma u+c_0K_\Sigma u=0.\]
Since \(K_\Sigma\le0\), this equation gives \[\Delta_\Sigma u=c_0K_\Sigma u\le0.\]
Thus, \(u\) is a positive superharmonic function on a surface conformally
equivalent to \(\mathbb C\) or \(\mathbb C\setminus\{0\}\). These
Riemann surfaces are parabolic, so \(u\) is constant. Hence $c_0K_\Sigma u=0$.
Since \(c_0>0\) and \(u>0\), we conclude that $K_\Sigma\equiv0$. We now show that \(A_\Sigma\equiv0\). Indeed, \(K_\Sigma=0\) means that one of the principal curvatures is zero. Since \(\operatorname{tr}(\Psi_F(\nu)A_\Sigma)=0\), and \(\Psi_F(\nu)\) is positive definite, the other principal curvature must also vanish. Hence, \(A_\Sigma\equiv0\).

Thus, \(\Sigma\) is locally contained in an affine plane. Since
\(\Sigma\) is connected, its image is contained in a single affine
plane \(P\). Moreover, completeness outside the origin rules out any
proper open subset of \(P\) with boundary away from the origin. Hence
\(\overline{\Sigma}=P\).
\end{proof}

\begin{rmk}
\label{rmk:ellipticity-threshold}
The ellipticity-ratio assumption should be viewed as a quantitative way
of placing the anisotropic problem in the range where Lin's scalar
spectral argument applies. Lin \cite{LinEigenvalue} proved a Bernstein-type result for stable
elliptic parametric integrals with constant coefficients under a
\(C^{2,\alpha}\)-closeness assumption to the area integrand. In our argument this closeness is replaced by the
explicit ellipticity condition
${M_F}/{m_F}<8$.
Indeed, the anisotropic stability inequality \eqref{eq:ineq-stability-A} is reduced to the inequality \eqref{interm1} with coefficient \(c_0=2m_F/M_F\), and Lin's
spectral theorem applies exactly in the range \(c_0>1/4\).

The constant \(1/4\) is sharp for this spectral method, as
reflected both in Lin's eigenvalue theorem and in Castillon's inverse
spectral result \cite{LinEigenvalue,C}, with the Poincar\'e disk realizing the endpoint. Thus the number \(8\) is the explicit threshold
produced by our comparison with inequality \eqref{interm1}. Of course an anisotropic stable Bernstein theorem may still hold beyond this range; however it is not attainable with this technique.
\end{rmk}

\subsection{The blow-up lamination is flat}

we recall the definition of a lamination:
\begin{definition}[{\cite[Definition 2]{MPR-limit-leaf}}]
    A surface lamination of a Riemannian $3$-manifold $N$ is the union of a collection of pairwise disjoint, connected, injectively immersed surfaces with a certain local product structure. More precisely, it is a pair $(\mathcal{L},\mathcal{A})$ satisfying:
    \begin{enumerate}
        \item $\mathcal{L}$ is a closed subset of $N$.
        \item $\mathcal{A} = \{\phi_\beta: \textbf{B}_1^2\times (0,1)\to U_\beta \}_\beta$ is a collection of coordinate charts of $N$.
        \item For each $\beta$, there exists a closed subset $C_\beta$ of $(0,1)$ such that $\phi_\beta^{-1}(U_\beta \cap \mathcal{L}) = \textbf{B}_1^2\times C_\beta$.
    \end{enumerate}
\end{definition}

Like in \cite{MPR-limit-leaf} we denote laminations by $\mathcal{L}$, omitting the charts. A lamination $\mathcal{L}$ is said to be a foliation if $\mathcal{L} = N$. Every lamination $\mathcal{L}$ naturally decomposes into a collection of disjoint connected surfaces, called the leaves of $\mathcal{L}$. We call a lamination $\mathcal{L}$ an \textit{$\F$-minimal lamination} if all its leaves are $\F$-minimal ($\F$-stationary) surfaces. In the next proposition we prove properties of sequences of these laminations with curvature bounds as in \cite{CM}:
\begin{proposition}[{\cite[Proposition B.1]{CM}}]\label{prop:lamination-convergence}
    Let $N^3$ be a fixed manifold and let $\mathcal{L}_i\subset B_{2R} \subset N$ be a sequence of $\F_i$-minimal laminations with finitely many leaves (possibly depending on $i$) and uniformly bounded curvature on each leaf. Moreover assume that $F_i$s are a sequence of elliptic integrands converging in $C^{2,\beta}$ to a limit $F$. Then (up to a subsequence) $\mathcal{L}_i$ converges in $C^{\alpha}$ on $B_{R}$ to an $\F$-minimal (Lipschitz) lamination $\mathcal{L}$ for any $\alpha<1$. Moreover if the leaves of $\mathcal{L}_i$ are $\F$-stable, then the leaves of the limit lamination $\mathcal{L}$ will be $\F$-stable as well.
\end{proposition}
\begin{proof}
    The proof is almost identical to \cite[Proposition B.1]{CM} with anisotropic equations in place of the minimal surface PDE. In the end also each leaf of $\mathcal{L}$ is the $C^{1,\alpha}$ limit of some leaf of $\mathcal{L}_i$, hence the leaves of $\mathcal{L}$ inherit $\F$-stationary/stability from $\mathcal{L}_i$.
\end{proof}

Then we show that the blow-up around any singularity in an $\F$-stationary and $\F$-stable surface, produces a lamination whose all leaves are flat.
\begin{proposition}[Flatness of limit laminations]\label{prop:limit-lamination-is-flat}
    Under the conditions of \cref{thm:removable-singularity-via-bernstein} and \cref{eq:ellipticity-bound-bernstein}, any blowup of $\Sigma$ around the origin (up to a subsequence) converges to a lamination $\mathcal{L}\subset \R^3\setminus \{0\}$ with flat leaves. Moreover for any $\epsilon>0$ small, there exists $r_\epsilon>0$ such that for all $|x|\leq r_\epsilon$ we have $|x||A_\Sigma(x)| \leq \epsilon$.
\end{proposition}
\begin{proof}
    First take any point $x\in\Sigma$, since $\Sigma$ is stable in $B_{\dist(x,p)/2}(x)$, hence by \cite[Theorem 2.8]{DePhilippisDeRosa} we can assert that:
    \begin{align}\label{curvature-estimate-lamination-blowup}
        |A_\Sigma|^2 \leq \frac{C}{\dist(x,p)^2}\,.
    \end{align}
    Take the map $\phi_\rho:\Sigma \supset \B_{1}(x)\to \mathbf{B}_{\rho^{-1}}(0)$ to be the blow-up maps and define $\Sigma_\rho = \phi_\rho(\Sigma)$ and push forward the elliptic integrand to $F_\rho$. Then the estimate \cref{curvature-estimate-lamination-blowup} asserts that for any compact $K\subset \R^3\setminus\{0\}$ we have $\sup_{K}|A_{\Sigma_{\rho}}| \leq C_K$. Now we are in a position to apply \cref{prop:lamination-convergence} and see that $\Sigma_\rho$ converges, as laminations, in $C^{\alpha}$ for any $\alpha<1$ to a limit lamination $\mathcal{L}$ and each leaf of $\mathcal{L}$ is the $C^{1,\beta}$ limit of a leaf of $\Sigma_\rho$. Hence each leaf of $\mathcal{L}$ is $\F$-stationary and $\F$-stable, where $F(0,\cdot)$ is the frozen integrand at the origin. Now each leaf $L\subset \mathcal{L}$ is an $\F$-stationary/stable smooth surface in $\R^3\setminus \{0\}$. Then by \cref{thm:anisotropic-bernstein}, we see that $\overline{L}$ is in fact a flat plane. Since this holds for each leaf and they all are disjoint, the lamination $\overline{\mathcal{L}}$ is a fully flat lamination in $\R^3$. Since the blow-up was arbitrary and the convergence is $C^{2,\alpha}$ on each leaf, we conclude the bound on the second fundamental form.
\end{proof}

\subsection{The density bound: proof of \cref{thm:removable-singularity-via-bernstein}}

We use a similar idea to \cite[Lemma 4.1]{MPR-removable-singularity} and prove the following proposition:
\begin{proposition}\label{prop:bounded-density-bernstein}
    Under the assumptions of \cref{thm:removable-singularity-via-bernstein}, the origin has finite density:
    \begin{align}
        \limsup_{r\to 0} \frac{\hau^2(\Sigma\cap \B_r)}{r^2} < \infty\,.
    \end{align}
\end{proposition}
\begin{proof}
    We show that the square distance function $f(x) = |x|^2$ on $\Sigma$ is a Morse function on $\Sigma\cap\B_\rho$ provided that $\rho>0$ is chosen small enough. Indeed, take a point $p\in\Sigma\cap\B_\rho$ such that $\nabla_\Sigma f(p) =0$, which means that at $p$ we have $x = (x\cdot\nu)\nu$ and calculate:
    \begin{align}
        \nabla^2_\Sigma f(p) = 2g_\Sigma - 2(x\cdot\nu)A_\Sigma\,.
    \end{align}
    The Hessian will vanish at a critical point, i.e. $\nabla^2_\Sigma f(p)=0$, if and only if $p$ is an umbilic point meaning at $p$ the two principal curvatures will satisfy $\kappa_1(p) = \kappa_2(p) = -{1}/{|x|}$. Now, since $F$ is smooth and elliptic and $\Sigma$ is $\F$-stationary, we know that $|\tr_\Sigma(\Psi_FA_\Sigma)|\leq C$, which would be impossible with $\kappa_1(p) = \kappa_2(p) = -{1}/{|x|}$ provided that $|x|\leq\rho$ is small enough.

    Now \cref{prop:limit-lamination-is-flat} asserts that for any $\epsilon>0$ small enough, there is $r_\epsilon>0$ such that for all $r<r_\epsilon$, we have $|A||x| \leq \epsilon$. This means that $\nabla^2_\Sigma f(p) = 2g_\Sigma - 2(x\cdot\nu)A_\Sigma$ cannot have negative eigenvalues, provided $\epsilon>0$ is chosen small enough. This in turn implies that the Morse complex induced by $f(x)$ on $\Sigma$ cannot have saddle or local maxima points in sufficiently small radii. This means that for all $\rho<r_\epsilon$ the number of curves in the intersection $\Sigma \cap \de\B_\rho$ is at most the number of curves in $\Sigma\cap \de\B_{r_\epsilon}$, which further implies that on each annulus $\Sigma\cap(\B_\rho\setminus\B_{\rho/2})$ the number of connected components is uniformly bounded. Then using the bound on the second fundamental form, we see that the area density on each annulus must be uniformly upper bounded. Taking a sum over annuli, we conclude.
\end{proof}
Now we proceed to the proof of the main result of this section.
\begin{proof}[Proof of \cref{thm:removable-singularity-via-bernstein}]
    Since the density ratio is bounded in a neighborhood of $p$ by \cref{prop:bounded-density-bernstein}, a standard capacity argument shows that the $\mathbf{F}$-stability inequality extends across the singularity. Applying \cite[Theorem~2.8]{DePhilippisDeRosa}, we conclude that the second fundamental form is pointwise bounded in the interior of $U$. Consequently, $p$ is a removable singularity. See also \cite[Proof of Theorem~5.1, Step~3]{DDL}.
\end{proof}

\section{Removing singularities via separation}\label{sec:remove-singularity-separation}

We now prove the second removability criterion. The
argument does not rely on the Bernstein theorem, nor on the ellipticity
ratio. Instead, it assumes that the extrinsic distance \(|x|\) is a
suitable subsolution for the anisotropic Jacobi operator, in the sense of
\cref{eq:Jacobi-lower-bound-condition}. The idea is
that if the density were to blow up, then on small dyadic annuli the
surface would contain many nearby sheets. The separation between adjacent
sheets satisfies an approximate Jacobi equation, and the lower bound for
\(\mathcal J_F|x|\) allows us to compare this separation with \(|x|\).
This gives quantitative decay of the region where the separation is very
small, and ultimately yields the logarithmic density bound needed for the
capacity argument. Once this density control is established, the desired removability follows. This condition is satisfied for the area functional.
We also introduce a \(C^3\)-pinching condition for $F$ in \cref{def:C3-pinching}
and show that it implies the required lower bound for \(\mathcal J_F|x|\).

Since the results here are all local in nature, we provide their Euclidean version. Let $\mathcal{J}_F$ be the Jacobi operator associated to the integrand $F$ with the sign convention:
\begin{align}
    \delta^2_\F \Sigma(\phi\nu_\Sigma) = \ang{-\mathcal{J}_F \phi,\phi}\,,
\end{align}
for any compactly supported $\phi \in C_c^\infty(\Sigma)$, where \(\nu_\Sigma\) is a
choice of unit normal. We prove a conditional removable singularity theorem as follows:
\begin{thm}\label{thm:removable-singularity-with-jacobi-lower-bound}
    Let \(F\) be a smooth anisotropic integrand on \(G_2(\B_1)\), and let $\Sigma\subset \B_1\setminus\{0\}$ be a smooth properly embedded surface which is \(\F\)-stationary and \(\F\)-stable in \(\B_1\setminus\{0\}\). Moreover assume that there is some small radius $\rho>0$ and constant $C_0$ such that:
    \begin{align}\label{eq:Jacobi-lower-bound-condition}
        \mathcal{J}_F|x| \geq -C_0\quad\text{ for all } x\in\Sigma\cap\B_\rho\,.
    \end{align}
    Then the origin is a removable singularity. More precisely, \(\overline{\Sigma}\cap\B_1\) is a smooth embedded $\F$-stable \(\F\)-minimal surface in \(\B_1\).
\end{thm}

Condition \cref{eq:Jacobi-lower-bound-condition} is always satisfied for the area functional and minimal surfaces.

To put this condition into better context, we show that under a $C^3$-pinching condition on the integrand $F$, \cref{eq:Jacobi-lower-bound-condition} will be satisfied on $\F$-stable surfaces.

\begin{definition}[\(C^3\)-pinching condition]
\label{def:C3-pinching}
Let \(\Psi_F(x,\nu)\) be defined as in \eqref{eq:def-Psi_F}, and define the pointwise ellipticity ratio
\[
    \beta(x,\nu)
    :=
    \frac{\lambda_{\max}(\Psi_F(x,\nu))}
         {\lambda_{\min}(\Psi_F(x,\nu))},
\]
where the eigenvalues are taken on \(T_\nu\mathbb S^2\).

We say that \(F\) satisfies the \textit{\(C^3\)-pinching condition} if,
after possibly an affine normalization, for every
\((x,\nu)\in\B_1\times\mathbb S^2\),
\[
    \|\nabla_{\mathbb S^2}\Psi_F\|(x,\nu)
    \le
    2\,\lambda_{\min}(\Psi_F(x,\nu))
    \sqrt{
        \frac{3\beta(x,\nu)(\beta(x,\nu)+1)}
             {3\beta(x,\nu)^2+1}
    } .
\]
\end{definition}
We observe that 
$$\sqrt{
        \frac{3\beta(x,\nu)(\beta(x,\nu)+1)}
             {3\beta(x,\nu)^2+1}
    }\geq 1,$$
    so in particular the $C^3$-pinching condition is satisfied if $\|\nabla_{\mathbb S^2}\Psi_F\|(x,\nu)
    \le
    2\,\lambda_{\min}(\Psi_F(x,\nu))$.

We now show the following.
\begin{proposition}\label{prop:C3-pinching-gives-jacobi-lower-bound}
    If  the integrand \(F\) satisfies the \(C^3\)-pinching condition in \cref{def:C3-pinching} and $\Sigma\subset \B_1\setminus\{0\}$ is a smooth embedded surface which is \(\F\)-stationary and \(\F\)-stable in \(\B_1\setminus\{0\}\), then there are two constants $\rho_\F,C_\F>0$ such that:
    \begin{align}
        \mathcal J_F |x| \ge -C_\F
    \qquad
    \text{on } \Sigma\cap \B_{\rho_\F}.
    \end{align}
\end{proposition}

As a direct corollary we then can show that under this condition, singularities are removable:
\begin{corollary}\label{cor:removable-singularity-with-C3-pinching}
    If $F$ satisfies the $C^3$-pinching condition in \cref{def:C3-pinching} and $\Sigma\subset\B_1\setminus\{0\}$ is smooth, embedded, $\F$-stationary and $\F$-stable away from the origin, then the origin is a removable singularity and $\overline\Sigma$ is smooth in all of $\B_1$.
\end{corollary}

\subsection{The $C^3$-pinching condition and consequences}

The main consequence of \cref{def:C3-pinching} is that on embedded $\F$-minimal surfaces in $\R^3$,the extrinsic distance $|x|$ is a \textit{sub-solution} (in a sense) to the anisotropic Jacobi operator on the surface. Define the $\F$-Jacobi operator:
\begin{align}
    \mathcal{J}_F v :=
    \operatorname{div}_\Sigma
    \bigl(\Psi_F(x,\nu_\Sigma)\nabla_\Sigma v\bigr)
    +
    \operatorname{tr}_\Sigma
    \bigl(\Psi_F(x,\nu_\Sigma)A_\Sigma^2\bigr)v
    +
    R_F(x,\nu_\Sigma,A_\Sigma)v.
\end{align}
Here \(R_F\) denotes the lower-order coefficient coming from the
\(x\)-dependence of the integrand. In the local coordinates under
consideration, the uniform \(C^3\)-bounds in \eqref{H:regularity} imply
\[
    |R_F(x,\nu_\Sigma,A_\Sigma)|
    \le
    C_F\bigl(1+|A_\Sigma|\bigr),
\]
where \(C_F\) depends only on the local ellipticity constants and the
local \(C^3\)-bounds for \(G\).

\begin{proposition}\label{prop:C3-pinching-implies-subsolution-F-stationary}
    Assume $F$ satisfies the assumptions in \cref{def:C3-pinching}. Let $\Sigma\subset \textbf{B}_1\setminus\{0\}\subset \R^3$ be an $\F$-minimal surface, then we have that:
    \begin{align}
        \mathcal{J}_F |x| \geq -C_{\F}\bigl(1+|x||A_\Sigma|\bigr)\,,
    \end{align}
    for all $x\in\B_\rho$ for small enough $\rho(\F)>0$.
\end{proposition}

\begin{proof}[Proof of \cref{prop:C3-pinching-gives-jacobi-lower-bound}]
    It is immediate from stability and boundary estimates using \cite[Theorem 2.8]{DePhilippisDeRosa} that there exists some constant $K>0$, depending on $\hau^2(\Sigma)$ such that:
    \begin{align}\label{eq:curvature-estimate}
        |A_\Sigma(x)| \leq \frac{K}{|x|}\,.
    \end{align}
    Then we apply \cref{prop:C3-pinching-implies-subsolution-F-stationary}.
\end{proof}

Before the proof of \cref{prop:C3-pinching-implies-subsolution-F-stationary} we collect some preliminary information:
\begin{lemma}\label{lem:prelim-info-C3-pinching}
    Given an anisotropy integrand $F$ as in \cref{def:C3-pinching} with $\|F\|_{C^{2,\alpha}}\leq \Lambda$ the following is true:
    \begin{enumerate}
        \item We have the following global ellipticity ratio bound for all $x\in\B_1$:
        \begin{align}
            \beta(x,\nu) \leq  e^{2\sqrt{3}\pi}\,,
        \end{align}
        where $\beta(x,\nu)$ is the ellipticity ratio as defined in \cref{def:C3-pinching}.
        \item There exists small enough $\rho(\Lambda,\F)>0$ such that for all $x\in \B_{\rho}$ we have 
        \begin{align}\label{eq:C3-pinching-lower-bound}
            |\lambda_{\min}\Psi_F(x,\nu)| \geq \epsilon_\F\,,
        \end{align}
        for some positive constant $\epsilon_\F>0$.
    \end{enumerate}
\end{lemma}
\begin{proof}
    At a point $\nu\in\mathbb{S}^2$ take the eigenvalues of $\Psi_F$ to be $0<\lambda_1 \leq \lambda_2$ and take a unit tangent vector $e$ in $T_\nu\mathbb{S}^2$, we know that:
    \begin{align}
        \de_e(\beta) = \de_e\left(\frac{\lambda_2}{\lambda_1}\right) = \frac{\de_e(\lambda_2) - \beta \de_e(\lambda_1)}{\lambda_1}\,.
    \end{align}
    Now note that the condition in \cref{def:C3-pinching} implies:
    \begin{align}
        |\de_e(\beta)| \leq \frac{\sqrt{\de_e(\lambda_1)^2 + \de_e(\lambda_2)^2}}{\lambda_1}\sqrt{1+\beta^2}  \leq \frac{|\nabla_{\mathbb{S}^2}\Psi_F|}{\lambda_1}\sqrt{1+\beta^2} \leq 2\sqrt{3}\beta\,.
    \end{align}
    This in turn implies that:
    \begin{align}
        |\nabla_{\mathbb{S}^2} \log(\beta) | \leq 2\sqrt{3}\,.
    \end{align}
    Now since any smooth convex shape has an Umbilical point, there exists $\nu_0$ such that $\beta(x,\nu_0) = 1$. This means:
    \begin{align}
        \beta \leq e^{2\sqrt{3}\pi}\,.
    \end{align}
    Moreover since every anisotropy has an umbilical point, then this implies a global ellipticity bound on $\mathbb{S}^2$ as well. Now define $c_\F:=\inf_{\nu\in\mathbb{S}^2}\lambda_{\min} (\Psi_F(0,\nu))$. Since $\|F\|_{C^{2,\alpha}}\leq \Lambda$, we see that $|\Psi_F(x,\nu) - \Psi_F(0,\nu)| \leq \Lambda|x|^\alpha$. Taking $|x| \leq \frac{1}{2\Lambda}c_\F^{1/\alpha}$ the second conclusion follows.
\end{proof}

\begin{proof}[Proof of \cref{prop:C3-pinching-implies-subsolution-F-stationary}]
    Since \(\Sigma\) is \(\F\)-stationary, the anisotropic mean curvature vanishes. In the local coordinates under consideration, this gives
    \begin{align}\label{eq:HF-meancurvature}
        |\tr_\Sigma(\Psi_F A_\Sigma)| \leq C_F\,.
    \end{align}
    With this we calculate:
    \begin{align}
    \begin{aligned}
        \mathcal{J}_F |x| \geq &\frac{1}{|x|}\left[\tr_\Sigma(\Psi_F) - \frac{\ang{\Psi_Fx^T,x^T}}{|x|^2} +\tr_\Sigma(\Psi_F A_\Sigma^2)|x|^2+ \ang{\mathrm{div}_\Sigma \Psi_F,x^T}\right] \\[5pt] &+ R_F(x,\nu_\Sigma,A_\Sigma)|x| - C_F.
    \end{aligned}
    \end{align}
    Here \(R_F(x,\nu_\Sigma,A_\Sigma)\) denotes the lower-order contribution coming from the \(x\)-dependence of \(F\), and satisfies
    \begin{equation}\label{eq:QF-bound}
        |R_F(x,\nu_\Sigma,A_\Sigma)|
        \le
        C_F(1+|A_\Sigma|).
    \end{equation}
    First note that using \cref{eq:HF-meancurvature} and the symmetry of $\nabla_{\mathbb{S}^2} \Psi_F$ we gather that:
    \begin{align}
    \begin{aligned}
        \left|\ang{\mathrm{div}_\Sigma \Psi_F,x^T}\right| &\leq |x||\nabla_{\mathbb{S}^2} \Psi_F| \left[ \sqrt{\frac{3\beta^2+1}{3\beta^2 + 3}}|A_\Sigma| + \frac{C_F}{\lambda_{\min}(\Psi_F)}\right]\\
        &\leq |x||\nabla_{\mathbb{S}^2} \Psi_F| \sqrt{\frac{3\beta^2+1}{3\beta^2 + 3}}|A_\Sigma| + C_{\F}|x| \,.
    \end{aligned}
    \end{align}
    Second by Cayley-Hamilton and \cref{eq:HF-meancurvature} we gather that:
    \begin{align}\label{eq:cayley-hamilton-1}
    \begin{aligned}
        \tr_\Sigma(\Psi_F A_\Sigma^2 ) &= \tr_\Sigma(A_\Sigma)\tr_\Sigma(\Psi_F A_\Sigma) -(\det A_\Sigma)\tr_\Sigma(\Psi_F)\\ &\geq -C_F|A_\Sigma| -(\det A_\Sigma)\tr_\Sigma(\Psi_F)\,.
    \end{aligned}
    \end{align}
    With the eigenvalues of $\Psi_F$ as $0<\lambda_1\leq\lambda_2$ and again by \cref{eq:HF-meancurvature} we assert that:
    \begin{align}\label{eq:lower-bound-HF}
        -\det(A_\Sigma) \geq \frac{\lambda_1\lambda_2}{\lambda_1^2 + \lambda_2^2}|A_\Sigma|^2 - \frac{C_F}{\sqrt{\lambda_1^2+\lambda_2^2}}|A_\Sigma|\,.
    \end{align}
    We use \cref{eq:cayley-hamilton-1}, \cref{eq:lower-bound-HF} and \cref{eq:C3-pinching-lower-bound} we see:
    \begin{align}
        \tr_\Sigma(\Psi_F A_\Sigma^2 ) \geq \lambda_1\frac{\beta(\beta+1)}{\beta^2+1}|A_\Sigma|^2 - C_{\F}|A_\Sigma|
    \end{align}
    Then we estimate as follows:
    \begin{align}
    \begin{aligned}
        &\tr_\Sigma(\Psi_F) - \frac{\ang{\Psi_Fx^T,x^T}}{|x|^2} +\tr_\Sigma(\Psi_F A_\Sigma^2)|x|^2+ \ang{\mathrm{div}_\Sigma \Psi_F,x^T}  \\
        &\qquad\geq \lambda_1 + \lambda_1|x|^2\frac{\beta(\beta+1)}{\beta^2+1}|A_\Sigma|^2 - |x||\nabla_{\mathbb{S}^2} \Psi_F|\sqrt{\frac{3\beta^2+1}{3\beta^2 + 3}}|A_\Sigma| - C_\F|x|^2|A_\Sigma|\,.
    \end{aligned}
    \end{align}
    Now if $|\nabla_{\mathbb{S}^2} \Psi_F| \leq 2\lambda_1 \sqrt{\frac{3\beta(\beta+1)}{3\beta^2+1}}$, with a simple completion of squares and note \eqref{eq:QF-bound} we gather that:
    \begin{align}
        \mathcal{J}_F|x| \geq -C_{\F}\bigl(1+|x||A_\Sigma|\bigr)
    \end{align}
\end{proof}

\subsection{The separation function}
First we collect some preliminary information. Name $\mathrm{An}_r = B_{r}(0)\setminus B_{r/2}(0)$. Then the estimate \cref{eq:curvature-estimate} will essentially show that each connected component of $\Sigma\cap \mathrm{An}_{r}$ belongs to a compact family. We will summarizes the information in the next proposition:
\begin{proposition}\label{prop:annuli-structure}
    Fix any $\tau>0$ small and let $\Sigma$ be embedded, smooth, $\F$-stationary and $\F$-stable away from the origin. Then there exists a constant $\delta_0>0$ such that for any $x\in\Sigma$ the intersection $\Sigma \cap B_{\delta_0 |x|}(x)$ is a disjoint union of disks, graphical over some plane with gradient bounded by $\tau>0$.
\end{proposition}
\begin{proof}
    The proof follows from the curvature estimate \cref{eq:curvature-estimate} and standard estimate for normal coordinates.
\end{proof}

\begin{proposition}\label{prop:separation-function-definition}
    Let $\Sigma$ be embedded, smooth , $\F$-stationary and $\F$-stable away from the origin. Then there exists a constant $\kappa_0(K)>0$ such that for any $\kappa<\kappa_0$ there exists an open set $U\subset\Sigma\cap B_1$ and a \textit{separation function} $\mathbf{s}:U\to\R^+$ with the following properties:
    \begin{enumerate}
        \item On $U$ we have $\mathbf{s}(x) \leq \kappa|x|$ and for $z\in \de U\setminus  \de B_1$ we have $\mathbf{s}(z) = \kappa|z|$.
        \item For any $x\in U$, the unique connected component $\Sigma\cap B_{\delta_0|x|}(x)$ that passes through $x$ has distance with the nearest component at least $\frac{1}{2}\mathbf{s}(x)$ and at most $2\mathbf{s}(x)$.
        \item The separation function is continuous and locally Lipschitz on $U$.
        \item $\mathbf{s}$ satisfies the following differential inequality:
        \begin{align}\label{eq:separation-PDE}
            \mathcal{J}_F\mathbf{s} \leq C(K,\F) \left[\frac{\mathbf{s}^2}{|x|^3} + \frac{\mathbf{s}}{|x|}\right]\,.
        \end{align}
        \item The following bound is true on $U$:
        \begin{align}\label{eq:log-separation-bound}
            |\nabla_\Sigma \log(\mathbf{s})(x)|\leq \frac{C(K,\F)}{|x|}\,,
        \end{align}
        provided $\delta_0,\kappa_0>0$ is chosen small enough.
        \item For any point $x\in\Sigma\setminus U$, the intersection $\Sigma\cap B_{\kappa|x|}(x)$ is made of just one smooth graphical disk (with respect to some plane).
    \end{enumerate}
\end{proposition}
\begin{proof}
    
    Take $\kappa_0>0$ small. Then for each $x\in\Sigma$ consider the unique disk in $\Sigma \cap B_{\delta_0|x|}(x)$ passing through $x$ to be $D$. Now take $U$ to be the set where $\Sigma \cap B_{\delta_0|x|}(x)$ has at least two connected components and the nearest disk in $\Sigma \cap B_{\delta_0|x|}(x)$ other than $D$ can be expressed as a normal graph $x+u\nu_D(x)$ over $D$ with $u < \kappa$ and define $\mathbf{s}(x)$ as the minimum height of the normal graph at $x$ (since there might be two on both sides).
    
    First by definition $U$ is open, since for any point $x\in U$ for a neighborhood nearby, all points have distance less than $\kappa|x|$. Continuity follows by similar arguments and using the structure given by \cref{prop:annuli-structure}.

    It remains to show the differential inequality. Take a point $x\in U$ and take the disk passing through $x$ in the ball of $B_{\delta_0|x|}(x)$ to be $D$ and take one of the normal graphs on one side as $(x,x+u\nu_\Sigma)$. Then we linearize the $\F$-stationary equation to see:
    \begin{align}
        \div_\Sigma (\Psi_F \nabla u) + \tr_\Sigma(\Psi_F A_\Sigma^2)u + Q_F(u,A) + R_F(u)= 0\,.
    \end{align}
    While the nonlinear term is bounded by:
    \begin{align}
        Q_F(u,A) \leq C_{\Lambda}\left[(|u||A| + |\nabla u|)\left(|\nabla^2 u| + |A||\nabla u| + |A|^2|u|\right) + |u||\nabla A||\nabla u|\right]\,,
    \end{align}
    and the divergence:
    \begin{align}
        |\div_{\Sigma}(\Psi_F)\cdot\nabla u| \leq C_{F}|A||\nabla u|\,.
    \end{align}
    Moreover the remainder $|R_F(u)| \leq {|u|}/{|x|}$. We treat $(|u||A| + |\nabla u|)|\nabla^2 u|$ as a perturbation for the non-divergence form operator and by scaling and compactness we see that $|\nabla A| \leq C_K/|x|^2$. Then with standard elliptic estimates and the Harnack inequality, we can deduce that locally:
    \begin{align}
        |\nabla u|(x) \leq \sup_{B_{\delta_0|x|}(x)} |\nabla u| \leq C\sup_{B_{\delta_0|x|}(x)}\frac{u}{|x|} \leq C\inf_{B_{\delta_0|x|}(x)}\frac{u}{|x|} \leq C\frac{u(x)}{|x|}\,.
    \end{align}
    With constants $C(K,\F)>0$ depending on the bound on $|x|A$ and the anisotropy. This means that in fact:
    \begin{align}
        |\nabla \log(u)(x)| \leq \frac{C(K,\F)}{|x|}\,.
    \end{align}
    Now we define $\mathbf{s}$ to be the minimum of the graphs that possibly exist on two sides. Since $\log(\min\{\cdot,\cdot\}) = \min\{\log(\cdot),\log(\cdot)\}$, we get the same Lipschitz bound on $\log(\mathbf{s})$.

    Moreover, using Schauder estimates for nondivergence form elliptic operators we also gather that \[|\nabla^2u| \leq C\sup_{B_{\delta_0|x|}}\frac{u}{|x|^2} \leq C\frac{u(x)}{|x|^2}.\] Putting all the information together and estimating the nonlinear terms, we gather at the differential inequality \cref{eq:separation-PDE}.
\end{proof}

\subsection{Super-exponential decay of sub-level sets of the separation}
In this section we prove that sub-level sets of the separation decay super-exponentially. We introduce the function:
\begin{align}
    \mathbf{v}(x) = \log\left(\frac{|x|}{\mathbf{s}(x)}\right)\,,
\end{align}
and the super-level sets (sub-level sets of ${\mathbf{s}}/{|x|}$):
\begin{align}
    S_k := \left\{\mathbf{v} \geq k\right\} = \left\{\mathbf{s} \leq e^{-k}|x|\right\}\cap \B_{\rho}\,.
\end{align}
Here $\rho \leq \rho_H$ in \cref{lem:hardy}. We aim to show the following:
\begin{proposition}[Super-exponential decay of sub-level sets]\label{prop:superexponential-decay}
    Let $F$ and $\Sigma$ be as in \cref{thm:removable-singularity-with-jacobi-lower-bound}. Then the following estimate holds:
    \begin{align}
        \hau^2(S_k) \leq (1+\theta_0)^{-e^{k}}\,,
    \end{align}
    for some $\theta_0>0$ and all large enough $k>k_0$.
\end{proposition}

First, we show that $\mathbf{v}$ satisfies a suitable differential inequality and that $S_k$ decays exponentially.
\begin{proposition}\label{prop:log-separation-PDE-Lp}
    Let $F$ and $\Sigma$ be as in \cref{thm:removable-singularity-with-jacobi-lower-bound}, then:
    \begin{align}\label{eq:PDE-log-separation}
        -\mathrm{div}_\Sigma\left(|x|^2 \Psi_F\nabla_\Sigma \mathbf{v}\right) \leq Ce^{-\mathbf{v}} + C|x|\,,
    \end{align}
    for all $x\in\Sigma\cap\B_\rho$ for some $0<\rho \leq \rho_H$ small enough ($\rho_H$ defined in \cref{lem:hardy}). Moreover, for any $1\leq p < \infty$ we have:
    \begin{align}\label{eq:integrability-log-separation}
        \mathbf{v} \in L^p(\Sigma\cap \B_\rho)\,.
    \end{align}
\end{proposition}
Before proceeding to the proof, we note that in the toy case of the PDE \cref{eq:PDE-log-separation} in $\R^2$, $\log\log(1/|x|)$ is a canonical solution. This motivates further the Super-exponential decay of the super level-sets in \cref{prop:superexponential-decay}.
\begin{proof}[Proof of \cref{prop:log-separation-PDE-Lp}]
    We use the ground state identity and use \cref{eq:separation-PDE} and \cref{eq:Jacobi-lower-bound-condition}:
    \begin{align}\label{eq:separation-ratio-PDE}
        \mathrm{div}_\Sigma\left(|x|^2 \Psi_F\nabla_\Sigma \left(\frac{\mathbf{s}}{|x|}\right)\right) = |x|\mathcal{J}_F\mathbf{s} - \mathbf{s}\mathcal{J}_F|x| \leq C_\F\left[\left(\frac{\mathbf{s}}{|x|}\right)^2 + \mathbf{s}\right]
    \end{align}
    Then \cref{eq:PDE-log-separation} follows immediately.
    Define:
    \begin{align}
        \mathbf{u} = \frac{\mathbf{s}}{|x|}\,.
    \end{align}
    For large enough $k\geq k_0$ we can assert that $\overline{S_k}\subset \B_\rho\cap\Sigma$, hence $(e^{-k} - \mathbf{u})_+$ is compactly supported in $B_\rho\cap\Sigma$. Since $|\mathbf{u}| \leq \kappa$, we test $\cref{eq:separation-ratio-PDE}$ with $\phi^2(e^{-k} - \mathbf{u})_+$ and integrate by parts and use the inequality \cref{lem:hardy} (note that $S_k\subset \B_\rho\subset \B_{\rho_H}$) to see that:
    \begin{align}
        \int_{S_k} \phi^2|(e^{-k}-\mathbf{u})_+|^2 \leq C\int_{S_k} \phi^2\mathbf{u}(\mathbf{u}+C|x|)(e^{-k} - \mathbf{u})_+ + |x|^2|\nabla\phi|^2(e^{-k}-\mathbf{u})_+^2\,,
    \end{align}
    for a function $\phi$ which is one outside of $B_{2\epsilon}$ and vanishes in $B_\epsilon$. Indeed $|x||\nabla\phi| \leq C$ on its support and $\mathbf{u}$ is integrable, so we can take $\epsilon\to0$ and see:
    \begin{align}
        \int_{S_k} |(e^{-k}-\mathbf{u})_+|^2 \leq C\int_{S_k} \mathbf{u}(\mathbf{u}+C|x|)(e^{-k} - \mathbf{u})_+\,.
    \end{align}
    Using Chebyshev, we estimate (bounding $\mathbf{u}+C|x| \leq C$):
    \begin{align}
        \hau^2(S_{k+1}) &\leq Ce^{2k}\int_{S_k}|(e^{-k}-\mathbf{u})_+|^2 \leq Ce^{2k}\int_{S_k} \mathbf{u}(e^{-k} - \mathbf{u})_+\\
        &\leq Ce^{2k}\sum_{i=k}^{\infty} e^{-i-k}\hau^2(S_i\setminus S_{i+1})\,.
    \end{align}
    The above means that there is some fixed $\beta$ such that:
    \begin{align}
        \hau^2(S_{k+\beta}) \leq C\hau^2(S_{k}\setminus S_{k+\beta})\,.
    \end{align}
    Then with a hole-filling argument we gather that there exists some $0<\theta < 1$ such that:
    \begin{align}
        \hau^2(S_{k+\beta}) \leq \theta \hau^2(S_k),
    \end{align}
    which implies that there exists another $0<\alpha<1$ such that for all $k>0$:
    \begin{align}
        \hau^2(S_k) \leq C\alpha^{k}\,.
    \end{align}
    From this we can estimate:
    \begin{align}
        \int_{\Sigma\cap\B_\rho} \mathbf{v}^p \leq  \sum_{k=0}^\infty k^p\hau^2(S_{k}\setminus S_{k+1}) \leq C\sum_{k=0}^\infty k^p\alpha^k < \infty\,.
    \end{align}
    This is indeed the desired conclusion.
\end{proof}

\begin{proof}[Proof of \cref{prop:superexponential-decay}]
    Take $k_0>0$ large enough such that for all $k\geq k_0$, $(\mathbf{v} - k)_+$ is compactly supported in $\B_\rho$. Then test \cref{eq:PDE-log-separation} with $\phi^2(\mathbf{v} - k)_+^{p-1}$ and integrate by parts and use \cref{lem:hardy} to see:
    \begin{align}
    \begin{aligned}
        \int_{\B_\rho\cap\Sigma} \phi^2(\mathbf{v}-k)_+^p \leq Cp\int_{\B_\rho\cap\Sigma} \phi^2(e^{-\mathbf{v}}+|x|)(\mathbf{v}-k)_+^{p-1} + |x|^2|\nabla\phi|^2(\mathbf{v}-k)^{p}_+\,,
        \end{aligned}
    \end{align}
    for a function $\phi$ vanishing in $B_{\epsilon}$ and one outside $B_{2\epsilon}$. Since $|x||\nabla\phi| \leq C$ on its support and $(\mathbf{v}-k)_+^p \in L^1(\Sigma\cap \B_\rho)$, we can take $\epsilon\to0$ and conclude:
    \begin{align}
        \int_{\B_\rho\cap\Sigma}(\mathbf{v}-k)_+^p \leq Cp\int_{\B_\rho\cap\Sigma} \phi^2(e^{-\mathbf{v}}+|x|)(\mathbf{v}-k)_+^{p-1}\,.
    \end{align}
    Now we deal with each term separately:
    \begin{align}
        \int_{\B_\rho\cap\Sigma} e^{-\mathbf{v}}(\mathbf{v}-k)_+^{p-1} \leq e^{-k}\left(\int_{\B_\rho\cap\Sigma} (\mathbf{v}-k)_+^{p}\right)^{\frac{p-1}{p}}\,.
    \end{align}
    Then we estimate the second term in turn using Michael--Simon inequality, \cref{eq:curvature-estimate} and \cref{eq:log-separation-bound}:
    \begin{align}
    \begin{aligned}
        \int_{\B_\rho\cap\Sigma} |x|&(\mathbf{v}-k)_+^{p-1} \leq \hau^2(S_k)^{\frac12}\left(\int_{\B_\rho\cap\Sigma} |x|^2(\mathbf{v}-k)_+^{2p-2}\right)^{\frac12}\\
        &\leq Cp\hau^2(S_k)^{\frac12}\int_{\B_\rho\cap\Sigma} (\mathbf{v}-k)_+^{p-1} \leq Cp\hau^2(S_k)^{\frac12+\frac1p}\left(\int_{\B_\rho\cap\Sigma} (\mathbf{v}-k)_+^{p}\right)^{\frac{p-1}{p}}\,.
    \end{aligned}
    \end{align}
    Putting the last three displays together we gather that:
    \begin{align}\label{eq:recursive-estimate-1}
        \left(\int_{\B_\rho\cap\Sigma} (\mathbf{v}-k)_+^p\right)^{\frac1p} \leq C_0pe^{-k} + C_1p^2\hau^2(S_k)^{\frac12+\frac1p}\,.
    \end{align}
    Now consider the numbers:
    \begin{align}
        \pi_k = \sum_{i=1}^{k} \left(8C_1 e^{2i}\right)^{\frac{e}{e^i+1}}\,
    \end{align}
    and define the sequence:
    \begin{align}
        a_k = \left(\int_{\B_\rho\cap\Sigma} (\mathbf{v}-\pi_k - \beta)_+^{2e^{k}}\right)^{\frac{1}{2e^{k}}},
    \end{align}
    for some $\beta>0$ to be chosen later. By Chebyshev's inequality we see that:
    \begin{align}
        \hau^2(S_{\pi_{k}+\beta}) \leq (\pi_k - \pi_{k-1})^{-2e^{k-1}}\int_{\B_\rho\cap\Sigma}(\mathbf{v}-\pi_{k-1}-\beta)_+^{2e^{k-1}} \leq \left(8C_1 e^{2i}\right)^{-\frac{e}{e^k+1}} a_{k-1}^{2e^{k-1}}\,.
    \end{align}
    Take $\beta \geq \log(4C_0) + \sup_{k\geq 1}(k-\pi_k)$. Putting this into \cref{eq:recursive-estimate-1}, we gather:
    \begin{align}
        a_k \leq \frac{1 + a_{k-1}^{e^{k-1} + \frac1e}}{2}\,.
    \end{align}
    With a simple induction we can see that if $a_k < 1$ we see that:
    \begin{align}
        a_{k+1} \leq \frac{1 + a_{k}^{e^{k}+\frac1e}}{2} < 1\,.
    \end{align}
    Since $\mathbf{v} \in L^{2e^{k}}(\B_\rho\cap\Sigma)$, we can take $\beta>0$ even larger so that $a_{k_0} < 1$ (since by \cref{prop:log-separation-PDE-Lp} $\mathbf{v}$ is integrable) and induction implies that for all $k>k_0$ we have $a_k \leq 1$. Now note that by a simple calculation:
    \begin{align}
        |\pi_k - k| \leq C_2\quad \text{ for all } k > 1\,.
    \end{align}
    This means that:
    \begin{align}
        \int_{\B_\rho\cap\Sigma}(\mathbf{v} - k- \beta - C_2)_+^{2e^{k}} \leq 1\,.
    \end{align}
    Now we apply Chebyshev again to see that:
    \begin{align}
        \hau^2(S_{k + \beta + e^{1/2} + C_2}) \leq e^{-e^{k}} \int_{\B_\rho\cap\Sigma} (\mathbf{v} - k - \beta - C_2)^{2e^{k}} \leq e^{-e^{k}}\,.
    \end{align}
    This is indeed the desired conclusion.
\end{proof}

\subsection{The logarithmic density blowup: proof of \cref{thm:removable-singularity-with-jacobi-lower-bound}}
Before proceeding to the proof we need a simple lemma:
\begin{lemma}[Key lemma]\label{lem:key-lemma}
    There exists $C(\theta_0),r_0>0$ such that for $0<r<r_0$:
    \begin{align}
        \sum_{k=1}^{\infty} \min\left\{e^k , \frac{(1+\theta_0)^{-e^k}}{r^2}\right\} \leq C(\theta_0)|\log(r)|\,.
    \end{align}
\end{lemma}
\begin{proof}
    First note that $e^k$ is increasing and $(1+\theta_0)^{-e^{k}}$ is decreasing. Now take the largest $k_0$ such that:
    \begin{align}
        r^{2} \leq (1+\theta_0)^{-e^{k_0}}\Rightarrow 
        k_0 \leq \log(\log(1/r)) + \log(2) - \log|\log(1+\theta_0)| \,.
    \end{align}
    provided $k_0>0$ is large and $r_0>0$ is small enough. Now plugging this in to the equation we gather:
    \begin{align}
        \sum_{k=1}^{\infty} \min\left\{e^k , \frac{(1+\theta_0)^{-e^k}}{r^2}\right\} &= \sum_{k=1}^{k_0} e^k  + \sum_{k=k_0+1}^{\infty} \frac{(1+\theta_0)^{-e^k}}{r^2}\\
        &\leq C|\log(r)| + 2 \,.
    \end{align}
    The last bound follows the fact that the sum of super-exponential series is controlled by their first term. The desired conclusion follows once $r_0>0$ is chosen small enough.
\end{proof}

We also need a simple and important lemma:
\begin{lemma}\label{lem:sublevelset-density-estimate}
    There exists $C(K,\F)>0$ such that for any $k\geq0$ and $0<r\leq 1$ we have:
    \begin{align}
        \hau^2\left(\mathrm{An}_{r}\cap (S_k\setminus S_{k+1}) \right) \leq Ce^{k} r^2\,.
    \end{align}
\end{lemma}
\begin{proof}
    The set $S_k\setminus S_{k+1}$ is the set where the separation is between $e^{-k}$ and $e^{-k-1}$. Now assume by contradiction that at some ball $B_{\delta_0|x|}(x)$ we have:
    \begin{align}
        \hau^2(B_{\delta_0|x|}(x) \cap S_k \setminus S_{k+1})/|x|^2 \gg e^k\,.
    \end{align}
    By \cref{prop:annuli-structure}, $\Sigma\cap B_{\delta_0|x|}(x)$ is made up of finitely many sheets with small Lipschitz bound over some plane $P$. Now if the density of $S_k \setminus S_{k+1}$ at this ball is $\gg e^{k}$, then this means there exists a point in $P$ such that the projection on $P$ of $B_{\delta_0|x|}(x) \cap S_k \setminus S_{k+1}$ is $\gg e^k$. This means that there is at least one point in its preimage that has separation $\ll e^{-k}$ which is a clear contradiction.
\end{proof}

Now we are in a position to prove \cref{thm:removable-singularity-with-jacobi-lower-bound}. First we show:
\begin{proposition}[Logarithmic blowup of density]\label{prop:log-bound-density}
    Under the assumptions of \cref{thm:removable-singularity-with-jacobi-lower-bound}, there is some constant $C(K,\F)>0$ such that:
    \begin{align}\label{eq:log-bound-density}
        \frac{\hau^2(\Sigma \cap \mathrm{An}_r)}{r^2} \leq C|\log(r)|\,,
    \end{align}
    for small enough $0<r<r_0$.
\end{proposition}
\begin{proof}
    We estimate using \cref{prop:superexponential-decay}, \cref{lem:sublevelset-density-estimate} and \cref{lem:key-lemma}:
    \begin{align}
        \frac{\hau^2(\Sigma \cap \mathrm{An}_r)}{r^2} &= \frac{\hau^2((\Sigma \cap \mathrm{An}_r)\setminus U) + \sum_{k=k_0}^\infty \hau^2(\Sigma \cap \mathrm{An}_r\cap  (S_k\setminus S_{k+1}))}{r^2}\\
        &\leq C + C\sum_{k=k_0}^{\infty} \min\left\{e^k , \frac{(1+\theta_0)^{-e^k}}{r^2}\right\} \leq C|\log(r)|\,,
    \end{align}
    provided $r>0$ is chosen small enough.
\end{proof}

\begin{corollary}\label{cor:capacity-zero}
    Under the conditions of \cref{thm:removable-singularity-with-jacobi-lower-bound}, the singularity at the origin is of zero capacity.
\end{corollary}
\begin{proof}
Fix \(j\geq 2\). We use the following \(\log\log\) cut-off:
\begin{align}
\phi_j(x)
=
\begin{cases}
0,
& 0\leq |x| \leq 2^{-j}, \\[4pt]
\dfrac{\log(j\log 2)-\log\log\frac{1}{|x|}}{\log j},
& 2^{-j} \leq |x| \leq 2^{-1}, \\[8pt]
1,
& 2^{-1}\leq |x|.
\end{cases}
\end{align}
Indeed, at the two transition radii one has
\[
\phi_j(2^{-j})=0,
\qquad
\phi_j\left(2^{-1}\right)=1.
\]
Moreover, on the annular region \(B_{1/2}\setminus B_{2^{-j}}\),
\[
|\nabla \phi_j(x)|
=
\frac{1}{\log j}\,
\frac{1}{|x|\log\frac{1}{|x|}}.
\]
Hence
\[
|\nabla_\Sigma \phi_j|
\leq
|\nabla \phi_j|
\leq
\frac{1}{\log j}\,
\frac{1}{r|\log r|},
\qquad r=|x|.
\]
Therefore,
\begin{align}
\begin{aligned}
\int_{\Sigma} |\nabla_\Sigma \phi_j|^2\,\mathrm{d}\mathcal H^2
&\leq
\frac{1}{|\log j|^2}
\int_{\Sigma \cap (B_{1/2}\setminus B_{2^{-j}})}
\frac{1}{r^2|\log r|^2}\,\mathrm{d}\mathcal H^2  \\
&=
\frac{1}{|\log j|^2}
\sum_{i=1}^{j-1}
\int_{\Sigma \cap \mathrm{An}_{2^{-i}}}
\frac{1}{r^2|\log r|^2}\,\mathrm{d}\mathcal H^2 .
\end{aligned}
\end{align}
On \(\mathrm{An}_{2^{-i}}=B_{2^{-i}}\setminus B_{2^{-i-1}}\), we have
$r\sim 2^{-i}$, $|\log r|\sim i$.
Consequently,
\[
\frac{1}{r^2|\log r|^2}
\leq
\frac{C}{2^{-2i}i^2}
\qquad
\text{on } \mathrm{An}_{2^{-i}}.
\]
Using \cref{eq:log-bound-density}, we obtain
\begin{align}
\begin{aligned}
\int_{\Sigma} |\nabla_\Sigma \phi_j|^2\,\mathrm{d}\mathcal H^2
&\leq
\frac{C}{|\log j|^2}
\sum_{i=1}^{j-1}
\frac{\mathcal H^2(\Sigma \cap \mathrm{An}_{2^{-i}})}
{2^{-2i}i^2} \\
&\leq_{\cref{eq:log-bound-density}}
\frac{C}{|\log j|^2}
\sum_{i=1}^{j-1}
\frac{|\log(2^{-i})|}{i^2} \leq
\frac{C}{|\log j|^2}
\sum_{i=1}^{j-1}
\frac{1}{i}\leq
\frac{C}{\log j}.
\end{aligned}
\end{align}
Letting \(j\to\infty\), we get
\[
\int_{\Sigma} |\nabla_\Sigma \phi_j|^2\,\mathrm{d}\mathcal H^2
\to 0.
\]
Since \(\phi_j\equiv 0\) in a neighbourhood of the origin and
\(\phi_j\equiv 1\) outside \(B_{1/2}\), this gives a sequence of admissible
cut-off functions whose Dirichlet energy tends to zero. Since the construction is local and scale-invariant, replacing \(x\) by
\(x/\rho\) gives the same conclusion in \(B_{\rho}\setminus B_{\rho 2^{-j}}\)
for every sufficiently small \(\rho>0\). Therefore the origin has zero
capacity.
\end{proof}

\begin{proof}[Proof of \cref{thm:removable-singularity-with-jacobi-lower-bound}]
    By \cref{cor:capacity-zero} we gather that $\Sigma$ is $\F$-stable in all of $\B_1$. Then \cite[Theorem 2.8]{DePhilippisDeRosa} implies that $\Sigma$ is smooth at the origin and we are done. 
    We can also argue as in \cite[Proof of Theorem 5.1, Step 3]{DDL}: one obtains the finite total curvature bound near the puncture and then applies \cite[Theorem 2, Page~250]{White} to remove the singularity.
\end{proof}

\appendix

\crefalias{section}{appendix}

\section{Auxiliary lemmas}
Here we prove a simple lemma regarding density of surfaces near boundary points:
\begin{lemma}\label{lem:density-boundary-estimate}
    Let $A\subset N$ be a smooth open simply connected domain and let $\{M_k\}_{k=1}^{\infty}\subset \mathcal{M}$ be a sequence of surfaces with $\de M_k\subset N\setminus A$. Moreover, take a relatively open subset $S\subset \de A$ such that $\Gamma_k = \de M_k \cap S$ is a simple smooth arc and $\Gamma_k$ converges smoothly to a limit simple arc $\Gamma\subset S$ as $k\to\infty$. Then, for every \(x\in\Gamma\), there exist constants \(\epsilon_0=\epsilon_0(x,\Gamma,A,S,N)>0\) and \(C>0\) such that, for all \(0<r\le \epsilon_0\) and all sufficiently large \(k\),
    \[
        \mathcal H^2(M_k\cap B_r(x)\cap A)\ge C r^2 .
    \]
\end{lemma}
\begin{proof}
Fix \(x\in\Gamma\). We work in normal coordinates centered at \(x\),
and choose \(\epsilon_0>0\) sufficiently small so that the metric is
uniformly comparable to the Euclidean metric in \(B_{\epsilon_0}(x)\).
Let \(\tau\) be the unit tangent vector to \(\Gamma\) at \(x\). After
shrinking \(\epsilon_0\), we may assume that, for every
\(0<r\le \epsilon_0\), the arc \(\Gamma\cap B_r(x)\) is contained in a
slab of thickness \(r/10\) around the line \(x+\mathbb R\tau\), and that
\(\partial A\cap B_r(x)\) is contained in a slab of thickness \(r/10\)
around \(T_x\partial A\).
For \(|t|<r\), let $\pi^{(t)}:=x+t\tau+\tau^\perp$
be the affine plane orthogonal to \(\tau\). Since
\(\Gamma_k\to\Gamma\) smoothly, for all large \(k\) and for all
\(|t|\le 3r/4\), the plane \(\pi^{(t)}\) intersects
\(\Gamma_k\cap B_r(x)\).
For a.e. such \(t\), the intersection
$\pi^{(t)}\cap M_k\cap A\cap B_r(x)$
is a finite union of smooth curves. At least one of these curves has an
endpoint on \(\Gamma_k\). Since, in \(B_r(x)\), the only contact of
\(M_k\) with \(\partial A\) is along \(\Gamma_k\), and since \(M_k\) is
embedded, this curve must either leave \(B_r(x)\) or meet another point
of \(\Gamma_k\). For \(\epsilon_0\) sufficiently small, the second
possibility is excluded for \(|t|\le 3r/4\), because
\(\pi^{(t)}\) intersects the arc \(\Gamma_k\) locally in a single point.
Hence the curve must reach
$\partial B_r(x)\cap A\cap \pi^{(t)}$.
By the slab estimates above, its length is bounded from below by
\(c r\), where \(c>0\) depends only on \(A,\Gamma,S\), and the local
geometry of \(N\). Therefore
\[
    \mathcal H^1
    \bigl(
        M_k\cap \pi^{(t)}\cap B_r(x)\cap A
    \bigr)
    \ge
    c r\quad\text{for a.e. } |t|\le \frac{3r}{4}.
\]
Applying the coarea formula to the projection onto the \(\tau\)-direction,
we get
\[
\begin{aligned}
    \mathcal H^2(M_k\cap A\cap B_r(x))
    \ge
    c
    \int_{-3r/4}^{3r/4}
    \mathcal H^1
    \bigl(
        M_k\cap \pi^{(t)}\cap B_r(x)\cap A
    \bigr)\,\mathrm{d}t \ge
    C r^2 .
\end{aligned}
\]
This proves the lemma.
\end{proof}

In the next lemma we prove a weighted estimate on $\F$-stationary surfaces:

\begin{lemma}[Anisotropic Hardy-type inequality]
\label{lem:hardy}
Let \(F\) be a smooth uniformly elliptic integrand on
\(G_2(\B_1)\), and let \(G=G(x,\nu)\) be the associated even
one-homogeneous normal integrand. Assume that \(G\) satisfies the local
versions of \eqref{H:comparability}--\eqref{H:ellipticity} in \(\B_1\).
Then there exist constants \(\rho_H>0\) and \(C_G<\infty\), depending
only on the local ellipticity and \(C^1\)-bounds of \(G\), such that the
following holds:

If $\Sigma\subset \B_{\rho_H}\setminus\{0\}$ is a smooth \(\F\)-minimal surface, then for every
\(f\in C_c^1(\Sigma\setminus\{0\})\),
\[
    \int_\Sigma |f|^2
    \le
    C_G
    \int_\Sigma |x|^2|\nabla_\Sigma f|^2 .
\]
Equivalently, the same conclusion holds for any smooth
\(\F\)-minimal surface \(\Sigma\subset\B_1\setminus\{0\}\), provided $\operatorname{spt} f\subset \Sigma\cap\B_{\rho_H}$.
\end{lemma}

\begin{proof}
Let \(\nu\) be a local choice of unit normal on \(\Sigma\). Define the
tangential vector field
\[
    W_G
    :=
    G(x,\nu)x^T
    -
    (x\cdot\nu)\nabla_{\mathbb S^2}G(x,\nu),
\]
where \(x^T\) denotes the tangential projection of the Euclidean position
vector onto \(T\Sigma\). Since \(G\) is even, this vector field is
independent of the choice of \(\nu\).

On an \(\F\)-minimal surface, we have the identity
\[
    \operatorname{div}_\Sigma W_G
    =
    2G(x,\nu)+x\cdot D_xG(x,\nu).
\]
By \eqref{H:comparability} and \(C^1\)-bounds for \(G\) (i.e. \eqref{H:regularity}), we know
\[
    G(x,\nu)\ge \lambda,
    \qquad
    |D_xG(x,\nu)|\le C_G .
\]
Hence, after choosing \(\rho_H>0\) sufficiently small, we have
\[
    2G(x,\nu)+x\cdot D_xG(x,\nu)
    \ge
    \lambda
\]
for all \(|x|<\rho_H\) and all \(\nu\in\mathbb S^2\).

Now let \(f\in C_c^1(\Sigma\setminus\{0\})\) with
\(\operatorname{spt}f\subset \B_{\rho_H}\). Since \(W_G\) is tangential
to \(\Sigma\), integration by parts gives
\[
\begin{aligned}
    \lambda\int_\Sigma f^2
    \le
    \int_\Sigma f^2\operatorname{div}_\Sigma W_G       =
    -2\int_\Sigma f\langle W_G,\nabla_\Sigma f\rangle \le
    2\int_\Sigma |f|\,|W_G|\,|\nabla_\Sigma f|.
\end{aligned}
\]
Moreover, by the \(C^1\)-bounds for \(G\),
we have $|W_G|
    \le
    C_G|x|$.
Therefore, by Cauchy--Schwarz,
\[
\begin{aligned}
    \lambda\int_\Sigma f^2
    &\le
    C_G
    \left(\int_\Sigma f^2\right)^{1/2}
    \left(\int_\Sigma |x|^2|\nabla_\Sigma f|^2\right)^{1/2}.
\end{aligned}
\]
Dividing by \(\left(\int_\Sigma f^2\right)^{1/2}\), unless \(f\equiv0\),
gives
\[
    \int_\Sigma f^2
    \le
    C_G
    \int_\Sigma |x|^2|\nabla_\Sigma f|^2 .
\]
The case \(f\equiv0\) is trivial. This proves the lemma.
\end{proof}

\section{Proof of the anisotropic optimal lifting lemma}
\label{app:optimal-lifting-proof}

In this appendix, we provide a detailed proof of \cref{prop:optimal-curve-lifting}. The argument follows Ketover's proof of \cite[Proposition~2.2]{K} closely. Apart from two modifications required in the anisotropic setting, namely a smoothing lemma for $\mathbf{F}$-almost minimizing surfaces and an anisotropic no-folding estimate in the overlap wedge, the proof is unchanged.

\begin{proof}[Proof of \cref{prop:optimal-curve-lifting}]
The proof follows Ketover's proof of \cite[Proposition~2.2]{K}. We give the details of
the reduction and indicate the precise anisotropic changes.

\smallskip

\noindent\emph{Reduction to one component and one curve.}
Since the limiting components \(\Sigma^{(k)}\) are pairwise disjoint, we may choose
disjoint tubular neighborhoods of them and work component by component. Thus, after restricting to a fixed tubular neighborhood of one component and suppressing
the other components from the notation, we may assume that in this neighborhood
\[
    \Sigma_j \to n\Gamma
\]
as varifolds, where \(\Gamma\) is connected, smooth, embedded and \(\F\)-minimal. We also
first treat a single simple closed curve \(\gamma\subset \Gamma\).

By \cref{prop:aniso-am-minmax}, after passing to a subsequence, the surfaces
\(\Sigma_j\) are almost minimizing in sufficiently small annuli centered at points of
\(\Gamma\). As in \cite[Section~4.2]{K}, a finite covering argument converts this annular
almost minimizing property into the following local statement away from finitely many
points: there is a finite set \(\mathcal S\subset \Gamma\) such that for every
\(x\in \Gamma\setminus\mathcal S\), there is \(\rho(x)>0\) such that \(\Sigma_j\) is
\(1/j\)-almost \(\F\)-minimizing in \(B_\tau(x)\) for every \(\tau\le \rho(x)\) and all large
\(j\). Indeed, if \(x\) is not one of the centers of the finite annular cover, then for some
annulus \(A\) in which \(\Sigma_j\) is \(1/j\)-almost \(\F\)-minimizing one can choose
\(\rho(x)>0\) so that \(B_{\rho(x)}(x)\subset A\). Since almost minimizing is inherited by
subsets, the claim follows.

By the genus collapse lemma \cite[Lemma~I.0.14]{CM-genus-collapse-ref}, after passing to
a further subsequence there is another finite set \(\mathcal G\subset \Gamma\) such that,
for every \(x\in\Gamma\setminus\mathcal G\), there exists \(r'(x)>0\) with
\[
    \operatorname{genus}\bigl(\Sigma_j\cap B_\tau(x)\bigr)=0
\]
for all \(\tau\le r'(x)\) and all large \(j\). Set
\[
    \mathcal B:=\mathcal S\cup \mathcal G.
\]
We perturb \(\gamma\) inside \(\Gamma\cap T_\epsilon(\gamma)\), preserving its homotopy
class in \(\Gamma\), so that \(\gamma\cap\mathcal B=\emptyset\).
We still denote the perturbed curve by \(\gamma\).

Choose \(\epsilon_0>0\) smaller than the normal injectivity radius of every limiting
component, smaller than one tenth of the mutual distance between distinct limiting
components, smaller than the ambient convexity radius, and smaller than the scales needed
for the anisotropic replacement and boundary regularity theorems. Fix \(0<\epsilon<\epsilon_0\).

\smallskip

\noindent\emph{The chain of balls.}
For \(r>0\) sufficiently small, choose \(2q\) points
$p_0,p_1,\ldots,p_{2q-1}$
ordered cyclically along \(\gamma\), with \(p_{2q}=p_0\), such that consecutive points are
separated by arclength comparable to \(r\) along \(\gamma\). After a further small homotopy supported in
\(T_\epsilon(\gamma)\), we may assume that the segment of \(\gamma\) between \(p_i\) and
\(p_{i+1}\) is the geodesic segment in \(\Gamma\) joining them. By choosing $r$ small enough, these large balls
are all contained still in $T_\epsilon(\gamma)$.

Set
\[
    r_1:=\frac{3r}{4},
    \qquad
    r_2:=\frac{15r}{8},
\]
and write
\[
    B_i:=B_{r_1}(p_i),
    \qquad
    B_i(1-\delta):=B_{(1-\delta)r_1}(p_i).
\]
For \(r\) sufficiently small, the following properties hold:
\begin{enumerate}
    \item \(B_i\) intersects only \(B_{i-1}\) and \(B_{i+1}\), with indices modulo \(2q\).
    \item For every \(s=0,\ldots,q-1\),
    $\overline B_{2s}\cup \overline B_{2s+1}\cup \overline B_{2s+2}
        \subset B_{r_2}(p_{2s+1})$.
    \item If \(i\) is even and \(B_i\cap B_{r_2}(p_{2s+1})\neq\emptyset\), then
    \(i=2s\) or \(i=2s+2\).
    \item \(\Sigma_j\) is \(1/j\)-almost \(\F\)-minimizing in \(B_{r_2}(p_{2s+1})\) for
    every \(s\) and all large \(j\).
    \item \(\Gamma\) intersects \(\partial B_t(p_i)\) transversely for every \(i\) and
    every \(t\) in a small neighborhood of \(r_1\).
    \item
    $\operatorname{genus}\bigl(\Sigma_j\cap B_{r_2}(p_{2s+1})\bigr)=0$
    for every \(s\) and all large \(j\).
\end{enumerate}
We define
\[
    B_E:=\bigcup_{s=0}^{q-1} B_{2s},
    \qquad
    B_O:=\bigcup_{s=0}^{q-1} B_{2s+1}.
\]
Moreover, by shrinking \(r\) further, we may assume that for every \(i\),
\[
    \exp_{p_i}^{-1}(\gamma\cap B_i)
    \subset
    T_{r/100}(L_i)
\]
for some line \(L_i\subset T_{p_i}N\). In other words, on the scale \(r\), the curve
\(\gamma\) cuts each ball \(B_i\) approximately along a straight line.

\smallskip

\smallskip

\noindent\emph{Step 1: first replacements in the even balls.}
By the choice of the balls, each even ball \(B_{2s}\) is contained in a large ball
\(B_{r_2}(p_{2s+1})\) where \(\Sigma_j\) is \(1/j\)-almost \(\F\)-minimizing. Since the
even balls are pairwise disjoint, we may perform the simultaneous replacement in
\[
    B_E=\bigcup_{s=0}^{q-1}B_{2s}.
\]
More precisely, let \(V_j\) be the limit of a minimizing sequence for
$\mathrm{Problem}\bigl(\Sigma_j,\mathrm{Is}_j(\Sigma_j,B_E)\bigr)$,
with respect to the anisotropic functional \(\F\). By the anisotropic Meeks--Simon--Yau
theory, namely
\cref{thm:Meeks-Simon-Yau-without-boundary} and \cref{thm:Meeks-Simon-Yau-with-boundary},
and by the replacement theory from \cref{prop:aniso-minmax-replacements}, \(V_j\) is
induced in \(B_E\) by a smooth embedded \(\F\)-stationary and \(\F\)-stable surface,
with multiplicity one for each fixed \(j\), still denoted by \(V_j\). Moreover \(V_j\) has
the same varifold limit as \(\Sigma_j\), and
$V_j\to n\Gamma$
smoothly and graphically on compact subsets of \(B_E\).

Hence, for fixed sufficiently small \(\delta>0\) and all sufficiently large \(j\),
$V_j\cap B_0(1-\delta)$
is a union of \(n\) normal graphs over \(\Gamma\cap B_0(1-\delta)\). By the transversality
assumption on \(\Gamma\cap\partial B_1\) and the smooth convergence \(V_j\to n\Gamma\),
the intersection
$ V_j\cap \partial B_1\cap B_0(1-\delta)$
is transverse for all large \(j\), and consists of \(n\) smooth arcs, denoted by
$\{\alpha_j^{i}\}_{i=1}^{n}$.
Without loss of generality, it suffices to explain how to pass from \(B_0\) to \(B_1\); the same argument then
propagates along the chain of balls.

\smallskip

\smallskip

\noindent\emph{Step 2: topological relation with the original surfaces.}
We next record the topological information carried by the first replacements. We shall use
the following anisotropic analogue of \cite[Proposition~4.7]{K}.

Let \(B\) be a sufficiently small ball and let \(S\subset B\) be a smooth embedded surface
with \(\partial S\subset \partial B\). Let \(S^\ell\) be a minimizing sequence for
$\mathrm{Problem}\bigl(S,\mathrm{Is}_j(S,B)\bigr)$
with respect to the anisotropic functional \(\F\), and suppose that
$S^\ell \to \Delta$
as varifolds, where \(\Delta\) is a smooth embedded multiplicity-one \(\F\)-minimal surface
with \(\partial\Delta=\partial S\). Then
\[
    \operatorname{genus}(\Delta)\leq \operatorname{genus}(S).
\]
Moreover, for every sufficiently small \(\eta>0\), after passing to \(\ell\) large, one can
perform finitely many \(\gamma\)-reductions, equivalently neck-pinch surgeries, supported
in \(B\), on \(S^\ell\) to obtain a surface \(\widehat S^\ell\) such that
\[
    \widehat S^\ell\subset T_\eta(\Delta),
    \qquad
    \widehat S^\ell=S^\ell
    \quad\text{in }T_{\eta/2}(\Delta).
\]

Indeed, the proof is the same as the proof of \cite[Proposition~4.7]{K}. The only analytic
input needed there is that the limit \(\Delta\) is smooth, embedded, multiplicity one and
has the same boundary as the minimizing sequence. In the present setting this is supplied
by the anisotropic Meeks--Simon--Yau theorem with boundary,
\cref{thm:Meeks-Simon-Yau-with-boundary}, together with the replacement theory from
\cref{prop:aniso-minmax-replacements}. Once this regularity is known, the rest of the
argument is purely topological: using varifold convergence and the coarea formula, one
performs neck-pinch surgeries to push the approximating surfaces into \(T_\eta(\Delta)\);
then the nearest-point projection \(T_\eta(\Delta)\to\Delta\) and the degree argument give
the genus inequality.

We now apply this statement to the minimizing sequence used in Step~1 to construct
\(V_j\). More precisely, in each even ball \(B_{2s}\), the replacement is obtained as the
limit of a local isotopy-minimizing sequence with the same boundary on \(\partial B_{2s}\).
Hence, after choosing a sufficiently far term in this minimizing sequence and relabeling,
we may regard \(V_j\) as a smooth surface obtained from \(\Sigma_j\) by finitely many
\(\gamma\)-reductions and isotopies inside the even balls, while preserving the graphical
properties obtained in Step~1.

Since every even ball \(B_{2s}\) is contained in the large ball
\(B_{r_2}(p_{2s+1})\), and since
$\Sigma_j\cap B_{r_2}(p_{2s+1})$
is planar by the setup, these surgeries cannot create genus in any of the relevant large
balls. Therefore \(V_j\) contains no local genus in the balls comprising \(B_E\cup B_O\).
Equivalently, for every \(i=0,\ldots,2q-1\),
$ V_j\cap B_i$
is planar, after passing to the chosen post-surgery representatives. This is the topological
information needed in the later Gauss--Bonnet estimate.

\smallskip

\smallskip

\noindent\emph{Step 3: curvature bounds on the first boundary arcs.}
Since
$\alpha_j^i\subset B_0(1-\delta)$,
we have
$ \operatorname{dist}(\alpha_j^i,\partial B_0)\ge \delta r_1.$.
The anisotropic stable curvature estimates for replacements, therefore, give
\[
    \sup_{\alpha_j^i}|A_{V_j}|
    \le
    C(\delta,r_1,\F,N).
\]
By transversality of \(\Gamma\) with \(\partial B_1\), and by the smooth convergence
\(V_j\to n\Gamma\) in \(B_0(1-\delta)\), the angle \(\vartheta_j\) between \(V_j\) and
\(\partial B_1\) along \(\alpha_j^i\) satisfies
\[
    |\sin\vartheta_j|\ge c>0
\]
for all sufficiently large \(j\). Thus, using the standard curvature estimate for the
intersection of two transverse surfaces,
\[
    |k_N(\alpha_j^i)|
    \le
    \frac{|A_{V_j}|+|A_{\partial B_1}|}{|\sin\vartheta_j|}.
\]
Since \(\partial B_1\) is a fixed geodesic sphere of radius \(r_1\), its second fundamental
form is uniformly bounded by a constant depending only on \(r_1\) and \(N\). Hence
\[
    \sup_{\alpha_j^i}|k_N|
    \le C,
    \qquad i=1,\ldots,n.
\]

\smallskip

\smallskip

\noindent\emph{Step 4: smoothing \(V_j\) and preserving almost minimizing.}
The surfaces \(V_j\) may fail to be smooth across the artificial boundary
\(\partial B_E\). We now smooth them in small collars of \(\partial B_E\), and show that
the resulting surfaces are still almost \(\F\)-minimizing in the odd large balls. This is
the anisotropic analogue of \cite[Lemma~4.1]{K}.

We first record the local smoothing statement needed below. Let \(B\) be a sufficiently
small ball, and let \(A\Subset B\) be a finite disjoint union of balls. Suppose that
\(S\) is a smooth embedded surface with \(\partial S\subset \partial B\), and that \(S\) is
\(\varepsilon_j\)-almost \(\F\)-minimizing in \(B\). Let
$\phi^\ell(1,S)$
be a minimizing sequence for
$\mathrm{Problem}\bigl(S,\mathrm{Is}_{\varepsilon_j}(S,A)\bigr)$
with respect to \(\F\), and suppose that
$ \phi^\ell(1,S)\to V$
as varifolds. Then, after smoothing \(V\) in an arbitrarily small collar of
\(\partial A\), one obtains a smooth surface \(\widehat V\) such that
\[
    \widehat V=V \quad\text{away from the smoothing collar},
\]
\[
    \F(\widehat V)-\F(V)\quad\text{is arbitrarily small},
\]
and \(\widehat V\) is \(\widetilde\varepsilon_j\)-almost
\(\F\)-minimizing in \(B\), for some \(\widetilde\varepsilon_j\to0\). After reindexing, we
shall again denote this sequence by \(1/j\).

Indeed, the proof is exactly the concatenation argument of \cite[Lemma~4.1]{K}, with
\(\mathcal H^2\) replaced by \(\F\). Let
$S^\ell:=\phi^\ell(1,S)$.
By definition of the minimizing problem, the isotopies \(\phi^\ell\) satisfy
\begin{equation}\label{eq:isotopy-minimizing}
     \F(\phi^\ell(t,S))
    \le
    \F(S)+\frac{\varepsilon_j}{8}
    \qquad \text{for all }t\in[0,1].
\end{equation}
The anisotropic Meeks--Simon--Yau theorem with boundary, \cref{thm:Meeks-Simon-Yau-with-boundary}, gives that \(V\) is smooth and
multiplicity one in \(A\), with the same boundary as the minimizing sequence on
\(\partial A\). Therefore one can smooth across \(\partial A\) by isotopies
\(\eta^\ell(t,\cdot)\), supported in an arbitrarily small collar of \(\partial A\), so that
\[
    (\eta^\ell\circ\phi^\ell)(1,S)\to \widehat V
\]
as varifolds, and the \(\F\)-energy increase along \(\eta^\ell\) is as small as prescribed.
In particular, by choosing the collar and then \(\ell\) sufficiently large, the concatenated
isotopy
$\eta^\ell\circ \phi^\ell$
still belongs to \(\mathrm{Is}_{\varepsilon_j}(S,B)\), up to an error which tends to zero
with \(j\).

Suppose, toward a contradiction, that \(\widehat V\) is not almost \(\F\)-minimizing in
\(B\). Then, for a sequence \(\sigma_j\to0\) chosen much larger than all smoothing and
approximation errors, there exists an isotopy \(\psi\) supported in \(B\) such that
\begin{equation}\label{eq:almost-minimizing-hatV-1}
    \F(\psi(t,\widehat V))
    \le
    \F(\widehat V)+\frac{\sigma_j}{8}
    \qquad\text{for all }t\in[0,1],
\end{equation}
while
\begin{equation}\label{eq:almost-minimizing-hatV-2}
    \F(\psi(1,\widehat V))
    \le
    \F(\widehat V)-\sigma_j.
\end{equation}
Since \((\eta^\ell\circ\phi^\ell)(1,S)\to\widehat V\) and \(\psi\) is smooth, the
continuity of \(\F\) under smooth isotopies gives, for \(\ell\) sufficiently large,
\begin{equation}\label{eq:isotopy-minimizing-2}
    \F\bigl(\psi(t,(\eta^\ell\circ\phi^\ell)(1,S))\bigr)
    \le
    \F(\psi(t,\widehat V))+o(\sigma_j)
    \qquad\text{for all }t\in[0,1].
\end{equation}
Combining \cref{eq:isotopy-minimizing,eq:almost-minimizing-hatV-1,eq:almost-minimizing-hatV-2,eq:isotopy-minimizing-2}, the concatenated isotopy
$ \psi\circ\eta^\ell\circ\phi^\ell$
is supported in \(B\), its intermediate \(\F\)-energy stays within the allowed
almost-minimizing height, and its final \(\F\)-energy decreases by more than the allowed
amount. This contradicts the almost \(\F\)-minimizing property of \(S\) in \(B\). Hence
\(\widehat V\) is almost \(\F\)-minimizing in \(B\), with a parameter tending to zero.

We now apply this local statement to the present situation. Fix an odd large ball
$B:=B_{r_2}(p_{2s+1}).$
By the choice of the ball system, the only even balls which meet \(B\) are
\(B_{2s}\) and \(B_{2s+2}\). Set
\[
    A:=B_{2s}\cup B_{2s+2}\subset B.
\]
Since \(\Sigma_j\) is \(1/j\)-almost \(\F\)-minimizing in \(B\), the preceding local
statement applies to the replacement of \(\Sigma_j\) in \(A\). Applying it successively for
\(s=0,\ldots,q-1\), and choosing the smoothing collars sufficiently small, we obtain
smooth surfaces \(\widehat V_j\) such that
$\widehat V_j\to n\Gamma$
as varifolds, \(\widehat V_j\) agrees with \(V_j\) away from arbitrarily small collars of
\(\partial B_E\), and \(\widehat V_j\) is almost \(\F\)-minimizing in each odd large ball
\(B_{r_2}(p_{2s+1})\). After reindexing the almost-minimizing parameter, we shall say
that \(\widehat V_j\) is \(1/j\)-almost \(\F\)-minimizing in these balls.

\smallskip

\smallskip

\noindent\emph{Step 5: second replacements in the odd balls.}
By Step~4, after reindexing, the surfaces \(\widehat V_j\) are \(1/j\)-almost
\(\F\)-minimizing in each large odd ball \(B_{r_2}(p_{2s+1})\). Since the almost minimizing
property is inherited by subsets, \(\widehat V_j\) is \(1/j\)-almost \(\F\)-minimizing in
each odd ball \(B_{2s+1}\). The odd balls are pairwise disjoint, and therefore we may
perform the simultaneous replacement in
\[
    B_O=\bigcup_{s=0}^{q-1}B_{2s+1}.
\]
More precisely, let \(W_j\) be the limit of a minimizing sequence for
$\mathrm{Problem}\bigl(\widehat V_j,\mathrm{Is}_j(\widehat V_j,B_O)\bigr)$
with respect to the anisotropic functional \(\F\).

By the anisotropic replacement theory from \cref{prop:aniso-minmax-replacements},
\(W_j\) is induced in \(B_O\) by a smooth embedded \(\F\)-stationary and \(\F\)-stable
surface, with multiplicity one for each fixed \(j\), still denoted by \(W_j\). Moreover
the replacement has the same varifold limit as \(\widehat V_j\), and hence
$W_j\to n\Gamma$
smoothly and graphically on compact subsets of \(B_O\). In particular, for all large \(j\),
$W_j\cap B_1(1-\delta)$
is a union of \(n\) normal graphs over \(\Gamma\cap B_1(1-\delta)\).

By the transversality assumption and the smooth convergence in \(B_1(1-\delta)\), the
intersection
$W_j\cap B_0(1-\delta)\cap\partial B_1(1-\delta)$
is transverse for all large \(j\), and consists of \(n\) smooth arcs
$\{\beta_j^i\}_{i=1}^{n}$.
As in Step~3, the anisotropic stable curvature estimates and the transversality of
\(\Gamma\) with \(\partial B_1(1-\delta)\) imply
\[
    \sup_{\beta_j^i}|k_N|\le C,
    \qquad i=1,\ldots,n.
\]

We summarize the information obtained so far. After choosing the smoothing collars in
Step~4 sufficiently small, the odd replacement does not modify the region
\(B_0(1-\delta)\setminus B_1\), and therefore \(W_j\) agrees there with the first
replacement \(V_j\). Hence \(W_j\) converges smoothly as \(n\) graphs to \(\Gamma\) on
$B_0(1-\delta)\setminus B_1$.
On the other hand, by the second replacement, \(W_j\) converges smoothly as \(n\) graphs
to \(\Gamma\) on
$B_1(1-\delta)$.
Moreover,
\[
    \sup_{\alpha_j^i}|k_N|\le C,
    \qquad
    \sup_{\beta_j^i}|k_N|\le C,
    \qquad i=1,\ldots,n.
\]
It remains to obtain control in the overlap wedge
\[
    \mathcal W
    :=
    B_0(1-\delta)\cap\bigl(B_1\setminus B_1(1-\delta)\bigr).
\]

\smallskip

\smallskip

\noindent\emph{Step 6: crossing the wedge.}
We prove the anisotropic analogue of Ketover's no-folding estimate. The only analytic
change is that the identity \(H=0\), used in the isotropic proof, is replaced by the
elliptic relation coming from \(\F\)-stationarity.

Let \(S\) be a smooth \(\F\)-stationary surface in a coordinate ball and choose a local
unit normal \(\nu\). The Euler--Lagrange equation for \(\F\) implies an elliptic
Weingarten-type relation
\begin{equation}
    \tr_S(\Psi_F A_S)+E_F=0,
    \qquad
    |E_F|\le C_\F,
    \label{eq:anisotropic-stationarity-relation}
\end{equation}
where \(\Psi_F\) is the matrix associated with the integrand in \eqref{eq:1st-variation}, restricted to
\(TS\). By the uniform convexity assumption, \(\Psi_F|_{TS}\) is uniformly positive definite:
\[
    \lambda I\le \Psi_F|_{TS}\le \lambda^{-1}I.
\]

We claim that
\begin{equation}
    |A_S|^2\le C(1-\det A_S).
    \label{eq:anisotropic-A-det-estimate}
\end{equation}
Indeed, at a fixed point diagonalize \(\Psi_F|_{TS}\) and write
\[
    \Psi_F|_{TS}=
    \begin{pmatrix}
        \alpha&0\\
        0&\beta
    \end{pmatrix},
    \qquad
    \lambda\le \alpha,\beta\le \lambda^{-1},
\]
and
\[
    A_S=
    \begin{pmatrix}
        a&b\\
        b&c
    \end{pmatrix}.
\]
Then \eqref{eq:anisotropic-stationarity-relation} gives $|\alpha a+\beta c|\le C_\F$.
Thus
\[
    c=-\frac{\alpha}{\beta}a+O(1),
\]
and hence
\[
    -\det A_S
    =
    b^2-ac
    =
    b^2+\frac{\alpha}{\beta}a^2+O(|a|).
\]
Absorbing the lower-order term gives
\[
    a^2+b^2+c^2
    \le
    C(1+C_\F^2-\det A_S).
\]
Since \(C_\F\) is fixed, this proves \eqref{eq:anisotropic-A-det-estimate}. By the
Riemannian Gauss equation,
\[
    K_S=\sec_N(TS)+\det A_S,
\]
we obtain
\begin{equation}
    |A_S|^2
    \le
    C\bigl(1+\sec_N(TS)-K_S\bigr).
    \label{eq:anisotropic-gauss-bound}
\end{equation}

Let \(B\) be a relatively compact domain in the wedge with piecewise smooth boundary,
chosen transversely to \(W_j\). Applying \eqref{eq:anisotropic-gauss-bound} to
\(S=W_j\), and then using Gauss--Bonnet on \(W_j\cap B\), gives
\begin{equation}
    \begin{aligned}
    \int_{W_j\cap B}|A_{W_j}|^2
    \le C\bigg(
        \mathcal H^2(W_j\cap B)
        + &\operatorname{genus}(W_j\cap B)
        + e(W_j\cap B)\\
        &\quad+
        \int_{\partial B\cap W_j}|k_N|
        +
        |T(W_j\cap\partial B)|
    \bigg).
\end{aligned}
    \label{eq:aniso-gb-wedge}
\end{equation}
Here \(e(W_j\cap B)\) denotes the number of boundary components and
\(T(W_j\cap\partial B)\) denotes the total exterior angle at the nonsmooth part of
\(\partial B\). The area term is uniformly bounded by \eqref{H:comparability} and the uniform
\(\F\)-energy bound. The genus term is also uniformly bounded; in fact it is zero in the
present balls, because the original surfaces are planar in the corresponding large balls,
and the first and second replacements are obtained through \(\gamma\)-reductions, which
do not create genus.

We now parametrize the wedge. Use exponential coordinates centered at \(p_1\), and
spherical coordinates \((\rho,\theta,\varphi)\) in \(T_{p_1}N\), chosen so that
\(\gamma\cap B_1\) is contained in a small tubular neighborhood of the \(\theta=0\) axis.
For small \(\theta_0,\varphi_0>0\), set
\[
    R_{\theta_0,\varphi_0}
    :=
    \exp_{p_1}
    \left\{
    (\rho,\theta,\varphi):
    (1-\delta)r_1\le \rho\le r_1,\,
    |\theta|\le \theta_0,\,
    |\varphi-\pi/2|\le \varphi_0
    \right\}.
\]
Choose \(\theta_0,\varphi_0\) so that
\[
    \gamma\cap\mathcal W\subset R_{\theta_0/2,\varphi_0/2},
    \qquad
    \Gamma\cap R_{\theta_0,\varphi_0}
    \subset R_{\theta_0,\varphi_0/2}.
\]

In Ketover's isotropic proof, the monotonicity formula is used to show that the surfaces
do not meet the top and bottom faces of the wedge. We avoid this step, since no
monotonicity formula is available for a general anisotropic integrand. Instead we cut off
in both angular variables.

Choose smooth cutoff functions \(\psi=\psi(\theta)\) and \(\chi=\chi(\varphi)\) satisfying
\[
    0\le \psi,\chi\le 1,
\]
\[
    \psi\equiv1 \quad\text{on } |\theta|\le \theta_0/2,
    \qquad
    \chi\equiv1 \quad\text{on } |\varphi-\pi/2|\le \varphi_0/2,
\]
\[
    \operatorname{spt}\psi\subset \{|\theta|<\theta_0\},
    \qquad
    \operatorname{spt}\chi\subset \{|\varphi-\pi/2|<\varphi_0\},
\]
and
\[
    \frac{|\nabla\psi|^2}{\psi}\le C,
    \qquad
    \frac{|\nabla\chi|^2}{\chi}\le C.
\]
Set $\zeta:=\psi\chi$.
Then \(\zeta\equiv1\) on \(R_{\theta_0/2,\varphi_0/2}\), \(\zeta\) is supported away from
the top, bottom, left and right angular faces of \(R_{\theta_0,\varphi_0}\), and
\begin{equation}
    \frac{|\nabla\zeta|^2}{\zeta}\le C.
    \label{eq:zeta-gradient-bound}
\end{equation}
Indeed,
\[
    |\nabla(\psi\chi)|^2
    \le
    2\chi^2|\nabla\psi|^2+2\psi^2|\nabla\chi|^2
    \le
    C\psi\chi
    =
    C\zeta.
\]

For a regular value \(t\in(0,1)\), define
$    B_t:=\{\zeta>t\}\cap R_{\theta_0,\varphi_0}$.
Then \(B_t\) has piecewise smooth boundary. Its fixed boundary part lies on the two
radial faces $\rho=r_1$, $\rho=(1-\delta)r_1$,
and its remaining boundary part lies on the angular level surface
$\{\zeta=t\}$.
Since \(\zeta\) is compactly supported in the angular variables, no top or bottom face of
\(R_{\theta_0,\varphi_0}\) appears as a boundary face of \(B_t\).

Applying \eqref{eq:aniso-gb-wedge} with \(B=B_t\), the boundary curvature terms on the
two radial faces are controlled by the bounds for \(\alpha_j^i\) and \(\beta_j^i\) from
Steps~3 and~5. More precisely, the total curvature and length of all radial boundary arcs
are uniformly bounded.

We now control the terms \(e(W_j\cap B_t)\) and \(T(W_j\cap\partial B_t)\). Boundary
components of \(W_j\cap B_t\) are of two types. The first type meets one of the radial
faces. The number of such components is bounded, after integrating in \(t\), by the coarea
formula along the finitely many radial curves \(\alpha_j^i,\beta_j^i\), whose lengths and
curvatures are uniformly bounded. The second type is contained entirely in the angular
level surface \(\{\zeta=t\}\). For such a closed component \(\sigma\), the ambient total
curvature satisfies
\[
    \int_\sigma |k_N|\,\mathrm{d}s \ge c_0>0,
\]
where \(c_0\) depends only on the coordinate wedge. This is the same elementary
total-curvature estimate used in Ketover's proof. Hence, for a.e. \(t\),
\begin{equation}
    e(W_j\cap B_t)
    \le
    C+ C\,N_j(t)
    +
    C\int_{W_j\cap\{\zeta=t\}} |k_N|\,\mathrm{d}s,
    \label{eq:e-bound-zeta}
\end{equation}
where \(N_j(t)\) denotes the number of intersection points of the radial boundary curves
\(\alpha_j^i,\beta_j^i\) with \(\{\zeta=t\}\). Moreover, the total corner term satisfies
the analogous estimate
\begin{equation}
    |T(W_j\cap\partial B_t)|
    \le
    C+ C\,N_j(t)
    +
    C\int_{W_j\cap\{\zeta=t\}} |k_N|\,\mathrm{d}s.
    \label{eq:T-bound-zeta}
\end{equation}
Here corners occur only where the radial boundary arcs meet the angular level surface.

Combining \eqref{eq:aniso-gb-wedge}, \eqref{eq:e-bound-zeta}, and
\eqref{eq:T-bound-zeta}, we obtain for a.e. \(t\)
\begin{equation}
    \int_{W_j\cap B_t}|A_{W_j}|^2
    \le
    C
    + C N_j(t)
    + C\int_{W_j\cap\{\zeta=t\}} |k_N|\,\mathrm{d}s.
    \label{eq:levelset-ineq-zeta}
\end{equation}

We now integrate in \(t\). By the layer-cake formula,
\[
    \int_{W_j}\zeta |A_{W_j}|^2
    =
    \int_0^1
    \int_{W_j\cap\{\zeta>t\}}|A_{W_j}|^2\,\mathrm{d}t.
\]
Using \eqref{eq:levelset-ineq-zeta}, we get
\begin{equation}
    \int_{W_j}\zeta |A_{W_j}|^2
    \le
    C
    +
    C\int_0^1N_j(t)\,\mathrm{d}t
    +
    C\int_0^1
    \int_{W_j\cap\{\zeta=t\}} |k_N|\,\mathrm{d}s\mathrm{d}t.
    \label{eq:layercake-zeta}
\end{equation}

The radial intersection term is uniformly bounded:
\[
    \int_0^1N_j(t)\,\mathrm{d}t\le C.
\]
Indeed, by the coarea formula on the finitely many curves \(\alpha_j^i,\beta_j^i\),
\[
    \int_0^1N_j(t)\,\mathrm{d}t
    \le
    \sum_i\int_{\alpha_j^i}|\nabla_{\alpha_j^i}\zeta|\,\mathrm{d}s
    +
    \sum_i\int_{\beta_j^i}|\nabla_{\beta_j^i}\zeta|\,\mathrm{d}s
    \le C,
\]
because these curves have uniformly bounded length and \(\zeta\) has uniformly bounded
gradient.

It remains to estimate the angular level curvature term. For regular levels, the curvature
of the curve \(W_j\cap\{\zeta=t\}\) satisfies, in the coarea sense,
\begin{equation}
    \int_0^1
    \int_{W_j\cap\{\zeta=t\}} |k_N|\,\mathrm{d}s\,\mathrm{d}t
    \le
    C\int_{W_j}(1+|A_{W_j}|)|\nabla_{W_j}\zeta|\,\mathrm{d}\mathcal H^2.
    \label{eq:coarea-curvature-zeta}
\end{equation}
To see this, write the level surface \(\{\zeta=t\}\) as \(L_t\). The curvature of the
transverse intersection \(W_j\cap L_t\) is controlled by
\[
    |k_N|
    \le
    \frac{|A_{W_j}|+|A_{L_t}|}{\sin\vartheta},
\]
where \(\vartheta\) is the angle between \(W_j\) and \(L_t\). Multiplying by the coarea
factor \(|\nabla_{W_j}\zeta|\) cancels \(\sin\vartheta\), and
$|A_{L_t}|\,|\nabla \zeta|$
is uniformly bounded by the \(C^2\)-norm of \(\zeta\) in the coordinate wedge. This gives
\eqref{eq:coarea-curvature-zeta}.

Using \eqref{eq:zeta-gradient-bound}, we have
\[
    |\nabla_{W_j}\zeta|^2
    \le
    |\nabla \zeta|^2
    \le
    C\zeta.
\]
Therefore, by Cauchy's inequality and the uniform area bound,
\[
\begin{aligned}
    \int_{W_j}(1+|A_{W_j}|)|\nabla_{W_j}\zeta|
    &\le
    C
    +
    C\int_{W_j}|A_{W_j}|\,\zeta^{1/2}  \le
    C
    +
    \frac12\int_{W_j}\zeta |A_{W_j}|^2 .
\end{aligned}
\]
Combining this with \eqref{eq:layercake-zeta} and
\eqref{eq:coarea-curvature-zeta}, and absorbing the last term into the left-hand side,
we conclude that
\begin{equation}
    \sup_j
    \int_{W_j\cap R_{\theta_0/2,\varphi_0/2}}
    |A_{W_j}|^2
    <\infty.
    \label{eq:aniso-no-folding}
\end{equation}
This is the desired anisotropic no-folding estimate.

\smallskip

\smallskip

\noindent\emph{Step 7: graphical convergence through the wedge.}
By the \(\F\)-energy bound and \eqref{H:comparability}, the areas of \(W_j\) in
\(R_{\theta_0/2,\varphi_0/2}\) are uniformly bounded. Together with the no-folding estimate
\eqref{eq:aniso-no-folding}, this gives a uniform area and total curvature bound in
\(R_{\theta_0/2,\varphi_0/2}\). Hence the compactness theorem for smooth
\(\F\)-stationary surfaces with bounded area and total curvature implies that, after
passing to a subsequence, \(W_j\) converges smoothly and graphically to \(n\Gamma\) in
\(R_{\theta_0/2,\varphi_0/2}\), away from a finite set of points
$\mathcal P\subset R_{\theta_0/2,\varphi_0/2}.$

The parameters in the wedge construction were chosen with room to spare. Therefore we
may perturb \(\gamma\cap\mathcal W\), inside
$\Gamma\cap R_{\theta_0/2,\varphi_0/2},$
preserving its endpoints on the two controlled sides of the wedge and preserving its
homotopy class, so that
$\gamma\cap \mathcal P=\emptyset.$
We still denote the perturbed curve by \(\gamma\). Since \(\gamma\) is compact and avoids
\(\mathcal P\), there is a tubular neighborhood \(U_\gamma\subset
R_{\theta_0/2,\varphi_0/2}\) of \(\gamma\) disjoint from \(\mathcal P\). On \(U_\gamma\),
the convergence is smooth. Because the varifold limit is \(n\Gamma\), the convergence has
multiplicity \(n\). Therefore, for all sufficiently large \(j\),
$W_j\cap U_\gamma$
is a union of \(n\) normal graphs over \(\Gamma\cap U_\gamma\).

\smallskip

\smallskip

\noindent\emph{Step 8: lifting with full multiplicity.}
Steps~1--7 imply the following. After perturbing \(\gamma\) inside
\(\Gamma\cap T_\epsilon(\gamma)\), preserving its homotopy class, there exists a tubular
neighborhood \(U_\gamma\subset T_\epsilon(\Gamma)\) of the perturbed curve such that the
second replacement \(W_j\) has full multiplicity \(n\) over \(\Gamma\) in \(U_\gamma\). More
precisely, over every sufficiently small subarc of \(\gamma\), the surface \(W_j\) is a
union of \(n\) normal graphs converging smoothly to \(n\Gamma\). If the normal bundle of
\(\Gamma\) along \(\gamma\) is trivial, these local sheets close globally as \(n\) sheets over
\(\gamma\); if the normal bundle over \(\gamma\) is nontrivial, the sheets are permuted by
the monodromy of the M\"obius band.

We now pass from the replacement \(W_j\) to a surface obtained from the original
\(\Sigma_j\) by finitely many \(\gamma\)-reductions. For each fixed \(j\), the surface
\(W_j\) is the varifold limit of the local isotopy-minimizing sequence used in the second
replacement. By the anisotropic analogue of Ketover's Proposition~4.7, as used in
Step~2, this minimizing sequence may be chosen so that its sufficiently far final slices
are obtained from \(\widehat V_j\) by finitely many \(\gamma\)-reductions and isotopies.
Moreover, \(\widehat V_j\) was obtained from the first replacement \(V_j\), and \(V_j\)
itself is obtained from \(\Sigma_j\) by finitely many \(\gamma\)-reductions and isotopies.
Thus, after choosing a sufficiently far term in the whole approximating sequence and
diagonalizing in \(j\), we obtain a smooth surface \(\tilde\Sigma_j\) such that:

\begin{enumerate}
    \item \(\tilde\Sigma_j\) is obtained from \(\Sigma_j\) by finitely many
    \(\gamma\)-reductions and isotopies;
    \item \(\tilde\Sigma_j\) is varifold-close to \(W_j\);
    \item after possibly shrinking \(U_\gamma\), the surface
    \(\tilde\Sigma_j\cap U_\gamma\) has the same \(n\)-sheet graphical structure as
    \(W_j\cap U_\gamma\).
\end{enumerate}

We denote the final perturbed curve by \(\tilde\gamma\). Since \(W_j\to n\Gamma\) as
varifolds and \(\tilde\Sigma_j\) is chosen varifold-close to \(W_j\), we have
$ \tilde\Sigma_j\to n\Gamma$
as varifolds in the chosen tubular neighborhood.

Let
\[
    \pi:T_\epsilon(\Gamma)\to\Gamma
\]
be the nearest point projection. If \(\pi^{-1}(\tilde\gamma)\) is an annulus, then the
normal bundle of \(\Gamma\) over \(\tilde\gamma\) is trivial. Hence the \(n\) local sheets
close up separately, and
$\pi^{-1}(\tilde\gamma)\cap \tilde\Sigma_j$
is a union of \(n\) closed curves. Each of these curves projects to \(\tilde\gamma\) with
degree one.

If \(\Gamma\) is non-orientable and \(\pi^{-1}(\tilde\gamma)\) is a M\"obius band, then
the monodromy of the normal interval bundle exchanges the two sides after one circuit
around \(\tilde\gamma\). Consequently the local sheets close only after passing twice
around \(\tilde\gamma\). Hence \(n\) is even and
$\pi^{-1}(\tilde\gamma)\cap \tilde\Sigma_j$
is a union of \(n/2\) closed curves, each projecting to \(\tilde\gamma\) with degree two.

Indeed, if \(n\) were odd in the M\"obius case, then one of the lifted curves would be
isotopic to the core of the M\"obius band \(\pi^{-1}(\tilde\gamma)\). The normal bundle of
the ambient tubular neighborhood over such a curve is nontrivial. Therefore the component
of \(\tilde\Sigma_j\) containing this curve would be non-orientable. This is impossible,
because \(\tilde\Sigma_j\) is obtained from the orientable sweepout surface \(\Sigma_j\)
by neck-pinch surgeries and isotopies, which preserve orientability of the remaining
components. This proves the lifting conclusion for the single curve \(\gamma\) on the
component \(\Gamma\).

\smallskip

\smallskip

\noindent\emph{Step 9: several curves and several components.}
It remains to explain how to perform the construction for a finite collection of curves.
We first assume that all curves lie on the same connected component
\(\Gamma\), with multiplicity \(n\), and that
$\gamma_1,\ldots,\gamma_q\subset \Gamma$
intersect only at one prescribed point \(p\in\Gamma\), and do so transversely there. If
\(p\) or one of the curves meets the finite bad set \(\mathcal B\) introduced in the setup,
we first perturb the whole curve system inside an arbitrarily small neighborhood in
\(\Gamma\), preserving the homotopy classes and the prescribed intersection pattern, so
that the common point and all curves avoid \(\mathcal B\).

We now choose the ball systems as in \cite[Section~4.9]{K}. Start with \(\gamma_1\).
Choose a scale \(\rho_1>0\) sufficiently small and choose the small and large balls along
\(\gamma_1\), of radii ${3\rho_1}/{4}$ and ${15\rho_1}/{8}$,
as in the one-curve construction. We arrange the labeling so that the small ball centered
at the common intersection point \(p\) is one of the odd balls, denoted
$B_1=B_{3\rho_1/4}(p).$
The large ball centered at \(p\) is then
$B_{15\rho_1/8}(p).$
By taking \(\rho_1\) sufficiently small, we may assume that for every
\(i=2,\ldots,q\), the curve \(\gamma_i\) intersects
$\partial B_{3\rho_1/4}(p)$
in two distinct points \(a_i,b_i\), and intersects
$\partial B_{15\rho_1/8}(p)$
in two distinct points \(c_i,d_i\), with all these points distinct as \(i\) varies. We also
arrange that, for \(i\ge2\), the curve \(\gamma_i\) is disjoint from all small balls
associated to \(\gamma_1\) except \(B_1\), and is disjoint from all large balls associated
to \(\gamma_1\) except \(B_{15\rho_1/8}(p)\).

For each \(i=2,\ldots,q\), choose a scale \(\rho_i>0\) much smaller than \(\rho_1\), and
construct the corresponding small and large balls along \(\gamma_i\), with radii
${3\rho_i}/{4}$ and ${15\rho_i}/{8}$.
The scales \(\rho_i\) are chosen inductively so small that the large balls associated to
\(\gamma_i\) are disjoint from the large balls associated to \(\gamma_\ell\) whenever
\(\ell\ne i\) and both indices are at least \(2\), except for the controlled common
configuration inside \(B_{15\rho_1/8}(p)\). Moreover, for each \(i\ge2\), we choose the
labeling so that:
\begin{enumerate}
    \item the two small balls associated to \(\gamma_i\) which contain \(a_i\) and \(b_i\)
    are even balls;
    \item the two points \(c_i,d_i\) are contained in odd small balls associated to
    \(\gamma_i\).
\end{enumerate}
Finally, for each \(i\ge2\), we discard all small balls associated to \(\gamma_i\) which
are contained entirely in \(B_1\). Then the remaining small balls associated to
\(\gamma_i\), together with \(B_1\), still cover \(\gamma_i\).

We now perform the first replacement simultaneously in all even small balls associated to
all curves \(\gamma_i\). The above choice of scales guarantees that these even balls are
pairwise disjoint, except for the controlled situation in which the even balls meeting the
large ball \(B_{15\rho_1/8}(p)\) are all contained in that large ball. Applying the
anisotropic smoothing lemma from Step~4 in this large ball, the resulting smoothed first
replacement remains almost \(\F\)-minimizing in \(B_{15\rho_1/8}(p)\), and hence in the
odd small balls contained in it. For the remaining odd small balls, the corresponding
large balls are disjoint from all other large balls by the inductive choice of the scales,
so the almost minimizing property follows directly from the one-curve setup.

Therefore the second replacement can be performed simultaneously in all odd small balls
associated to all curves. The local no-folding argument from Steps~6--7 applies in every
overlap wedge. Thus each curve \(\gamma_i\) can be lifted with the full multiplicity \(n\),
after an arbitrarily small perturbation preserving the prescribed intersection pattern.

This proves the proposition for finitely many curves lying on one limiting component
\(\Gamma\).

We now return to the general case
\[
    \Sigma_j\to \sum_{k=1}^m n_k\Sigma^{(k)}.
\]
Since the limiting components are pairwise disjoint, choose pairwise disjoint tubular
neighborhoods
$T_\epsilon(\Sigma^{(k)}).$
The preceding construction can be carried out independently inside each such
neighborhood, using the multiplicity \(n_k\) of the component \(\Sigma^{(k)}\). The
supports of the corresponding replacements, smoothings and \(\gamma\)-reductions are
disjoint for different \(k\), so the operations commute. Taking a diagonal subsequence in
\(j\), we obtain curves \(\tilde\gamma_i\) and surfaces \(\tilde\Sigma_j\), obtained from
the original \(\Sigma_j\) by finitely many \(\gamma\)-reductions and isotopies, such that
all three conclusions of the proposition hold for every curve on every component. This
completes the proof of \cref{prop:optimal-curve-lifting}.
\end{proof}

\bibliographystyle{acm}
\bibliography{references.bib}

\end{document}